%% file: main.tex
\documentclass[lettersize,journal]{IEEEtran}
\input{my_macros.tex}

\begin{document}

\title{A DC Composite Optimization \\via Variable Smoothing for Robust Phase Retrieval\\ with Nonconvex Loss Functions}
\author{Kumataro Yazawa, Keita Kume \IEEEmembership{Member, IEEE}, Isao Yamada \IEEEmembership{Fellow, IEEE}
\thanks{
  K.~Yazawa, K.~Kume and I.~Yamada are with
  the Department of Information and Communications Engineering,
  Institute of Science Tokyo,
  2-12-1-S3-60, O-okayama,
  Meguro-ku, Tokyo 152-8550, Japan
  (e-mail: \{yazawa, kume, isao\}@sp.ict.e.titech.ac.jp).
  This work was partially supported by JSPS Grants-in-Aid (19H04134, 24K23885).
  A preliminary short version of this paper was presented in \cite{yazawa2025variable} as a conference paper.
  Compared to \cite{yazawa2025variable}, this paper includes complete proofs of the mathematical results and more illustrative experimental results.
  This work has been submitted to the IEEE for possible publication. Copyright may be transferred without notice, after which this version may no longer be accessible.}
}

\date{}

\maketitle

\begin{abstract}
    In this paper, we propose an optimization-based method for robust phase retrieval problem where the goal is to estimate an unknown signal from a quadratic measurement corrupted by outliers.
    To enhance the robustness of existing optimization models with the $\ell_1$ loss function, we propose a generalized model that can handle DC (Difference-of-Convex) loss functions beyond the $\ell_1$ loss.
    We view the cost function of the proposed model as a composition of a DC function with a smooth mapping, and develop a variable smoothing algorithm for minimizing such DC composite functions.
    At each step of our algorithm, we generate a smooth surrogate function by using the Moreau envelope of each (weakly) convex function in the DC function, and then perform the gradient descent update of the surrogate function.
    Unlike many existing algorithms for DC problems, the proposed algorithm does not require any inner loop.
    We also present a convergence analysis in terms of a DC composite critical point for the proposed algorithm.
    Our numerical experiment demonstrates that the proposed method with DC loss functions is more robust against outliers compared to existing methods with the $\ell_1$ loss.
\end{abstract}


\input{footnote.tex}

\input{introduction.tex}
\input{preliminary.tex}
\input{algorithm.tex}
\input{experiment.tex}

\input{conclusion.tex}
\bibliographystyle{IEEEtran}
\bibliography{IEEEabrv,main}
\input{appendix.tex}

\end{document}


\title{Supplementary Material for \\ ``A DC Composite Optimization\\ via Variable Smoothing for Robust Phase Retrieval with Nonconvex Loss Functions''}
\author{Kumataro Yazawa, Keita Kume \IEEEmembership{Member, IEEE}, Isao Yamada \IEEEmembership{Fellow, IEEE}}

\date{}

\maketitle

As described in \cref{sec:experiment}, we conducted the experiments by using not only Cauchy outlier but also uniformly distributed outliers.
In this material, we present the results for the uniformly distributed outliers for completeness.

\cref{fig:results_d_100_uniform,fig:results_d_500_uniform,fig:different scales of uniform outliers,tab:execution time for uniform outlier}
below correspond to \cref{fig:results_d_100_Cauchy,fig:results_d_500_Cauchy,fig:different scales of outliers,tab:execution time} in \cref{sec:experiment}, respectively.
Despite the change in the distribution of $\xi_i$, each figure exhibits similar trends to those in the corresponding figure.
Specifically, (i) \cref{fig:results_d_100_uniform,fig:results_d_500_uniform} show the robustness of the proposed method with the capped $\ell_1$ and the trimmed $\ell_1$,
(ii) \cref{fig:different scales of uniform outliers} illustrates that the appropriate values of $\beta$ and $K$ depend on the scale and number of outliers, respectively,
and (iii) \cref{tab:execution time for uniform outlier} demonstrates that the proposed method achieves lower computational time than the existing method.

\begin{figure*}[h]
\centering
\begin{subfigure}{0.23\textwidth}
\includegraphics[width=\linewidth]{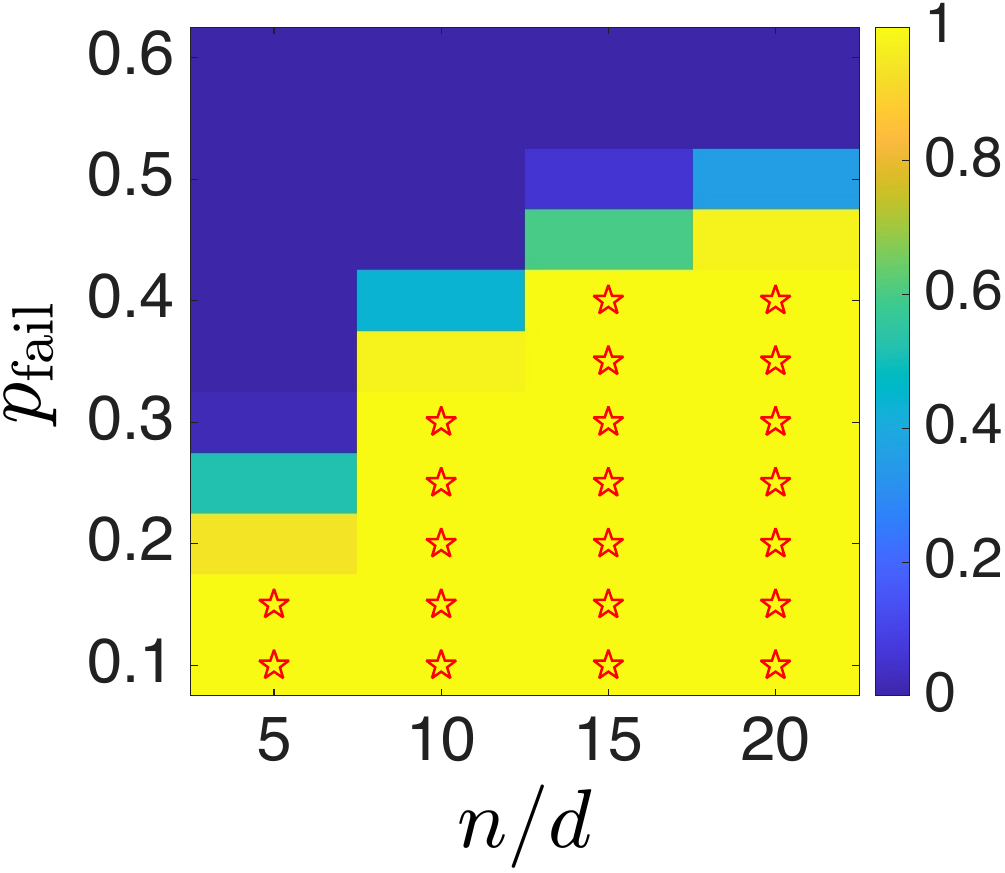}
\caption{$\ell_1$ by IPL}
\end{subfigure}
\begin{subfigure}{0.23\textwidth}
\includegraphics[width=\linewidth]{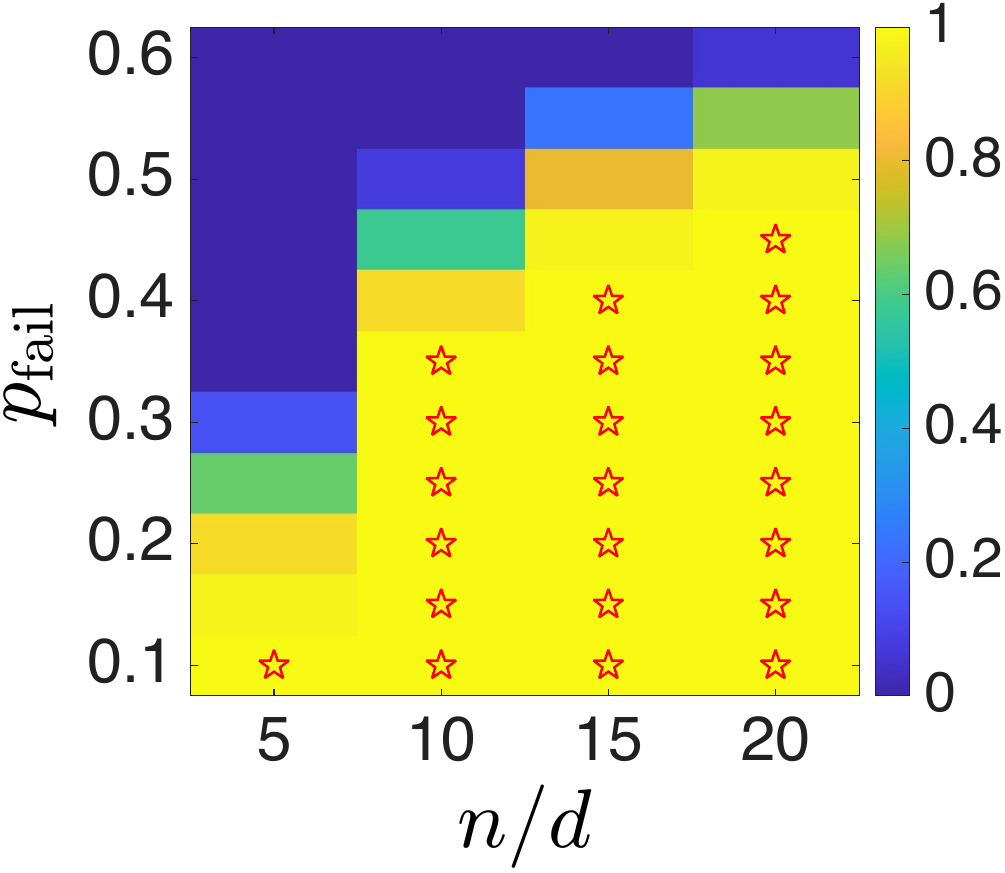}
\caption{Capped $\ell_1$ ($\beta=100$)}
\end{subfigure}
\begin{subfigure}{0.23\textwidth}
\includegraphics[width=\linewidth]{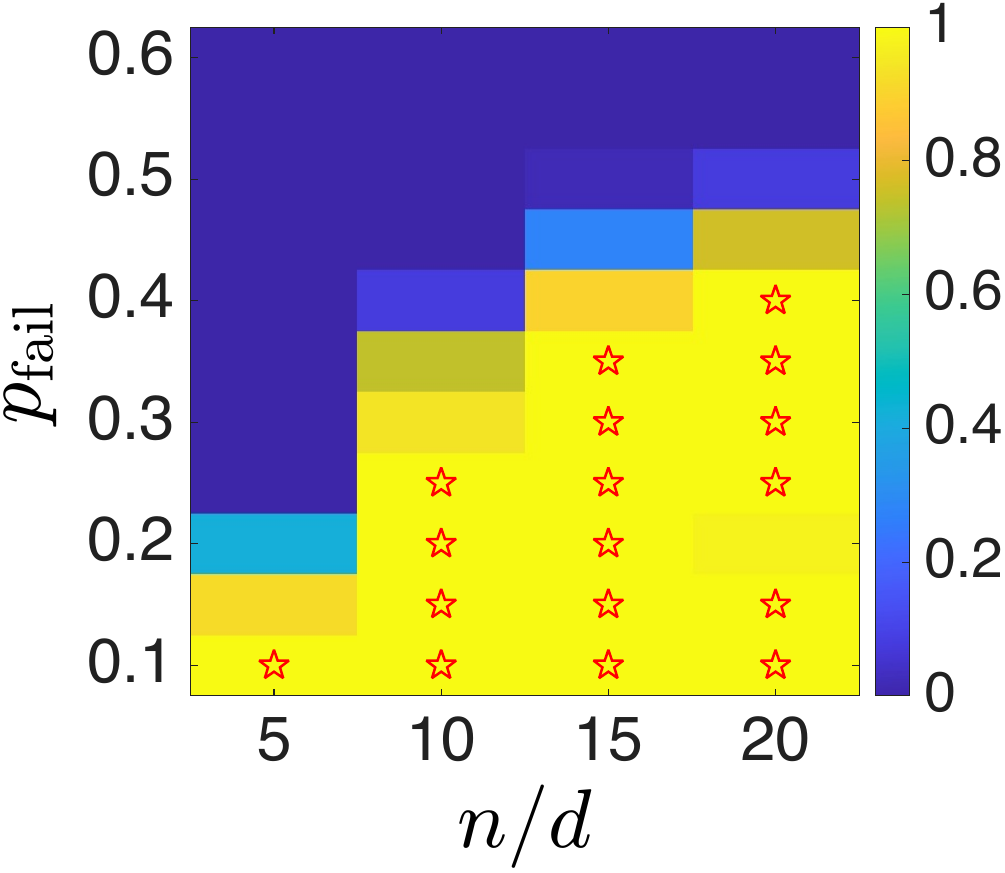}
\caption{Capped $\ell_1$ ($\beta=1000$)}
\end{subfigure}
\begin{subfigure}{0.23\textwidth}
\includegraphics[width=\linewidth]{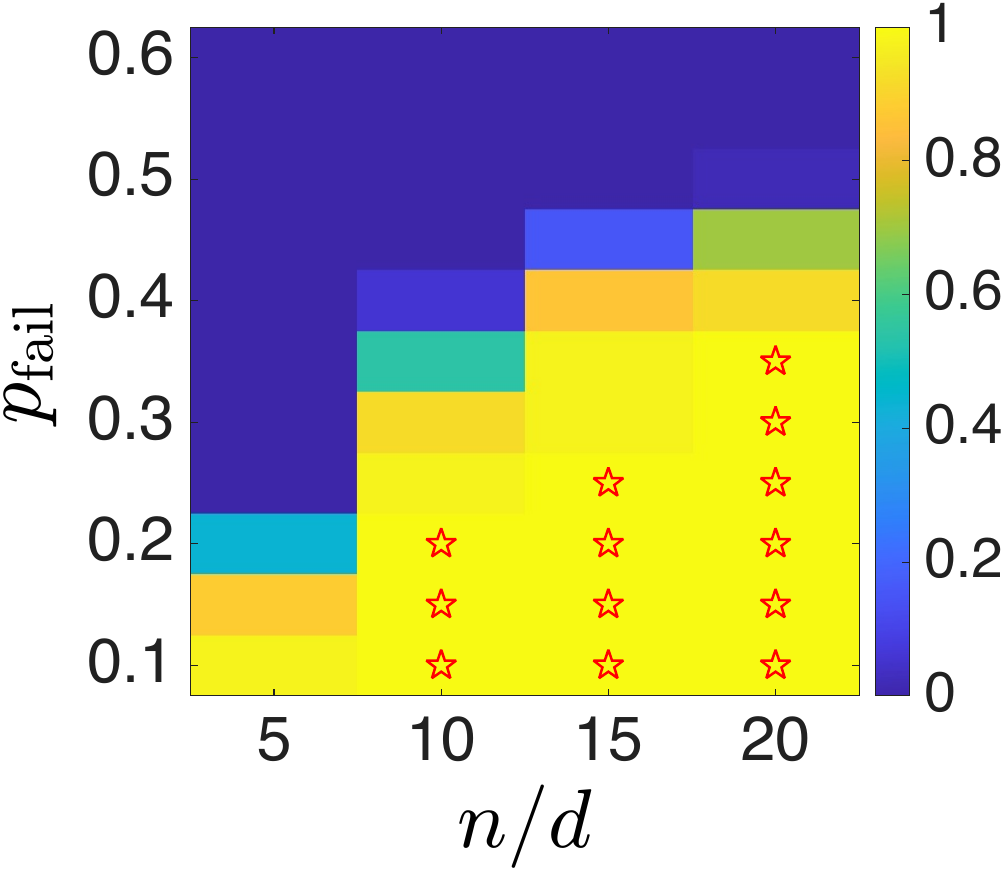}
\caption{Capped $\ell_1$ ($\beta=10000$)}
\end{subfigure}

\begin{subfigure}{0.23\textwidth}
\includegraphics[width=\linewidth]{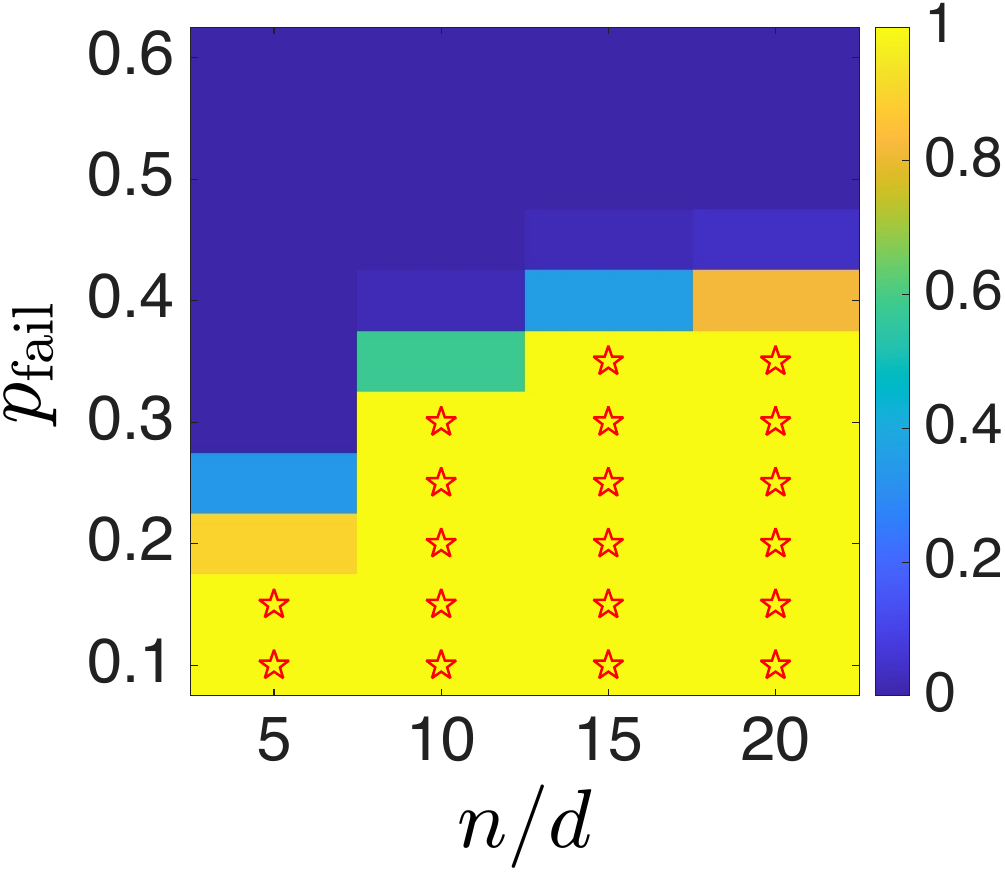}
\caption{$\ell_1$ by Robust-AM}
\end{subfigure} 
\begin{subfigure}{0.23\textwidth}
\includegraphics[width=\linewidth]{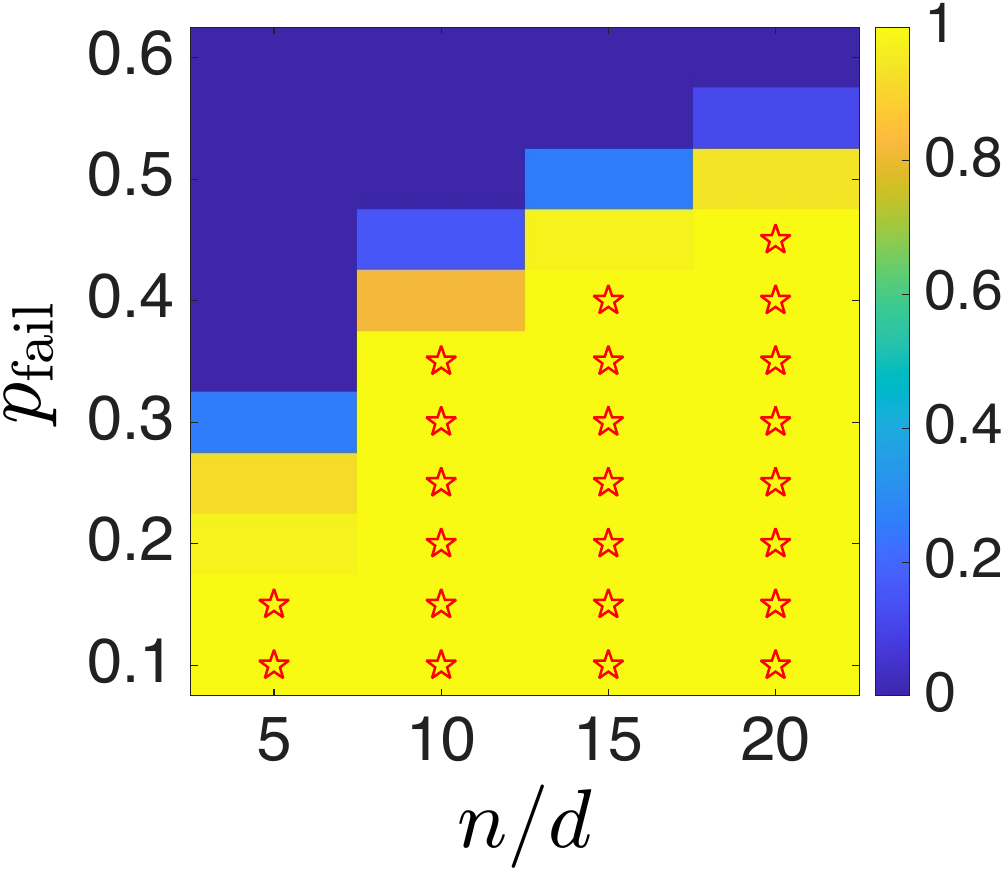}
\caption{Trimmed $\ell_1$ ($K/n=0.2$)}
\end{subfigure}
\begin{subfigure}{0.23\textwidth}
\includegraphics[width=\linewidth]{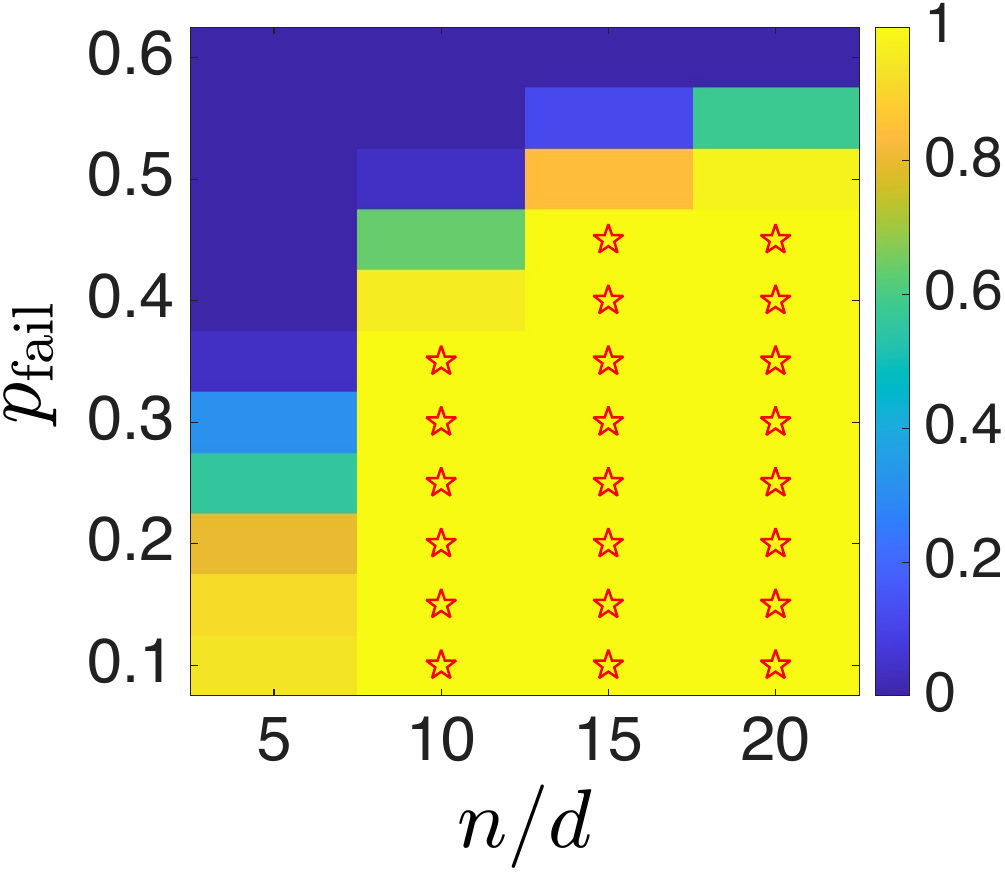}
\caption{Trimmed $\ell_1$ ($K/n=0.3$)}
\end{subfigure}
\begin{subfigure}{0.23\textwidth}
\includegraphics[width=\linewidth]{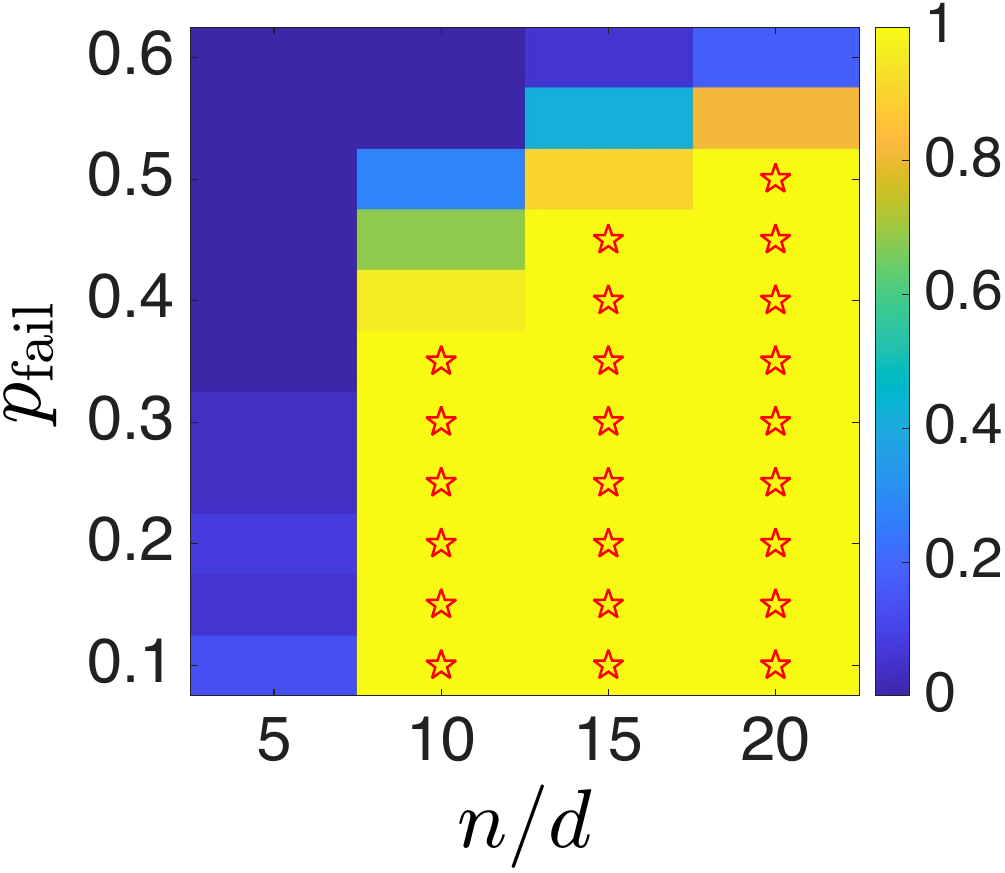}
\caption{Trimmed $\ell_1$ ($K/n=0.4$)}
\end{subfigure}

\caption{Success rate of each method.
We used uniformly distributed $\xi_i$ and the same parameters $(d,s)=(100,1)$ as in \cref{fig:results_d_100_Cauchy}.
}
\label{fig:results_d_100_uniform}
\end{figure*}

\begin{figure*}[h]
\centering
\begin{subfigure}{0.23\textwidth}
\includegraphics[width=\linewidth]{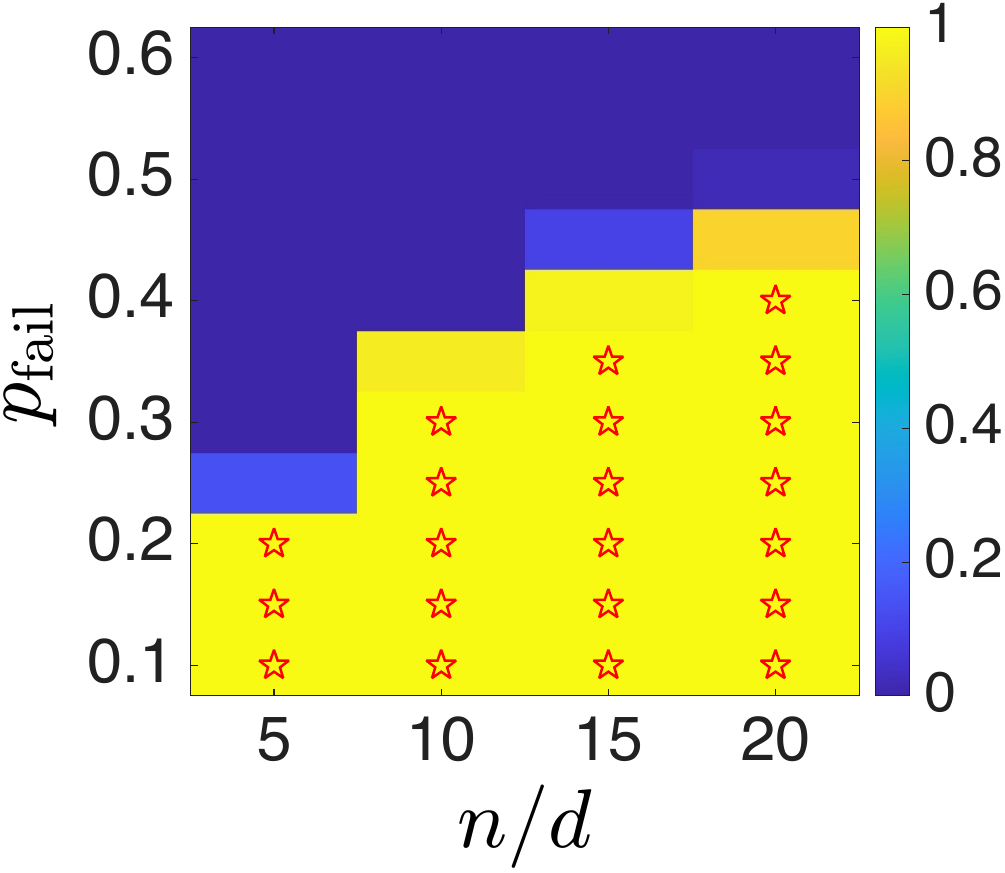}
\caption{$\ell_1$ by IPL}
\end{subfigure}
\begin{subfigure}{0.23\textwidth}
\includegraphics[width=\linewidth]{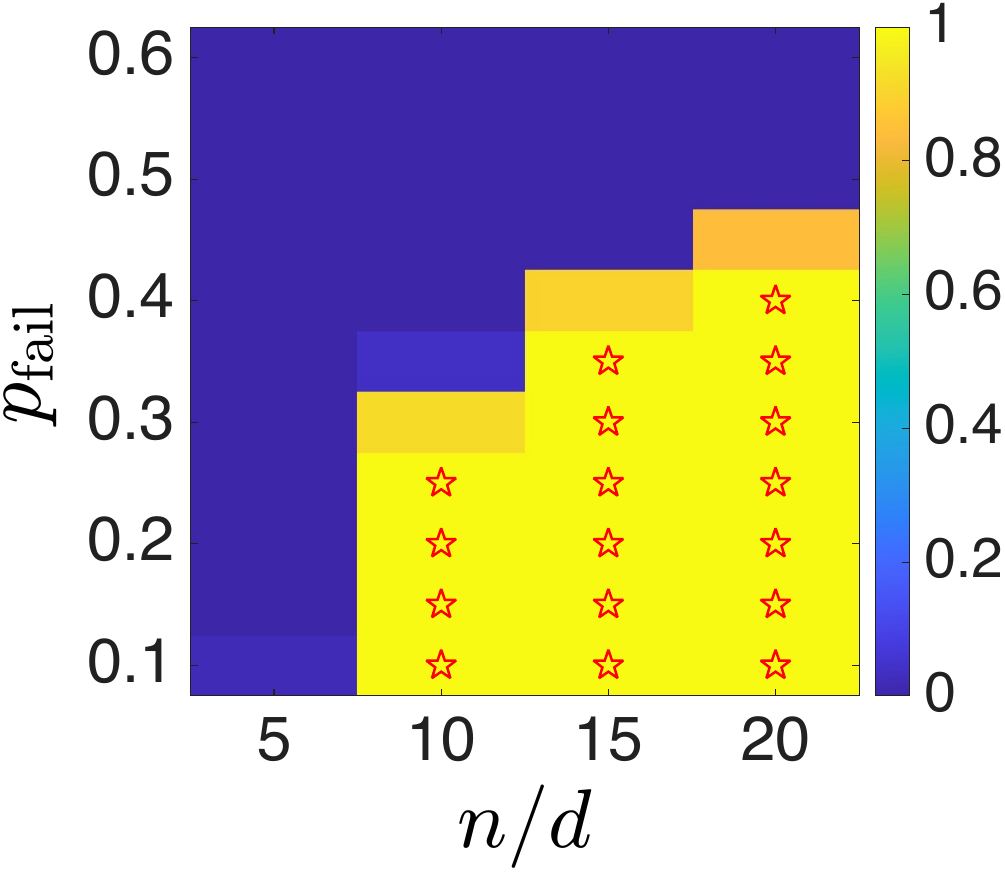}
\caption{Capped $\ell_1$ ($\beta=100$)}
\end{subfigure}
\begin{subfigure}{0.23\textwidth}
\includegraphics[width=\linewidth]{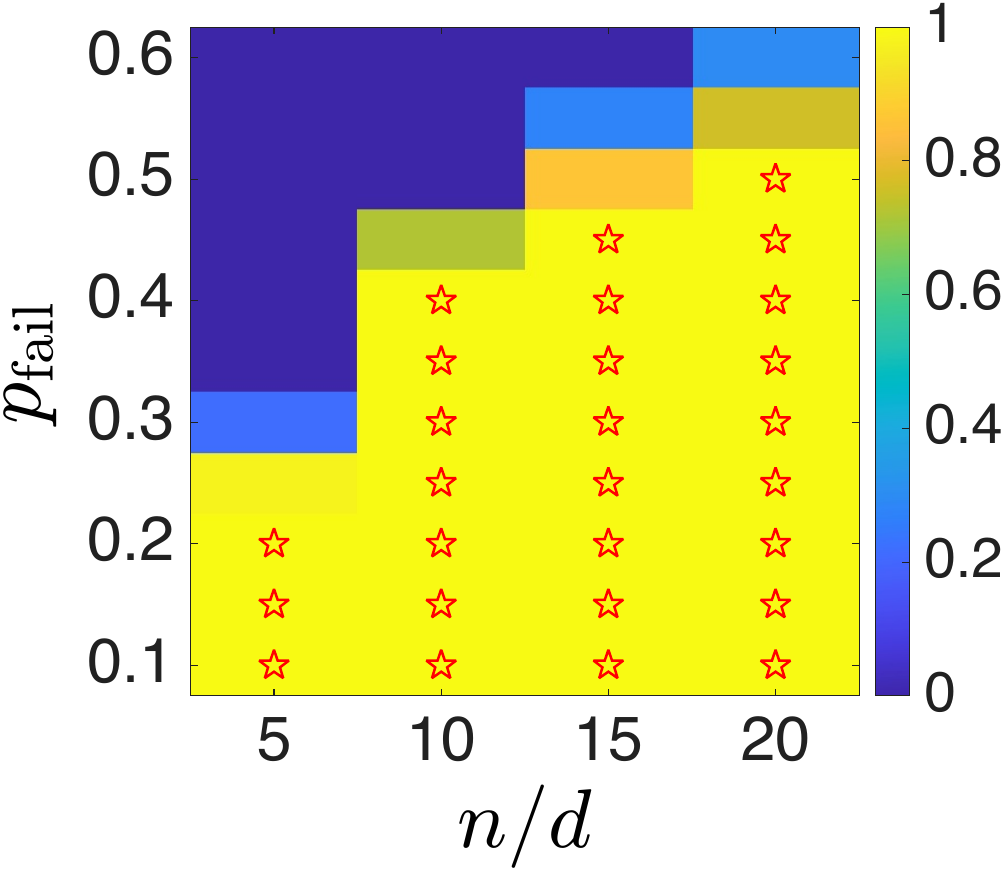}
\caption{Capped $\ell_1$ ($\beta=1000$)}
\end{subfigure}
\begin{subfigure}{0.23\textwidth}
\includegraphics[width=\linewidth]{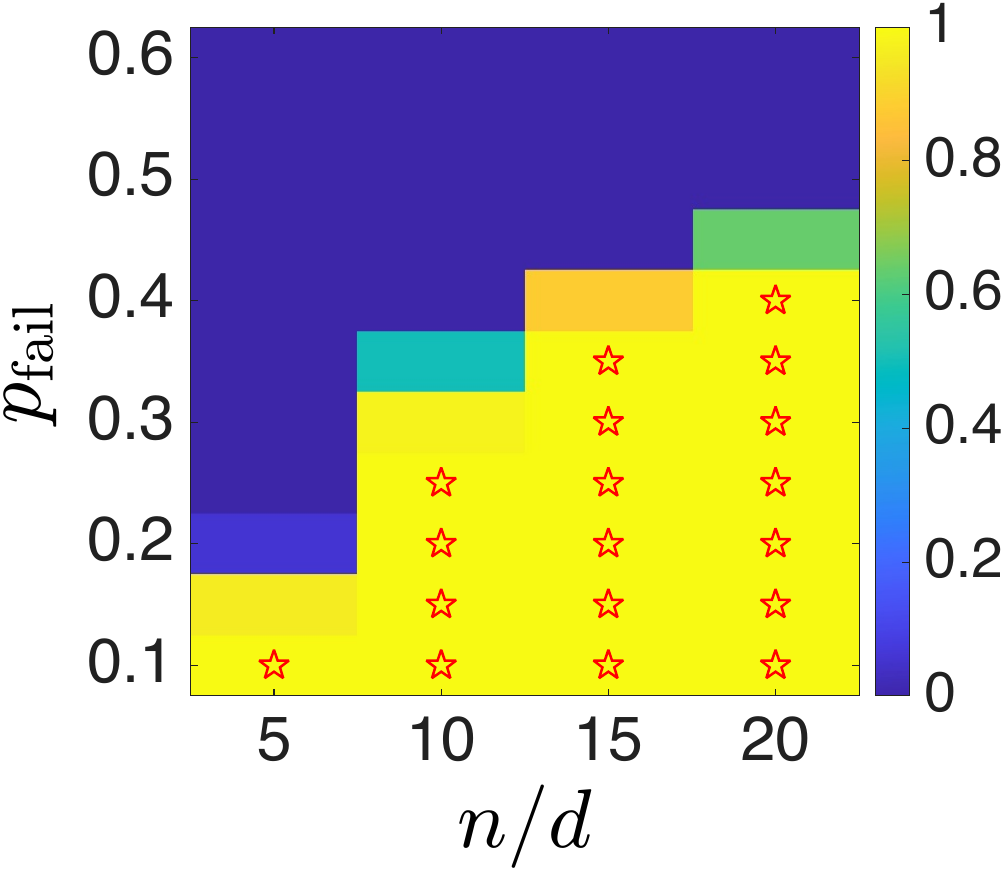}
\caption{Capped $\ell_1$ ($\beta=10000$)}
\end{subfigure}

\hspace{0.23\textwidth}
\begin{subfigure}{0.23\textwidth}
\includegraphics[width=\linewidth]{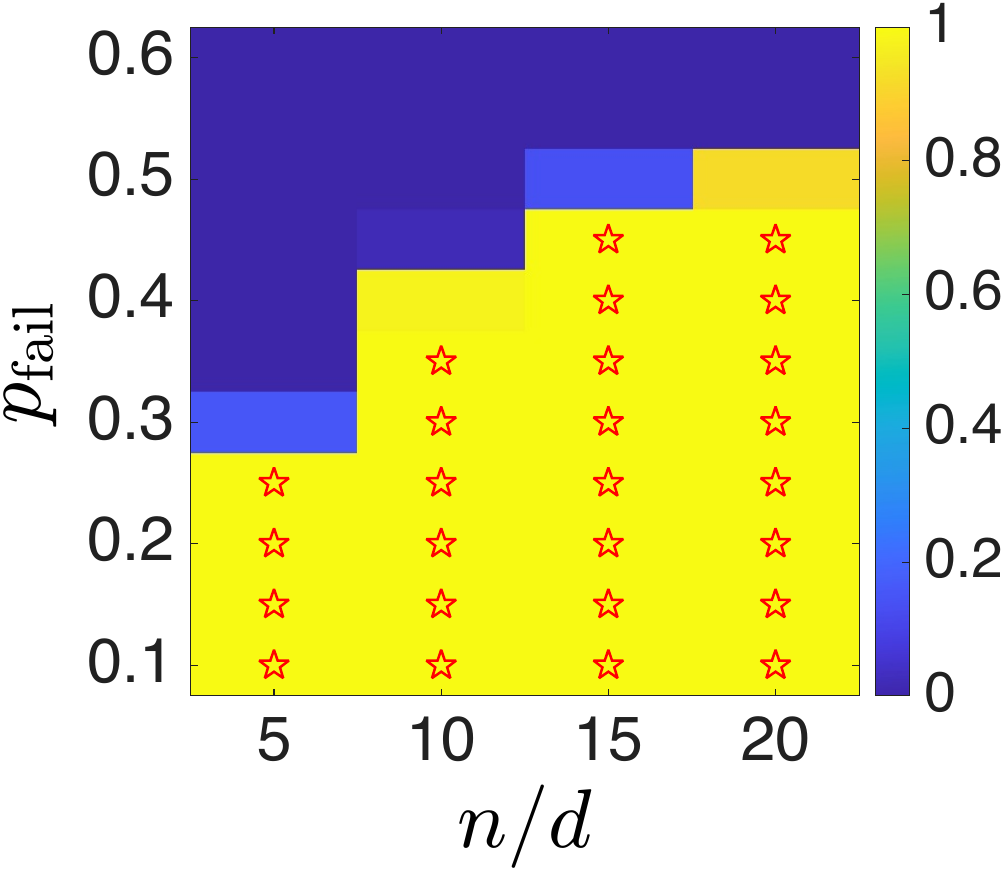}
\caption{Trimmed $\ell_1$ ($K/n=0.2$)}
\end{subfigure}
\begin{subfigure}{0.23\textwidth}
\includegraphics[width=\linewidth]{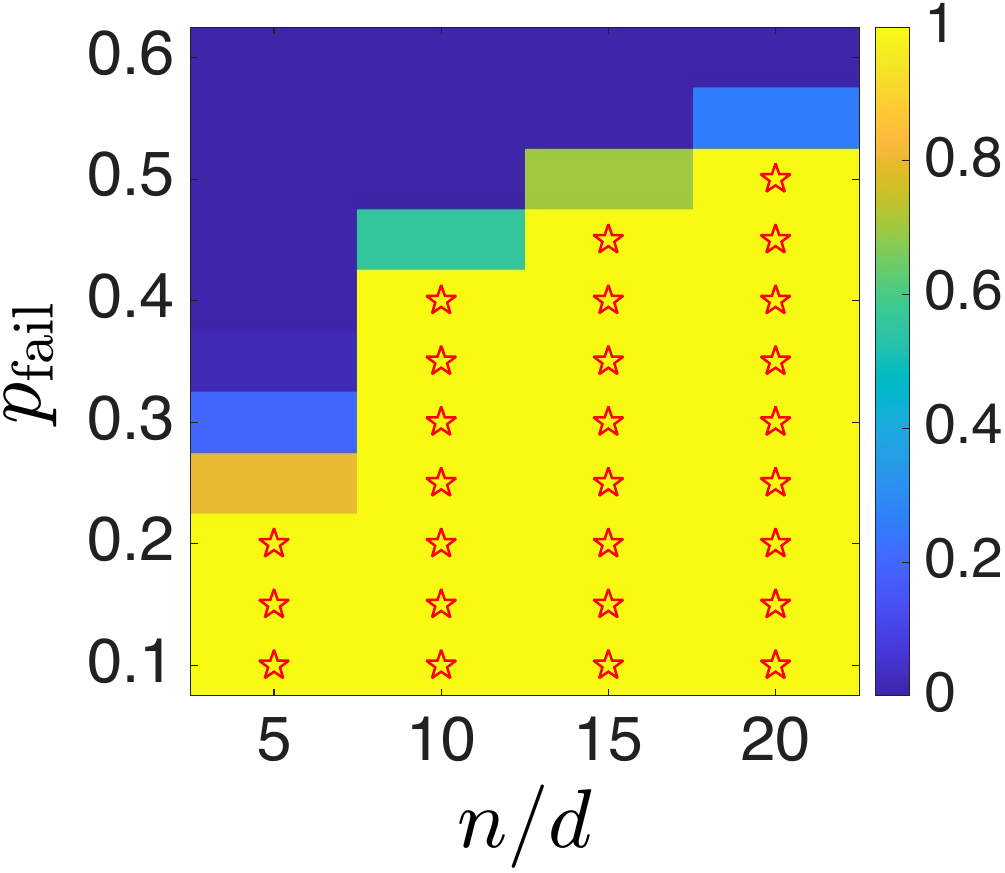}
\caption{Trimmed $\ell_1$ ($K/n=0.3$)}
\end{subfigure}
\begin{subfigure}{0.23\textwidth}
\includegraphics[width=\linewidth]{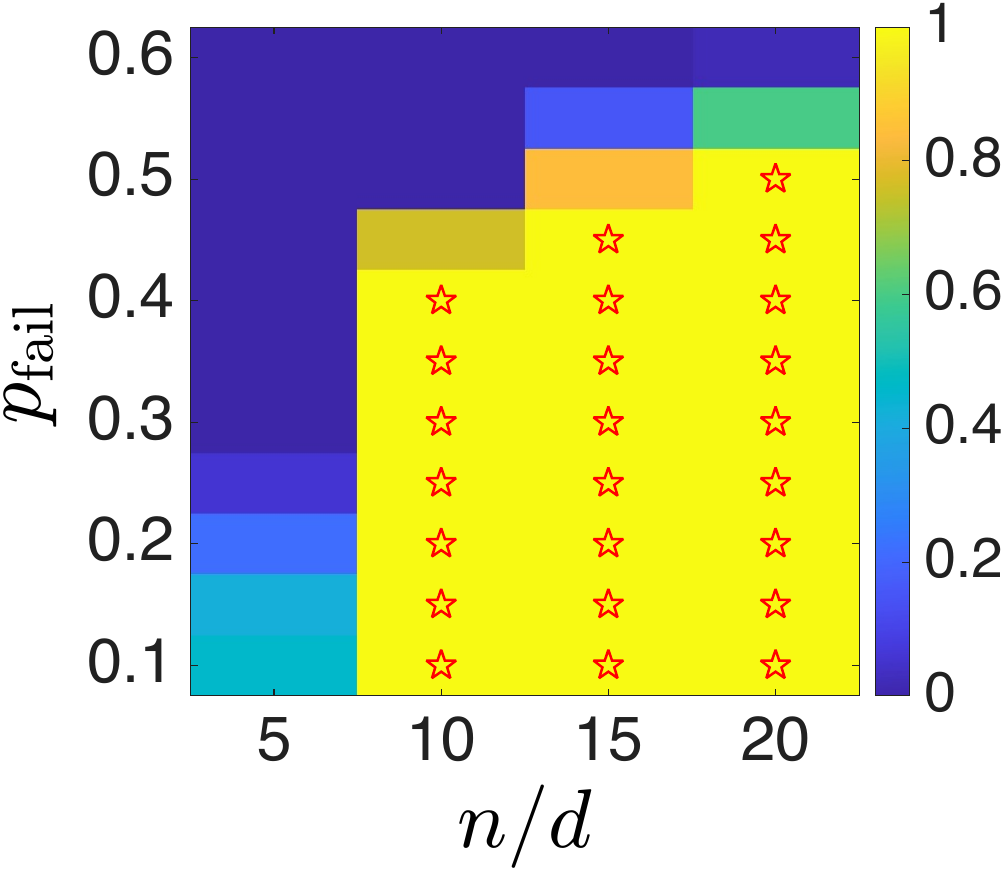}
\caption{Trimmed $\ell_1$ ($K/n=0.4$)}
\end{subfigure}
\caption{Same as \cref{fig:results_d_100_uniform} except that $d=500$. We omitted the result of Robust-AM because the subproblem solver (ADMM-LAD) did not reach its stopping criterion within 30 seconds even when $p_{\text{fail}}=0.1$.}
\label{fig:results_d_500_uniform}
\end{figure*}

\begin{table*}[t]
    \centering
    \caption{The running CPU time in seconds (outside the parentheses) and the the number of iterations (inside the parentheses) for the case where $\xi_i$ is uniformly distributed and $(d,p_{\text{fail}},s)=(500,0.35,1)$.}
    \begin{tabular}{|l|c|c|c|c|}
        \hline
        $n/d$ & 5 & 10 & 15 & 20 \\
        \hline
        \hline
        $\ell_1$ by IPL & 30.00 (251.66) & 7.14 (24.74) & 7.95 (7.44) &  10.13 (6.00)  \\
        Capped $\ell_1$ $(\beta = 100)$ & 0.62 (258.26) & 1.39 (313.78) & 1.47 (242.10) & 1.60 (213.88)\\
        Capped $\ell_1$ $(\beta = 1000)$ & 0.37 (150.86) & 0.56 (118.94) & 0.55 (82.28) & 0.53 (63.14)\\
        Capped $\ell_1$ $(\beta = 10000)$ & 0.39 (158.62) & 0.71 (154.98) & 0.42 (59.80) & 0.45 (51.16)\\
        Trimmed $\ell_1$ $(K/n = 0.2)$ & 1.44 (464.40) & 0.82 (129.86) & 0.93 (102.00) & 1.03 (89.76) \\
        Trimmed $\ell_1$ $(K/n = 0.3)$ & 2.15 (698.48) & 5.38 (969.08) & 6.27 (781.86) & 5.89 (595.94)\\
        Trimmed $\ell_1$ $(K/n = 0.4)$ & 1.87 (601.60) & 1.86 (53.18) & 1.82 (44.98) & 4.20 (44.98)\\
        \hline
    \end{tabular}
    \label{tab:execution time for uniform outlier}
\end{table*}
\clearpage
\begin{figure}[h]
\centering
\includegraphics[width=\columnwidth]{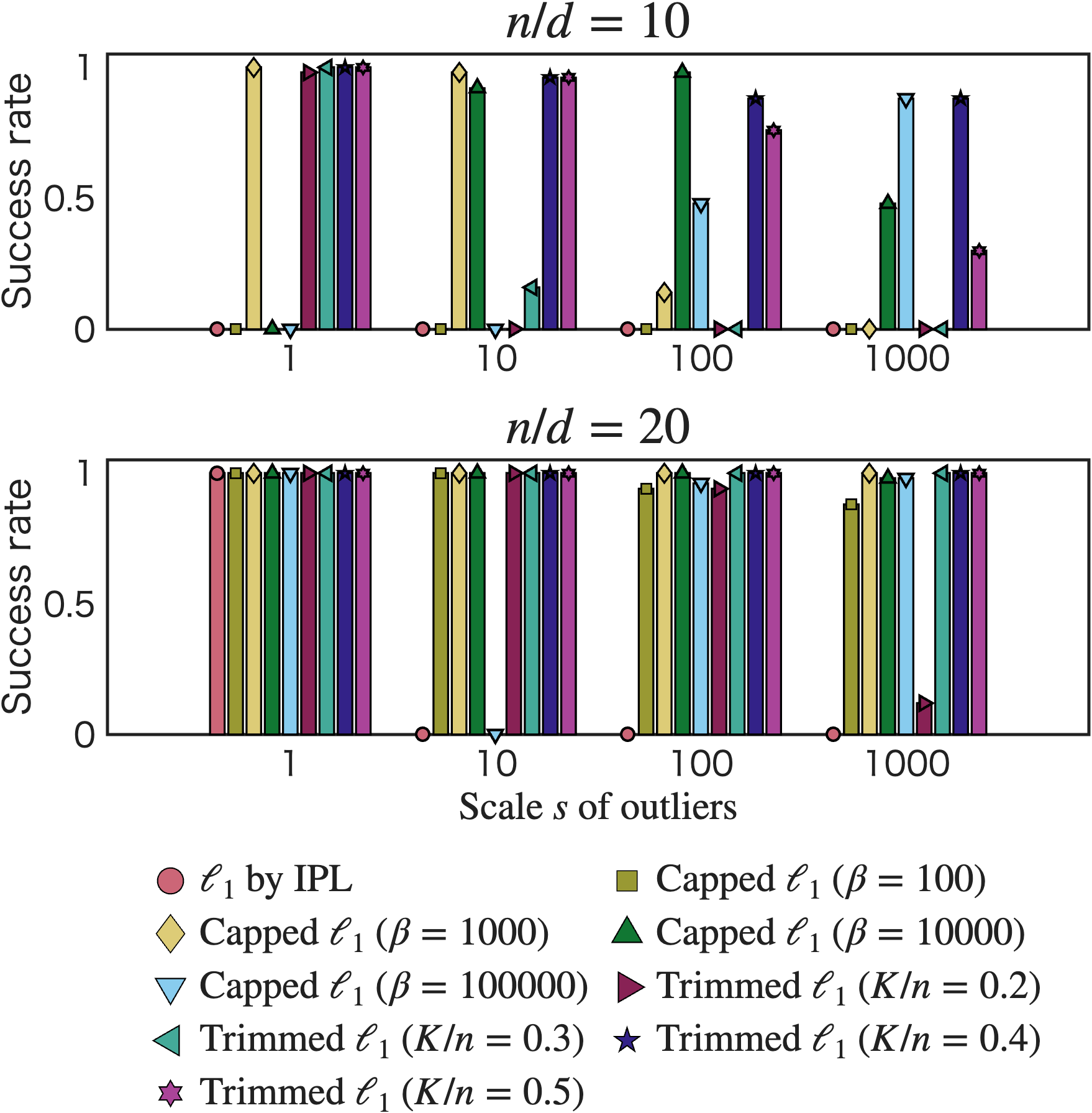}
\caption{Success rate for different values of $s$. 
We employed uniformly distributed $\xi_i$ and the same parameters $(d,p_{\text{fail}})=(500,0.4)$ as in \cref{fig:different scales of outliers}}
\label{fig:different scales of uniform outliers}
\end{figure}

%% file: my_macros.tex
\usepackage{cite}
\usepackage{amsmath}
\usepackage{amssymb}
\usepackage{bm}
\usepackage{amsthm}
\usepackage{mathtools}
\usepackage{algorithm}
\usepackage{algpseudocode}
\usepackage{xcolor}
\usepackage{graphicx}
\usepackage{subcaption}
\usepackage{array}
\usepackage{threeparttable}
\usepackage{dsfont}
\usepackage[hidelinks]{hyperref}
\usepackage{xurl}
\usepackage{cleveref}
\usepackage{autonum}
\usepackage{tikz}
\usetikzlibrary{intersections}
\usepackage{dblfloatfix}
\theoremstyle{definition}
\newtheorem{dfn}{Definition}[section]
\newtheorem{proposition}[dfn]{Proposition}
\newtheorem{theorem}[dfn]{Theorem}

\newtheorem{lemma}[dfn]{Lemma}
\newtheorem{fac}[dfn]{Fact}
\newtheorem{remark}[dfn]{Remark}
\newtheorem{assumption}[dfn]{Assumption}
\newtheorem{problem}[dfn]{Problem}
\newtheorem{example}[dfn]{Example}
\newtheorem*{notation}{Notation}
\newtheorem*{relatedwork}{Related works}

\crefname{theorem}{Theorem}{Theorems}
\crefname{proposition}{Proposition}{Propositions}
\crefname{dfn}{Definition}{Definitions}
\crefname{lemma}{Lemma}{Lemmas}
\crefname{fac}{Fact}{Facts}
\crefname{remark}{Remark}{Remarks}
\crefname{assumption}{Assumption}{Assumptions}
\crefname{problem}{Problem}{Problems}
\crefname{example}{Example}{Examples}

\crefname{equation}{}{}
\crefname{algorithm}{Algorithm}{Algorithms}
\crefname{figure}{Fig.}{Fig.}
\crefname{table}{Table}{Tables}
\crefname{chapter}{Chapter}{Chapters}
\crefname{section}{Section}{Sections}

\renewcommand{\labelenumi}{(\theenumi)}

\DeclarePairedDelimiter{\abs}{\lvert}{\rvert} %
\DeclarePairedDelimiter{\norm}{\lVert}{\rVert} %
\DeclarePairedDelimiter{\rbra}{\lparen}{\rparen} %
\DeclarePairedDelimiter{\cbra}{\lbrace}{\rbrace} %
\DeclarePairedDelimiter{\abra}{\langle}{\rangle} %
\DeclarePairedDelimiterX{\Set}[2]{\lbrace}{\rbrace}{#1\,\delimsize\vert\,#2}

\newcommand{\R}{\mathbb{R}}
\newcommand{\N}{\mathbb{N}}
\newcommand{\x}{\bm{x}}
\newcommand{\y}{\bm{y}}
\newcommand{\z}{\bm{z}}
\newcommand{\ba}{\bm{a}}
\newcommand{\bb}{\bm{b}}
\newcommand{\btilde}[1]{\bm{\tilde{#1}}}

\newcommand{\frakS}{\mathfrak{S}}

\newcommand{\backin}[2]{\quad(#1 \in #2)}
\newcommand{\map}[3]{#1: #2 \to #3}
\newcommand{\seq}[2]{(#1_#2)_{#2=1}^{\infty}}

\newcommand{\del}{\partial}
\newcommand{\nb}{\nabla}
\newcommand{\D}{\mathrm{D}}
\newcommand{\inv}[1]{\frac{1}{#1}}
\newcommand{\minimize}[2]{\underset{#1\in#2}{\text{minimize }}}
\newcommand{\dist}{\text{dist}}

\newcommand{\argmin}[1]{\underset{#1}{\text{argmin }}}
\newcommand{\mor}[2]{{}^{#2}#1}
\newcommand{\prox}[1]{\text{Prox}_{#1}}
\newcommand{\sign}{\text{sign}}
\newcommand{\toinf}{\to \infty}
\newcommand{\rmone}{\mathrm{I}}
\newcommand{\rmtwo}{{\mathrm{I}\hspace{-1.2pt}\mathrm{I}}}

\newcommand{\eqsize}[2]{\scalebox{#1}[1]{$\displaystyle #2$}}
\DeclareMathOperator*{\Limsup}{Limsup}

%% file: footnote.tex
\newcommand{\noteA}{\footnote{
  The name \textit{phase retrieval} comes from its applications where the measurement matrix $A:=[\ba_1,\ba_2,\cdots,\ba_n]^T$
  and the measurement $\bb^\star:=[\abs{\abra{\ba_1,\x^\star}}^2,\abs{\abra{\ba_2,\x^\star}}^2,...,\abs{\abra{\ba_n,\x^\star}}^2]^T\in\R^n$ represent a Fourier-type transform and a phaseless measurement, respectively.
  (For this reason, $A$ is also often considered as a complex matrix in phase retrieval literature; however, in this paper, we assume $A$ to be real-valued for simplicity as in many previous studies \cite{hand2016robust,hand2016corruption,zhang2018median,zheng2024new,kim2024robust}.)
}}

\newcommand{\noteB}{\footnote{
  The proximity operator of every $f$ in \cref{tab:options of phi} is found, e.g., in \cite[Exm. 6.8]{beck2017first}.
  On the other hand, the proximity operators of $g$ for the MCP function and the capped $\ell_1$ norm can be expressed with those of the univariate prox-friendly functions $\widehat{r}_{\lambda,\beta}$ and $\max\cbra*{|\cdot|-\beta,0}$ \cite[Thm. 6.6]{beck2017first}
  (see also \cite[Exm. 6.66, Thm. 6.12]{beck2017first} and \cite{prox_repository} for the detailed expressions of their proximity operators).
  Lastly, the proximity operator of $g(\z):=\sum_{i=1}^{K} |[\z]_{\downarrow i}|\ (\z\in\R^n)$ can be computed by a special case of \cite[Alg.4]{bogdan2015slope}.
}}
 
\newcommand{\noteC}{\footnote{
  Equation \cref{eq:equivalence of limiting and frechet subdifferential} is obtained by combining \cite[(Step1) of Proof of Lem.2.4]{kume2024variableLong} and \cite[Cor.8.11]{rockafellar2009variational} as below.
  Since $f\circ\mathfrak{S}$ and $g\circ\mathfrak{S}$ have a property called \textit{subdifferential regularity} (see \cite[Def. 7.25]{rockafellar2009variational}) by \cite[(Step1) of Proof of Lem. 2.4]{kume2024variableLong},
  we see from \cite[Cor. 8.11]{rockafellar2009variational} that \cref{eq:equivalence of limiting and frechet subdifferential} holds if $\del_{\mathrm{L}}(f\circ\frakS)(\x)\ne \emptyset$ and $\del_{\mathrm{L}}(g\circ\frakS)(\x)\ne \emptyset$.
  In a case where $\del_{\mathrm{L}}(f\circ\frakS)(\x) = \emptyset$ or $\del_{\mathrm{L}}(g\circ\frakS)(\x) = \emptyset$, \cref{eq:equivalence of limiting and frechet subdifferential} trivially holds
  from $\del_{\mathrm{L}}(f\circ\frakS)(\x)\supset \del_{\mathrm{F}}(f\circ\frakS)(\x)$ and $\del_{\mathrm{L}}(g\circ\frakS)(\x)\supset \del_{\mathrm{F}}(g\circ\frakS)(\x)$.
}}
\newcommand{\noteD}{\footnote{
  The code of IPL can be found in ``\url{https://github.com/zhengzhongpku/IPL-code-share}''.
}}

\newcommand{\noteE}{\footnote{
  We received the code of Robust-AM directly from Mr. Kim (The Ohaio State University) who is the first author of \cite{kim2024robust}.
}}

\newcommand{\noteF}{\footnote{
  However, we also see that the large $K$ results in low success rate when the number $n$ of measurements is small (see the lower-left region of \cref{fig:results_d_100_Cauchy} (h) and \cref{fig:results_d_500_Cauchy} (g)).
  This is probably because the trimmed $\ell_1$ with a large $K$ ignores not only outliers, but also many inliers, thereby excessively reducing the number of effective inlier measurements available for estimation, especially when the number $n$ of measurements is small.
}}

\newcommand{\noteG}{\footnote{
  Although this fact is elemntary, we provide a simple proof for our self-containedness.
  For any $\bm{v} \in \Limsup_{k\to\infty}(\bm{p}_k + \bm{q}_k)$, there exists a subsequence $(\bm{p}_{m(l)} + \bm{q}_{m(l)})_{l=1}^{\infty}\to \bm{v}$, where $\map{m}{\N}{\N}$ is monotonically increasing.
  From the boundedness of $\seq{\bm{q}}{k}$ and the Bolzano-Weierstrass theorem (see, e.g., \cite[Thm. 3.24]{Apostol1985mathematical}), we get $\bm{q}_{m(l)} \to \exists \bm{q} \in \R^d$ by passing through further a subsequence of $(\bm{q}_{m(l)})_{l=1}^{\infty}$ (and renaming it as $(\bm{q}_{m(l)})_{l=1}^{\infty}$).
  Hence, we have $\bm{q} \in \Limsup_{k\toinf} \bm{q}_k$ and $\bm{v}-\bm{q} = \lim_{l\toinf} ((\bm{p}_{m(l)} + \bm{q}_{m(l)}) - \bm{q}_{m(l)}) = \lim_{l\toinf} \bm{p}_{m(l)} \in \Limsup_{k\toinf} \bm{p}_k$, which implies $\bm{v} = (\bm{v}-\bm{q}) + \bm{q} \in \Limsup_{k\toinf} \bm{p}_k + \Limsup_{k\toinf} \bm{q}_k$.
}}

%% file: introduction.tex
\section{Introduction}\label{sec:introduction}
\IEEEPARstart{P}{hase} retrieval is a problem of estimating an original signal $\x^\star\in \R^d$ or $-\x^\star$ from a quadratic measurement
\begin{equation}
  \bb^\star := [\abra{\ba_1,\x^\star}^2,\abra{\ba_2,\x^\star}^2,\cdots,\abra{\ba_n,\x^\star}^2]^T \in \R^n,
\end{equation}
where $\ba_1,\ba_2,...,\ba_n\in \R^d$ are known measurement vectors.\noteA
The measurement $\bb^\star$ can also be expressed as 
\begin{equation}
  \bb^\star = (A\x^\star) \odot (A\x^\star)
\end{equation}
by using the matrix $A:=[\ba_1,\ba_2,\cdots,\ba_n]^T\in \R^{n\times d}$ and the Hadamard product (i.e., entry-wise product) $\odot$.
Phase retrieval problem arises in various applications including crystallography \cite{millane1990phase,hauptman1991phase}, optics \cite{walther1963question,shechtman2015phase}, and astronomy \cite{fienup1987phase}.
In this paper, we consider a scenario where only the following measurement $\bb$, corrupted by outliers, is available:
\begin{equation}\label{eq:corrrupted measurement}
  [\bm{b}]_i := 
  \begin{cases}
    \abra{\bm{a}_i,\x^\star}^2+\varepsilon_i & i\in \mathcal{I}_{\text{in}}\\
    \xi_i & i\in \mathcal{I}_{\text{out}},
  \end{cases}
\end{equation}
in which $\mathcal{I}_\text{out}\subset \{1,2,\ldots,n\}$ and $\mathcal{I}_\text{in} := \{1,2,\ldots,n\} \setminus \mathcal{I}_\text{out} $ denote index sets for outliers $\xi_i\in \R$ and inliers respectively, and $\varepsilon_i\in \R$ represents a certain additive noise, e.g., white Gaussian noise.
Such a corrupted measurement often appears in many phase retrieval imaging applications \cite{weller2015undersampled}, for example due to sensor failures and recording errors.

In a case where there is no outlier, i.e., $\mathcal{I}_\text{out}=\emptyset$, 
a variety of phase retrieval methods have been proposed over the decades, including \textit{Gerchberg-Saxton algorithm} \cite{gerhberg1972practical}, \textit{hybrid-input output method} \cite{fienup1978reconstruction,fienup1982phase}, \textit{PhaseLift} \cite{candes2015phase_PhaseLift}, and \textit{Wirtinger flow} \cite{candes2015phase_Wirtinger}, to name a few.
Recently, for the case $\mathcal{I}_\text{out} \ne \emptyset$, numerous studies have aimed to develop \textit{robust phase retrieval} methods for achieving high-precision estimation even in the presence of outliers \cite{weller2015undersampled,hand2016robust,hand2016corruption,zhang2018median,duchi2019solving,zheng2024new,kim2024robust}.
Among these, \cite{duchi2019solving} and \cite{zheng2024new} consider optimization-based methods analogous to the well-known \textit{least absolute deviation method} (see, e.g., \cite{bloomfield1983least}) in the field of robust statistics.
To be precise, they utilize a solution of the nonconvex optimization model 
\begin{equation}\label{eq:model with l1}
  \minimize{\x}{\R^d} \eqsize{0.94}{\Phi_1(\x) := \norm*{(A\x) \odot (A\x)-\bm{b}}_1 = \sum_{i=1}^n \abs*{\abra{\ba_i,\x}^2 - [\bb]_i}}
\end{equation}
as an estimated signal of $\x^\star$ or $-\x^\star$, where the $\ell_1$ norm  $\map{\norm{\cdot}_1}{\R^n}{\R},\ \z \mapsto  \sum_{i=1}^{n} |[\z]_i|$ serves as a loss function.
In addition, \cite{weller2015undersampled}, \cite{hand2016robust}, and \cite{kim2024robust} also adopt similar optimization models with the $\ell_1$ loss function.
More specifically, \cite{weller2015undersampled} uses a sparse regularized version of \cref{eq:model with l1} under the sparsity assumption on the target signal $\x^\star$;
\cite{hand2016robust} studies a convex relaxation of the model \cref{eq:model with l1};
and \cite{kim2024robust} employs a modified formulation of \cref{eq:model with l1} 
\begin{equation}\label{eq:modified model with l1}
  \minimize{\x}{\R^d} \Phi_2(\x) := \sum_{i=1}^n\abs*{\abs{\abra{\ba_i,\x}}-\sqrt{\abs{[\bb]_i}}},
\end{equation}
in which $\abra{\ba_i,\x}^2$ and $[\bb]_i$ are replaced by their square roots $\abs{\abra{\ba_i,\x}}$ and $\sqrt{\abs{[\bb]_i}}$, respectively.
It is reported in \cite{duchi2019solving,kim2024robust} that these methods using the $\ell_1$ norm achieve superior numerical estimation performance compared to other robust phase retrieval methods \cite{hand2016corruption,zhang2018median}.

Nevertheless, it is questionable whether the $\ell_1$ norm in the model \cref{eq:model with l1} can adequately suppress effects caused by outliers.
To explain this, we rewrite the cost function in \cref{eq:model with l1} as
$\sum_{i\in\mathcal{I}_{\text{in}}}\abs*{\abra{\bm{a}_i,\x}^2-\abra{\bm{a}_i,\x^\star}^2-\varepsilon_i} + \sum_{i\in\mathcal{I}_{\text{out}}}\abs*{\abra{\bm{a}_i,\x}^2-\xi_i}$.
When there are numerous outliers, or when each $\xi_i$ is large, the second summation also becomes large even if $\x$ is close to the target signal $\x^\star$ or $-\x^\star$.
In this situation, solutions of \cref{eq:model with l1} may deviate from the target $\x^{\star}$ or $-\x^{\star}$, from which the performance of the estimation via \cref{eq:model with l1} may deteriorate.
Indeed, similar deteriorations caused by the $\ell_1$ loss have been reported in robust regression \cite{yukawa2023linearly}, robust matrix factorization \cite{yao2018scalable}, and robust tensor recovery \cite{yang2015robust}.
From these observations, the $\ell_1$ norm is not necessarily an ideal loss function for achieving a robust estimation against outliers.
Hence, it is expected that the estimation performance of \cref{eq:model with l1} can be enhanced by replacing the $\ell_1$ norm with a more appropriate robust loss function.

\begin{table}[t]
  \begin{threeparttable}
    \small
    \setlength{\tabcolsep}{0.15em} 
    \setlength{\extrarowheight}{1ex}
    \caption{
      Examples of DC loss function $\varphi$ in the model \cref{eq:model with nonconvex}
    }
    \label{tab:options of phi}
    \begin{tabular}{|c|c|c|c|}
      \hline
      name & $\eqsize{0.95}{\varphi(\z)=(f-g)(\z)}$ & $f(\z)$ & $g(\z)$ \\ \hline
      $\ell_1$ norm & $\displaystyle\sum_{i=1}^{n}|[\z]_i|$ & $\displaystyle\sum_{i=1}^{n}|[\z]_i|$ & 0 \\ 

      \begin{tabular}{c}
        MCP \\[-1ex]
        \cite{zhang2010nearly}
      \end{tabular}& 
      \begin{tabular}{c}
        $\displaystyle\sum_{i=1}^{n} r_{\lambda,\beta}([\z]_i)$ \tnote{*a} \\[-1ex]
        ($\lambda, \beta\in\R_{++}$)
      \end{tabular}
      & $\displaystyle\lambda\sum_{i=1}^{n}|[\z]_i|$ & $\displaystyle\sum_{i=1}^{n} \widehat{r}_{\lambda,\beta}([\z]_i)$ \tnote{*b} \\

      \begin{tabular}{c}
        Capped $\ell_1$ \\[-1ex]
        \cite{zhang2008multi}
      \end{tabular}&
      \begin{tabular}{c}
        $\displaystyle\sum_{i=1}^{n} \min\cbra*{|[\z]_i|,\beta}$ \\[-1ex]
        ($\beta\in\R_{++}$)
      \end{tabular}
      & $\displaystyle\sum_{i=1}^{n}|[\z]_i|$ & $\eqsize{0.94}{\displaystyle\sum_{i=1}^{n}\max\cbra*{|[\z]_i|\text{\,--\,}\beta,0}}$ \\ 

      \begin{tabular}{c}
        Trimmed $\ell_1$\\[-1ex]
        \cite{hossjer1994rank}
      \end{tabular}
      &
      \begin{tabular}{c}
        $\displaystyle\sum_{i=K+1}^{n} |[\z]_{\downarrow i}|$ \tnote{*c}\\[-1ex]
        ($0\le K< n$)
      \end{tabular}
      & $\displaystyle\sum_{i=1}^{n}|[\z]_i|$ & $\displaystyle\sum_{i=1}^{K} |[\z]_{\downarrow i}|$ \\ \hline
    \end{tabular}
  
    \vspace{0.5ex}
    \begin{tablenotes}
      \item[*a] 
      $r_{\lambda,\beta}(t) :=
      \begin{cases*}
        \lambda|t| - \frac{t^2}{2\beta} & $|t|\le \beta \lambda$, \\
        \frac{\beta \lambda^2}{2}         & otherwise. 
      \end{cases*}
      $
      \vspace{0.5ex}
      \item[*b] 
      $\widehat{r}_{\lambda,\beta}(t) :=
      \begin{cases*}
        \frac{t^2}{2\beta} & $|t|\le \beta \lambda$, \\
        \lambda|t| - \frac{\beta \lambda^2}{2}         & otherwise.
      \end{cases*}
      $
      
      \item[*c]
      $[\z]_{\downarrow i}$ denotes the entry of $\z$ whose absolute value is the $i$-th largest.
    \end{tablenotes}
  \end{threeparttable}
  \vspace{-10pt}
\end{table}

In this paper, in order to adopt more robust loss functions than the $\ell_1$ norm, we propose the following generalized model of $\cref{eq:model with l1}$:
\begin{equation}\label{eq:model with nonconvex}
  \minimize{\x}{\R^d} \Phi_3(\x) := \varphi \rbra*{(A\x) \odot (A\x)-\bm{b}},
\end{equation}
where $\map{\varphi}{\R^n}{\R}$ is given as a DC (Difference-of-Convex) function, i.e., $\varphi$ can be expressed as a difference $f-g$ of two convex functions $f$ and $\map{g}{\R^n}{\R}$.
The DC loss functions $\varphi$ include not only the $\ell_1$ norm but also nonconvex functions
such as the MCP (Minimax-Concave-Penalty) function \cite{zhang2010nearly,yukawa2023linearly}, the capped $\ell_1$ norm \cite{zhang2008multi,sun2013robust}, and the trimmed $\ell_1$ norm \cite{hossjer1994rank,mafusalov2016cvar,gotoh2018dc,yagishita2025exact} to name a few (see \cref{tab:options of phi} for typical DC loss functions and their DC decompositions).
Such nonconvex DC functions have been employed as loss functions instead of the $\ell_1$ norm in the fields of robust estimation, such as robust regression \cite{yukawa2023linearly},\cite{hawkins1999applications} and robust low rank matrix recovery \cite{sun2013robust,Kume-Yamada25C}
(see \cref{rma:Robustness of nonconvex DC functions} for intuitive reasons why these nonconvex functions are promising robust loss functions).

\begin{remark}[Robustness of nonconvex DC functions]\label{rma:Robustness of nonconvex DC functions}
  \leavevmode
  \begin{enumerate}
    \item (Upper bounded loss functions)
    As seen from \cref{tab:options of phi}, the MCP function and the capped $\ell_1$ norm are given respectively by the separable sum of the univariate functions whose outputs are always bounded above by a certain tunable constant.
    Hence, with a properly tuned constant, these loss functions do not overpenalize large outliers, unlike the $\ell_1$ norm.
    \item (Trimmed $\ell_1$ norm)
    The trimmed $\ell_1$-norm outputs the sum of the smallest $n-K$ absolute values of input vector entries, where $K$ is a tunable parameter.
    If $K$ is set properly, e.g., as the number of outliers, the trimmed $\ell_1$ norm can suppress an influence caused by large outliers. 
    Thus, the trimmed $\ell_1$ norm can be an alternative robust loss function.
  \end{enumerate}
\end{remark}

Although these nonconvex functions are promising loss functions, existing optimization algorithms \cite{weller2015undersampled,hand2016robust,duchi2019solving,zheng2024new,kim2024robust} employed for $\ell_1$ loss-based models are not directly applicable to the proposed model \cref{eq:model with nonconvex} with nonconvex $\varphi$
because these algorithms rely on the convexity of the $\ell_1$ norm.

In this paper, we propose an optimization algorithm applicable to the model \cref{eq:model with nonconvex} by exploiting a fact that the cost function in \cref{eq:model with nonconvex} is the composition of a DC function $\varphi$ with a smooth mapping
\begin{equation}\label{eq:S for robust phase retrieval}
  \map{\frakS_{\text{RPR}}}{\R^d}{\R^n},\ \x \mapsto (A\x)\odot(A\x)-\bb.
\end{equation}
To broaden the applicability of the proposed algorithm (see \cref{rem:Applicability of problem} (c)), we consider the following optimization problem that includes the model \cref{eq:model with nonconvex} (see \cref{rem:Applicability of problem} (b)).
\begin{problem}[DC composite-type problem]\label{problem}
  \begin{equation}
    \underset{\x \in \R^d}{\text{minimize }}F(\x):= \underbrace{(f-g)}_{\textstyle\varphi} \circ\, \mathfrak{S}(\x), \label{eq:problem}
  \end{equation}
  where
  \begin{enumerate}
    \item $\map{f}{\R^n}{\R}$ and $\map{g}{\R^n}{\R}$ are
    \item[]
    \begin{enumerate}
      \item $\eta_{f}$- and $\eta_{g}$-weakly convex with $\eta_{f}, \eta_{g} >0$, i.e., $f+\frac{\eta_f}{2}\norm{\cdot}^2$ and $g+\frac{\eta_g}{2}\norm{\cdot}^2$ are convex\\(we define \scalebox{0.95}{$\eta:=\max\{\eta_f,\eta_g\}$} for convenience),
      \item $L_f$- and $L_g$-Lipschitz continuous with $L_f,L_g>0$, i.e.,  $\abs{f(\x)-f(\y)}\le L_f\norm{\x-\y}$ and $\abs{g(\x)-g(\y)}\le L_g\norm{\x-\y}$ hold for all $\x,\y\in \R^d$,
      \item prox-friendly, i.e, their \textit{proximity operators} (see \cref{moreau envelope}) are available as computable tools
    \end{enumerate}
    (see \cref{rem:Applicability of problem}~(a) for a reason why we assume weak convexity in (i) instead of convexity);
    \item $\map{\mathfrak{S}}{\R^d}{\R^n}$ is differentiable and its Fr\'{e}chet derivative $\map{{\rm D} \mathfrak{S}}{\R^d}{\R^{n \times d}}$ is $L_{\D\frakS}$-Lipschitz continuous (see Notation for the definition of Fr\'{e}chet derivative);
    \item $F$ is bounded below, i.e., $\inf_{\x \in \R^d} F(\x)> -\infty$.
  \end{enumerate}
\end{problem}
\begin{remark}[Applicability of \cref{problem}]\label{rem:Applicability of problem}
  \leavevmode
  \begin{enumerate}
    \item (Functions expressed as $f-g$)
      All functions $\varphi$ in \cref{tab:options of phi} admit DC decompositions $\varphi=f-g$ where both $f$ and $g$ are just convex, Lipschitz continuous, and prox-friendly\noteB  functions.
      By virtue of assuming weak convexity (rather than convexity) for $f$ and $g$ as in (a)(i),
      we can also employ a wider variety of robust loss functions (or sparsity-promoting functions; see \cref{rem:Applicability of problem}~(c)) as $\varphi$ beyond those listed in \cref{tab:options of phi}.
      For example, the \textit{Cauchy loss} (see, e.g., \cite{mlotshwa2022cauchy}) and the \textit{log-sum penalty} \cite{candes2008enhancing} are weakly convex, Lipschitz continuous, and prox-friendly;
      hence, they can be used as $\varphi$ in \cref{problem} by setting $f=\varphi$ and $g\equiv 0$.
    \item (Robust phase retrieval)
    The model \cref{eq:model with nonconvex} with all $\varphi$ in \cref{tab:options of phi} is reproduced as a special case of \cref{problem} by setting $f$ and $g$ as in \cref{tab:options of phi}, and $\frakS := \frakS_{\text{RPR}}$.
    Indeed, $f$ and $g$ in \cref{tab:options of phi} satisfy the assumptions in \cref{problem} as stated in (a), and $\D\frakS_{\text{RPR}}:\x\mapsto [2\abra{\ba_1,\x }\ba_1,2\abra{\ba_2,\x }\ba_2,...,2\abra{\ba_n,\x }\ba_n]^T$ is $(2\sqrt{\sum_{i=1}^n \norm{\ba_i}^4})$-Lipschitz continuous (see \cref{eq:Lipschitz continuity of DS_RPR}).
    \item (Example of applications beyond phase retrieval)
    DC functions $\varphi = f-g$ in \cref{tab:options of phi} have also been used as regularization functions to promote the sparsity of $\frakS(\x)$.
    Hence, \cref{problem} also appears in sparsity-aware applications such as image restoration \cite{you2019nonconvex}, compressed sensing \cite{huang2015two}, and cardinality-constrained linear regression \cite{gotoh2018dc}.
    In such applications, optimization models with a sparse regularization term $(f-g) \circ \frakS (\x)$ are typically formulated as 
    \begin{equation}
      \underset{\x \in \R^d}{\text{minimize }} h(\x) + (f-g) \circ \frakS (\x), \label{eq:problem with h}
    \end{equation}
    where $\map{h}{\R^d}{\R}$ is differentiable with a Lipschitz continuous gradient and serves as a data fidelity, e.g., least squares.
    The problem \cref{eq:problem with h} is a special instance of \cref{problem},
    because the cost function of \cref{eq:problem with h} can be translated into the form $(\widehat{f}-\widehat{g}) \circ \widehat{\frakS} (\x)$
    by introducing $\map{\widehat{\frakS}}{\R^d}{\R^n\times\R}: \x \mapsto [\frakS(\x)^T, h(\x)]^T$, $\map{\widehat{f}}{\R^n\times\R}{\R}:[\z^T,t]^T\mapsto f(\z) + t$, and $\map{\widehat{g}}{\R^n\times\R}{\R}:[\z^T,t]^T\mapsto g(\z)$.
    Indeed, if $f,g$, and $\frakS$ satisfy the assumptions in \cref{problem}, then $\widehat{f},\widehat{g}$, and $\widehat{\frakS}$ do as well.
  \end{enumerate}
\end{remark}

The proposed algorithm for DC composite-type problem (\cref{problem}) is designed as a gradient descent update of a time-varying smoothed surrogate function of $F$ in \cref{eq:problem}.
With the \textit{Moreau envelopes} (see \cref{moreau envelope}) $\mor{f}{\mu}$ of $f$ and $\mor{g}{\mu}$ of $g$, the proposed surrogate function is given as $(\mor{f}{\mu_k}-\mor{g}{\mu_k})\circ \mathfrak{S}$, 
where $\seq{\mu}{k} \subset \R$ is a monotonically decreasing sequence of convergence to zero.
By utilizing the proximity operators of $f$ and $g$, the proposed algorithm can be implemented as a single-loop algorithm for \cref{problem} including the model \cref{eq:model with nonconvex}.
We present an asymptotic convergence analysis (\cref{thm:convergence theorem}) in the sense of a \textit{DC composite critical point} (see \cref{df:critical point}) under \cref{asm:descent assumption} on the surrogate function. (This assumption is satisfied for the model \cref{eq:model with nonconvex}; see \cref{pro:sufficient condition for descent assumption} for details.)

Our numerical experiment in scenarios with numerous outliers demonstrates that the proposed method, based on the model \cref{eq:model with nonconvex} with DC loss functions, 
achieves higher estimation performance than existing state-of-the-art methods \cite{zheng2024new,kim2024robust}. 

\begin{relatedwork}
For \cref{problem}, a recently developed \textit{DC composite algorithm} (DCCA) \cite{le2024minimizing} can be used.
If an exact solution to a certain subproblem in DCCA is available, then DCCA has a convergence guarantee in terms of a DC composite critical point.
In practice, however, DCCA requires an infinite number of iterations of an inner loop in order to find the exact solution of the subproblem. 
The convergence analysis of DCCA does not cover realistic cases where only inexact solutions of the subproblem are available.
In contrast, the proposed algorithm has a convergence guarantee and does not require infinite iterations of an inner loop.

The proposed algorithm serves as an extension of \textit{variable smoothing-type algorithms} \cite{bohm2021variable,kume2024variable},
originally developed for a special case $g\equiv 0$ of the problem \cref{eq:problem with h} (more precisely, $\frakS$ is assumed to be linear in \cite{bohm2021variable}), where \cref{eq:problem with h} is an instance of \cref{problem} (see \cref{rem:Applicability of problem} (c)).
Note that our extension enables us to cover the model \cref{eq:model with nonconvex} with non-weakly convex DC loss function $\varphi$ such as the capped $\ell_1$ norm and the trimmed $\ell_1$ norm.

For a special case $\frakS = \mathrm{Id}$ of \cref{eq:problem with h}, we also found a similar algorithm \cite{sun2023algorithms} to the proposed algorithm in the sense that the Moreau envelopes of $f$ and $g$ are exploited.
This algorithm is based on a certain approximate gradient descent method of a smoothed surrogate function $h+ \mor{f}{\mu} - \mor{g}{\mu}$ with fixed $\mu>0$. 
Even with such a fixed surrogate function, the algorithm \cite{sun2023algorithms} has a convergence guarantee to a DC composite critical point (see \cite[Thm. 2]{sun2023algorithms}) without requiring any inner loop. 
However, we have not found yet any extension of the idea in \cite{sun2023algorithms} that can handle a nonlinear $\frakS$ such as $\frakS_{\text{RPR}}$.
\end{relatedwork}

\begin{notation}
  $\N$, $\R$ and $\R_{++}$ denote respectively the sets of all positive integers, all real numbers and all positive real numbers.
  $\norm{\cdot}$ and $\abra{\cdot,\cdot}$ are respectively the Euclidean norm and the Euclidean inner product.
  For subsets $S_1,S_2$ of Euclidean space, we define $S_1 \pm S_2 := \Set{\bm{v}_1 \pm \bm{v}_2}{\bm{v}_1\in S_1, \bm{v}_2\in S_2}$.
  For $\bm{v}\in\R^n$, $[\bm{v}]_i\in\R$ stands for the $i$-th entry.
  The operator norm for a matrix $X\in\R^{n\times d}$ is defined by $\norm{X}_{\mathrm{op}}:=\sup_{\norm{\y}\le 1}\norm{X\y}$.
  We use $\mathrm{Id}$ to denote the identity mapping.
  For Euclidean spaces $\mathcal{X},\mathcal{Y}$ and a continuously differentiable mapping $J:\mathcal{X} \to \mathcal{Y}$, its Fr\'{e}chet derivative at $\x \in \mathcal{X}$ is 
  the linear operator $\map{\mathrm{D}J(\x)}{\mathcal{X}}{\mathcal{Y}}$ such that $\lim_{\mathcal{X}\setminus{\{\bm{0}\}}\ni\bm{h} \to \bm{0}}\frac{J(\x+\bm{h})-J(\x)-\D J(\x)[\bm{h}]}{\norm{\bm{h}}} = 0$.
  (Note that we also regard $\D J(\x)$ as a matrix, because every linear operator can be represented by matrix-vector multiplication in finite dimensions.)
  In particular with $\mathcal{Y}=\R$, $\map{\nb J}{\mathcal{X}}{\mathcal{X}}$ is called the gradient of $J$ if
  $\nb J(\x) \in \mathcal{X}$ at $\x \in \mathcal{X}$ satisfies $\mathrm{D}J(\x)[\bm{v}] = \abra{\nb J(\x),\bm{v}}\ (\bm{v}\in\mathcal{X})$.
  For a point sequence $\seq{\bm{p}}{k}\subset \mathcal{X}$, we define its \textit{outer limit} as %
  \begin{equation}
    \Limsup_{k\toinf}\bm{p}_k := \Set{\bm{p}\in \mathcal{X}}{\bm{p}\text{ is a cluster point of }\seq{\bm{p}}{k}},
  \end{equation}
  where the outer limit is originally defined for set sequences (see, e.g., \cite[Def. 4.1]{rockafellar2009variational}), but we only use the outer limit for point sequences (i.e., the outer limit for sequences of singletons).
\end{notation}

%% file: preliminary.tex
\section{Preliminary}
\noindent As an extension of the subdifferential of convex functions, we use the following subdifferential of nonconvex functions
(see, e.g., a recent survey \cite{li2020understanding} for readers who are unfamiliar with nonsmooth analysis).
\begin{dfn}[Subdifferential {\cite[Def. 8.3]{rockafellar2009variational}}]
  For a function $\map{\psi}{\R^N}{\R}$, the \textit{limiting (or general) subdifferential} of $\psi$ at $\bar{\x} \in \R^N$ is defined by
  \begin{equation}
    \eqsize{0.90}{
    \del_{\mathrm{L}}\psi(\bar{\x}):=\Set*{\bm{v}\in\R^N}{
      \begin{aligned}
        &\exists \seq{\x}{k}\to\bar{\x},\ \exists\bm{v}_k \in \del_{\mathrm{F}}\psi(\x_k)\ (k\in\N)\\ 
        &\text{s.t. } \seq{\bm{v}}{k}\to \bm{v} \text{ and } \psi(\x_k)\to \psi(\bar{\x})
      \end{aligned}
    }.}
  \end{equation}
  Here, $\del_{\mathrm{F}} \psi(\hat{\x}) \subset \R^N$ denotes the \textit{Fr\'{e}chet (or regular) subdifferential} at $\hat{\x}\in\R^N$, and it is the set of all vectors $\bm{w} \in \R^N$ such that
  \begin{equation}
    \lim_{\delta \searrow 0}\ \inf_{0 < \norm{\x - \hat{\x}} < \delta}\frac{\psi(\x)-\psi(\hat{\x})-\abra{\bm{w},\x-\hat{\x}}}{\norm{\x-\hat{\x}}} \ge 0.
  \end{equation}
\end{dfn}
From the definition, $\del_{\mathrm{L}}\psi(\bar{\x})\supset \del_{\mathrm{F}}\psi(\bar{\x})$ obviously holds.
If $\psi$ is convex, the limiting subdifferential is equivalent to the convex subdifferential \cite[Prop. 8.12]{rockafellar2009variational}.
Furthermore, if $\psi$ is continuously differentiable on a neighborhood of $\bar{\x}$, $\del_{\mathrm{L}}\psi(\bar{\x}) = \cbra*{\nb \psi(\bar{\x})}$ holds \cite[Exe. 8.8 (b)]{rockafellar2009variational}.

\begin{fac}[{\cite[(Step1) of Proof of Lem. 2.4]{kume2024variableLong},\cite[Cor. 8.11]{rockafellar2009variational}}]\label{fac:equivalence of limiting and frechet subdifferential}
  Consider \cref{problem}. Then, for any $\x\in\R^d$, we have\noteC
  \begin{equation}\label{eq:equivalence of limiting and frechet subdifferential}
    \del_{\mathrm{L}} (f\circ\frakS)(\x) = \del_{\mathrm{F}} (f\circ\frakS)(\x),\ \del_{\mathrm{L}} (g\circ\frakS)(\x) = \del_{\mathrm{F}} (g\circ\frakS)(\x).
  \end{equation}
\end{fac}
By \cref{fac:equivalence of limiting and frechet subdifferential}, useful facts for limiting and Fr\'{e}chet subdifferentials (see, e.g., \cite[Ch. 8]{rockafellar2009variational}) are applicable to $f\circ \mathfrak{S}$ and $g\circ\mathfrak{S}$.

Unfortunately, finding a global minimizer of \cref{problem} is not realistic due to the severe nonconvexity of $F$.
Instead, in this paper, we focus on finding a DC composite critical point defined, with the limiting subdifferentials, as follows.
\begin{dfn}[DC composite critical point for \cref{problem} {\cite[Def. 1.1]{le2024minimizing}}]\label{df:critical point}
  We call $\x^\star \in \R^d$ a \textit{DC composite critical point} for \cref{problem} if
  \begin{equation}\label{eq:critical point}
    \del_{\mathrm{L}} (f \circ \mathfrak{S})(\x^\star)- \del_{\mathrm{L}} (g \circ \mathfrak{S})(\x^\star) \ni \bm{0},
  \end{equation}
  or equivalently,
  \begin{equation}\label{eq:critical point alternative form}
    \del_{\mathrm{L}} (f \circ \mathfrak{S})(\x^\star) \cap \del_{\mathrm{L}} (g \circ \mathfrak{S})(\x^\star)\ne \emptyset.
  \end{equation}
  (More precisely, \cite[Def. 1.1]{le2024minimizing} employs the condition obtained by replacing $\del_{\mathrm{L}}$ in \cref{eq:critical point alternative form} with $\del_{\mathrm{F}}$;
  however, \cref{fac:equivalence of limiting and frechet subdifferential} shows that this condition coincides with \cref{eq:critical point alternative form} in our setting.)
\end{dfn}

\begin{lemma}[Local optimality implies DC composite criticality]\label{lem:Local optimality implies DC criticality}
Let  $\x^\star \in \R^d$ be a local minimizer of $F$ in \cref{problem}. Then, $\x^\star$ is a DC composite critical point for \cref{problem}.
\end{lemma}
\begin{proof}
  See Appendix A.
\end{proof}
From \cref{lem:Local optimality implies DC criticality}, being a DC composite critical point is a necessary condition for being a local minimizer.
In a case where $f\circ\frakS$ and $g\circ\frakS$ are convex, finding such a DC composite critical point has been used as an acceptable goal in many DC optimization literature \cite{le2024open,gotoh2018dc,zhang2024inexact,sun2023algorithms}.
Even when $f\circ\frakS$ and $g\circ\frakS$ are not convex, DC composite critical points are adopted as the target of DC composite algorithm in \cite{le2024minimizing}.

The Moreau envelope plays an important role in this paper for designing the proposed algorithm. 
\begin{dfn}[Moreau envelope, proximity operator \cite{bohm2021variable}]\label{moreau envelope}
  Let $\map{\psi}{\R^n}{\R}$ be an $\eta_\psi$-weakly convex function with $\eta_\psi > 0$. Its Moreau envelope and proximity operator at $\bar{\z} \in \R^n$ with $\mu \in (0,\eta_\psi^{-1})$ are respectively defined as
  \begin{align}
    \mor{\psi}{\mu}(\bar{\z}) &:= \min_{\z \in \R^n} \cbra*{\psi(\z) + \frac{1}{2\mu}\norm*{\z-\bar{\z}}^2}, \\
    \prox{\mu \psi}(\bar{\z})  &:= \argmin{\z \in \R^n} \cbra*{\psi(\z) + \frac{1}{2\mu}\norm*{\z-\bar{\z}}^2},
  \end{align}
  where $\prox{\mu \psi}$ is single-valued due to the strong convexity of $\psi + (2\mu)^{-1}\norm{\ \cdot\  - \bar{\z}}^2$.
\end{dfn}

The Moreau envelope $\mor{\psi}{\mu}$ serves as a smoothed surrogate function of $\psi$
because of the next properties.
\begin{fac}[Properties of Moreau envelope]\label{Properties of Moreau envelope}
  Let $\map{\psi}{\R^n}{\R}$ be an $\eta_\psi$-weakly convex function with $\eta_\psi > 0$.
  For $\mu \in (0,\eta_\psi^{-1})$, the following hold.
  \begin{enumerate}
    \item \cite[Thm. 1.25]{rockafellar2009variational} $(\z \in \R^n)\ \lim_{\mu \searrow 0} \mor{\psi}{\mu}(\bm{z}) = \psi(\bm{z})$.
    \item \cite[Cor. 3.4]{hoheisel2020regularization} $\mor{\psi}{\mu}$ is continuously differentiable with $(\z \in \R^n)\ \nb\mor{\psi}{\mu}(\z) = \mu^{-1}\rbra*{\z - \prox{\mu \psi}(\z)}$.
    \item \cite[Cor. 3.4]{hoheisel2020regularization} $\nb\mor{\psi}{\mu}$ is $L_{\nb\mor{\psi}{\mu}}$-Lipschitz continuous with $L_{\nb\mor{\psi}{\mu}}:=\max\cbra{\mu^{-1},\frac{\eta_\psi}{1-\eta_\psi\mu} }$.
  \end{enumerate}
\end{fac}
Note that for $f$ and $g$ in \cref{problem}, we can compute $\nb\mor{f}{\mu}$ and $\nb\mor{g}{\mu}$ in closed-forms because these functions are assumed to be prox-friendly (see \cref{problem} (a)(iii)). 

%% file: algorithm.tex
\section{Variable Smoothing Algorithm for \\ \hspace{1.5em} DC Composite Problem}
\subsection{Design of Smooth Surrogate Function}
\noindent In our algorithm, we use the smooth surrogate function  
\begin{equation}
    F^{\abra{\mu}} := (\mor{f}{\mu}-\mor{g}{\mu})\circ \mathfrak{S} \backin{\mu}{\rbra{0,\eta^{-1}}} \label{eq:surrogate function}
\end{equation}
of $F$ in place of the direct utilization of the nonsmooth function $F$.
The next theorem suggests how to find a DC composite critical point in \cref{eq:critical point} using the surrogate function $F^{\abra{\mu}}$.
\begin{theorem}[DC gradient sub-consistency]\label{thm:dc gradient sub-consistency}
  Suppose that a positive sequence $\seq{\mu}{k} \subset \rbra{0,\eta^{-1}} $ converges to $0$.
  For $F_k := F^{\abra{\mu_k}}\ (k\in\N)$ with \cref{eq:surrogate function} and any convergent sequence $\seq{\x}{k} \subset \R^d \to \exists \bar{\x} \in \R^d$, the following hold:
    \begin{enumerate}
      \item
      \begin{equation}
        \Limsup_{k\toinf} \nb F_k(\x_k) \subset \del_{\mathrm{L}} (f\circ\frakS)(\bar{\x}) - \del_{\mathrm{L}} (g\circ\frakS)(\bar{\x}).
      \end{equation}
      \item
      \begin{multline}\label{eq:bound of distance}
        \text{dist}\rbra[\Big]{\bm{0},\del_{\mathrm{L}}(f \circ \mathfrak{S})(\bar{\x})-\del_{\mathrm{L}}(g \circ \mathfrak{S})(\bar{\x})}\\
        \le \liminf_{k \to \infty}\norm{\nb F_k(\x_k)},
      \end{multline}
      where $\text{dist}(\bm{v},S):=\inf_{\bm{w}\in S}\norm{\bm{v}-\bm{w}}$ stands for the distance between a point $\bm{v}\in\R^d$ and a set $S\subset\R^d$.
    \end{enumerate}
\end{theorem}
\begin{proof}
  See Appendix B.
\end{proof}

\cref{thm:dc gradient sub-consistency} (b) implies that $\bar{\x}$ is a DC composite critical point in the sense of \cref{eq:critical point} if the right hand side of \cref{eq:bound of distance} is zero.
Hence, our goal of finding a DC composite critical point of \cref{problem} is reduced to designing an algorithm to generate a point sequence $\seq{\x}{k}$ such that $\liminf_{k \to \infty}\norm{\nb F_k(\x_k)}=0$.
To design such an algorithm, we introduce the following assumption.
Note that this assumption holds for the model \cref{eq:model with nonconvex} as will be stated in \cref{pro:sufficient condition for descent assumption}.
\begin{assumption}[Descent assumption]\label{asm:descent assumption}
  For $F$ in \cref{problem}, consider $F^{\abra{\cdot}}$ in \cref{eq:surrogate function}.
  There exist $\varpi_1,\varpi_2\in \R_{++}$ such that the following inequality holds for all $\x,\y \in \R^d$ and $\mu\in(0,(2\eta)^{-1}]$:
  \begin{equation}\label{eq:descent assumption}
    F^{\abra{\mu}}(\y) \le F^{\abra{\mu}}(\x) + \abra{\nb F^{\abra{\mu}}(\x), \y-\x} + \frac{\kappa_\mu}{2}\norm{\y-\x}^2,
  \end{equation}
  where $\kappa_\mu := \varpi_1 + \varpi_2\mu^{-1}$. (Note: we can implement the proposed algorithm, illustrated in \cref{subsec:Proposed Algorithm and Its Convergence Analysis}, without any knowledge on the value of $\varpi_1$ and $\varpi_2$.)
\end{assumption}

From the descent lemma (see, e.g.,\cite[Lem. 5.7]{beck2017first}), \cref{asm:descent assumption} is satisfied if $\nb F^{\abra{\mu}}$ is Lipschitz continuous with a Lipschitz constant $\varpi_{1}+\varpi_{2}\mu^{-1}$.
By exploiting this standard result, \cref{pro:sufficient condition for descent assumption} (b) below presents a sufficient condition for \cref{asm:descent assumption}.
However, this sufficient condition can not be applied to the proposed model \cref{eq:model with nonconvex} because $\frakS_{\text{RPR}}$ in \cref{eq:S for robust phase retrieval} is not Lipschitz continuous.
Instead, we show directly that the model \cref{eq:model with nonconvex} with $\varphi$ in \cref{tab:options of phi} satisfies \cref{asm:descent assumption} in \cref{pro:sufficient condition for descent assumption} (a).
\begin{proposition}[\fontsize{9.7pt}{10pt}\selectfont Sufficient conditions for \cref{asm:descent assumption}]\label{pro:sufficient condition for descent assumption}
  \leavevmode
  \begin{enumerate}
    \item For the model \cref{eq:model with nonconvex} with $\varphi=f-g$ chosen from \cref{tab:options of phi}, $F:=\Phi_3=(f-g)\circ\frakS_{\text{RPR}}$ with $\frakS_{\text{RPR}}$ in \cref{eq:S for robust phase retrieval} satisfies \cref{asm:descent assumption} with
    \begin{multline}
      \kappa_\mu:= 2L_g \sqrt{\sum_{i=1}^n \norm{\ba_i}^4} + 6\lambda \max_{1\le i \le n}\norm{\ba_i}^2\\
      + 4 \rbra[\Big]{\max_{1\le i \le n}\norm{\ba_i}^2|[\bb]_i|}\mu^{-1},
    \end{multline}
    where $\lambda > 0$ is a parameter of the MCP function, and $\lambda$ is understood as $1$ for the other $\varphi$ in \cref{tab:options of phi}.
    \item If $\frakS$ is $L_{\frakS}$-Lipschitz continuous, then \cref{asm:descent assumption} holds with $\kappa_\mu:=L_{\D \frakS}(L_f + L_g)+2L_{\frakS}^2 \mu^{-1}$. %
    (This is a variant of \cite[Prop. 4.2 (a)]{kume2024variableLong}.)
    \item
    Consider the cost function $h + F:=h+(f-g)\circ \frakS$ of the problem \cref{eq:problem with h} and its reformulation $\widehat{F}:= (\widehat{f}-\widehat{g})\circ\widehat{\frakS}$ in \cref{rem:Applicability of problem} (c).
    If $F^{\abra{\cdot}}$ satisfies \cref{asm:descent assumption} with $\kappa_\mu$, then $\widehat{F}^{\abra{\cdot}}$ also satisfies \cref{asm:descent assumption} with $\widehat{\kappa}_\mu:= L_{\nb h}+\kappa_\mu$
    with a Lipschitz constant $L_{\nb h}>0$ of $\nb h$.
  \end{enumerate}
\end{proposition}
\begin{proof}
  See Appendix C.
\end{proof}
\subsection{Proposed Algorithm and Its Convergence Analysis}\label{subsec:Proposed Algorithm and Its Convergence Analysis}
\noindent We propose \cref{algorithm} based on the gradient descent method of the smoothed surrogate function $F_k:=F^{\abra{\mu_k}}$.

We design $\seq{\mu}{k}\subset(0,(2\eta)^{-1}]$ to satisfy the following condition (introduced in \cite{kume2024variableLong}) so as to establish a convergence analysis of \cref{algorithm}:
\begin{equation}\label{eq:conditions of mu}
  \left\{
    \begin{aligned}
      &\text{(i) }\textstyle\lim_{k\toinf}\mu_k = 0, \quad \text{(ii) }\textstyle\sum\nolimits_{k=1}^{\infty}\mu_k = \infty,\\
      &\text{(iii) }(\forall k \in \N)\ \mu_k\ge\mu_{k+1}.
    \end{aligned}
  \right.
\end{equation}
For example, $\mu_k:=(2\eta)^{-1}k^{-\inv{\alpha}}$ with $\alpha\ge1$ enjoys the condition \cref{eq:conditions of mu}
($\alpha=3$ is reported to be an appropriate value for a reasonable convergence rate of a special case of \cref{algorithm} with $g\equiv 0$ \cite{bohm2021variable,kume2024variableLong}).
\begin{algorithm}[t]
  \caption{Variable smoothing algorithm for\\ DC composite type problem (\cref{problem})}
  \label{algorithm}
  \begin{algorithmic}[1]
      \Require $\x_1 \in \R^d,\seq{\mu}{k}\subset(0,(2\eta)^{-1}]$ enjoying \cref{eq:conditions of mu}.
      \For{$k=1,2,3,\dots$}
      \State Set $F_k:= F^{\abra{\mu_k}}=\rbra*{\mor{f}{\mu_k} - \mor{g}{\mu_k}} \circ \mathfrak{S}$
      \State Obtain $\gamma_k$ by \cref{backtracking}
      \State $\x_{k+1} \leftarrow \x_k - \gamma_k\nb F_k(\x_k)$
      \EndFor
  \end{algorithmic}
\end{algorithm}
\begin{algorithm}[t]
  \caption{Backtracking algorithm to find $\gamma_k$}
  \label{backtracking}
  \begin{algorithmic}[1]
    \Require $\gamma_{\text{init},k} \in \R_{++}$ (see \cref{exa:initial guess satisfying assumption} for its choices),
    \Statex \hspace{\algorithmicindent} $\rho\in(0,1),\ c\in(0,1)$
    \State $\tilde{\gamma} \leftarrow \gamma_{\text{init},k}$
    \While{the condition \cref{eq:armijo condition} with $\gamma:=\tilde{\gamma}$ is false}
    \State $\tilde{\gamma} \leftarrow \rho\tilde{\gamma}$
    \EndWhile
    \Ensure $\gamma_k:=\tilde{\gamma}$
  \end{algorithmic}
\end{algorithm}

We employ \cref{backtracking} in order to obtain a stepsize $\gamma_k$ enjoying the following \textit{Armijo condition} with $\gamma:=\gamma_k$:
\begin{equation}\label{eq:armijo condition}
  F_k(\x_k-\gamma \nb F_k(\x_k)) \le F_k(\x_k) - c\gamma\norm{\nb F_k(\x_k)}^2.
\end{equation}
\cref{backtracking} is called the \textit{backtracking algorithm}, and it has been utilized as a standard stepsize selection for smooth optimization (see, e.g., \cite{andrei2020nonlinear}).
The while-loop in \cref{backtracking} terminates after a finite number of iterations as follows.
\begin{lemma}[Properties of \cref{backtracking}]\label{lem:finite termination of backtracking}
  Consider \cref{problem} under \cref{asm:descent assumption}, and \cref{algorithm} with arbitrary inputs $\x_1\in\R^d$ and $\seq{\mu}{k}\subset (0,(2\eta)^{-1}]$.
  With any inputs $(\gamma_{\text{init},k},\rho,c)\in \R_{++}\times(0,1)\times(0,1)$, \cref{backtracking} for estimating $\gamma_k$ satisfies the following properties:
  \begin{enumerate}
    \item \cref{backtracking} outputs a stepsize $\gamma_k$ enjoying the Armijo condition \cref{eq:armijo condition} with $\gamma:= \gamma_{k}$ and
    \begin{equation}\label{eq:lower bound of stepsize}
      \gamma_k \ge \min \cbra*{\gamma_{\text{init},k},2(1-c)\kappa_{\mu_k}^{-1}\rho}
    \end{equation}
    (see \cref{asm:descent assumption} for the definition of $\kappa_{\mu_k}$).
    \item The while-loop in \cref{backtracking} is guaranteed to terminate after at most $\max\cbra[\big]{0,\big\lceil \log_{\rho}\rbra[\big]{2(1-c)\kappa_{\mu_k}^{-1}\gamma_{\text{init},k}^{-1}} \big\rceil}$ iterations,
     where $\lceil \cdot \rceil$ denotes the ceiling function.
  \end{enumerate}
\end{lemma}
\begin{proof}
  For any $\gamma\in(0,2(1-c)\kappa_{\mu_k}^{-1})$, it follows from \cref{eq:descent assumption} with $(\x,\y,\mu):=(\x_k,\x_k-\gamma \nb F_k(\x_k),\mu_k)$ that
      $F_k(\x_k-\gamma \nb F_k(\x_k))
      \le F_k(\x_k) + \gamma\rbra*{2^{-1}\gamma\kappa_{\mu_k} -1} \norm{\nb F_k(\x_k)}^2
      \le F_k(\x_k) - c\gamma\norm{\nb F_k(\x_k)}^2$.
  Hence, \cref{backtracking} terminates when $\gamma_k$ becomes less than $2(1-c)\kappa_{\mu_k}^{-1}$ at the latest,
  which leads to both (a) and (b).
\end{proof}

For our convergence analysis, we choose initial guesses $(\gamma_{\text{init},k})_{k=1}^\infty$ of \cref{backtracking} enjoying the next assumption.
\begin{assumption}[Condition for initial guesses $(\gamma_{\text{init},k})_{k=1}^\infty$]\label{asm:initial guess of backtracking}
  Consider \cref{problem} under \cref{asm:descent assumption}.
  For an input $\seq{\mu}{k}\subset (0,(2\eta)^{-1}]$ of \cref{algorithm} and initial guesses $(\gamma_{\text{init},k})_{k=1}^\infty \subset \R_{++}$ in \cref{backtracking}, the following holds:
  \begin{equation}\label{eq:assumption on initial guess of backtracking}
    (\exists \delta>0, \forall k\in\N)\quad \gamma_{\text{init},k} \ge \delta \kappa_{\mu_k}^{-1},
  \end{equation}
  where $\kappa_{\mu_k}$ is given in \cref{asm:descent assumption}.
  (Note: any knowledge on the value of $\delta$ is not required in \cref{algorithm,backtracking}.)
\end{assumption}

\begin{example}[$(\gamma_{\text{init},k})_{k=1}^\infty$ achieving \cref{asm:initial guess of backtracking}]\label{exa:initial guess satisfying assumption}
  \leavevmode
  The following choices of initial guesses $(\gamma_{\text{init},k})_{k=1}^\infty$ achieve \cref{asm:initial guess of backtracking} (see Appendix D for the proof).
  \begin{enumerate}
    \item (Constant multiple of $\kappa_{\mu_k}^{-1}$)
    We can choose $\gamma_{\text{init},k}:=2(1-c)\kappa_{\mu_k}^{-1}\ (k\in\N)$ 
    in a case where the value $\kappa_{\mu_{k}}$ is known (see \cref{pro:sufficient condition for descent assumption}).
    \cref{backtracking} with this $\gamma_{\text{init},k}$ does not execute the while-loop and outputs $\gamma_k:=\gamma_{\text{init},k}$
    because $\gamma_{\text{init},k}$ is already small enough to satisfy the Armijo condition \cref{eq:armijo condition}
    (see \cref{lem:finite termination of backtracking} (b)).
    \item (Constant value)
    We can also choose a constant value $\gamma_\text{init}\in\R_{++}$ for $\gamma_{\text{init},k}$.
    With this choice, the stepsize $\gamma_k$ is selected adaptively through the while-loop in \cref{backtracking}.
    In practice, the resulting stepsize tends to be larger than $2(1-c)\kappa_{\mu_k}^{-1}$.
    Consequently, compared with the choice $\gamma_{\text{init},k}:=2(1-c)\kappa_{\mu_k}^{-1}$ described in (a),
    this strategy may lead to faster convergence of \cref{algorithm}.
    \item (Stepsize used in previous iteration)
    We can also use $\gamma_{\text{init},k}:=\gamma_{k-1}$, that is, the stepsize used in the $(k-1)$-th iteration of \cref{algorithm}, where $\gamma_0\in\R_{++}$ is a given constant.
    With this choice, the while-loop in \cref{backtracking} may terminate in fewer iterations than in the case $\gamma_{\text{init},k}:=\gamma_{\text{init}}$ in (b).
    In our experiment in \cref{sec:experiment}, we adopted this initial guess because
    it empirically yields shorter convergence time of \cref{algorithm} than other choices of $\gamma_{\text{init},k}$ in (a) and (b).
    
  \end{enumerate}
\end{example}

Under \cref{asm:descent assumption,asm:initial guess of backtracking}, we present below a convergence theorem for \cref{algorithm}.
\begin{theorem}[Convergence theorem]\label{thm:convergence theorem}
  Consider \cref{problem} under \cref{asm:descent assumption}.
  Let $\seq{\mu}{k} \subset (0,(2\eta)^{-1}]$ and $(\gamma_{\text{init},k})_{k=1}^\infty \subset \R_{++}$ satisfy \cref{eq:conditions of mu} and \cref{asm:initial guess of backtracking},
  while the remaining inputs $(\x_1,\rho, c)\in\R^d\times (0,1)\times(0,1)$ of \cref{algorithm,backtracking} are arbitrarily chosen.
  For the function sequence $\seq{F}{k}$ and the point sequence $\seq{\x}{k}$ produced by \cref{algorithm}, the following hold:
  \begin{enumerate}
    \item
    For any $\underline{k},\bar{k}\in \N$ such that $\underline{k}\le\bar{k}$, we have
    \begin{equation}
      \min_{\underline{k}\le k \le \bar{k}}\norm{\nb F_k(\x_k)} \le \sqrt{\frac{C}{\sum_{k=\underline{k}}^{\bar{k}}\mu_k}},\label{eq:lemma for convergence analysis}
    \end{equation}
    where $C\in\R_{++}$ is a constant.
    \item
    \begin{equation}\label{eq:convergence theorem}
      \liminf_{k \to \infty}\norm{\nb F_k(\x_k)} = 0. 
    \end{equation}
    \item
    We can choose a subsequence $(\x_{m(l)})_{l=1}^{\infty}$ such that $\lim_{l \to \infty}\norm{\nb F_{m(l)}(\x_{m(l)})} = 0$,
    where $\map{m}{\N}{\N}$ is monotonically increasing. %
    Moreover, every cluster point of $(\x_{m(l)})_{l=1}^{\infty}$ is a DC composite critical point of \cref{problem}.
  \end{enumerate}
\end{theorem}
\begin{proof}
  See Appendix E.
\end{proof}

%% file: experiment.tex
\section{Numerical experiment}\label{sec:experiment}
\noindent We conducted numerical experiments in order to evaluate estimation performance of our robust phase retrieval method based on the proposed model \cref{eq:model with nonconvex} and \cref{algorithm}.
In our experiments, we adopted the capped $\ell_1$ norm and the trimmed $\ell_1$ norm as the DC loss function $\varphi$ in the model \cref{eq:model with nonconvex} (see \cref{tab:options of phi} for the expressions of $\varphi$), where several choices of parameters $\beta$ and $K$ were tested.
In \cref{algorithm}, we used $\mu_k:=k^{-1/3}\ (k\in\N)$ for the parameters of the Moreau envelope.
To compute the stepsize $\gamma_k$ via \cref{backtracking}, we used $(\rho,c):=(0.8,0.0001)$, which is a typical choice in smooth optimization \cite{andrei2020nonlinear}.
As stated in \cref{exa:initial guess satisfying assumption} (c), we employed $\gamma_{\text{init},k}:=\gamma_{k-1}\ (k\in\N)$, where $\gamma_0$ was set to $\max\{1,\norm{\nb F_1(\x_1)}^{-1}\}$.

For comparison, we also employed state-of-the-art existing methods \cite{zheng2024new,kim2024robust} based on the $\ell_1$ loss function.
The method in \cite{zheng2024new} applies the \textit{inexact proximal linear (IPL) algorithm} to the $\ell_1$ loss-based model \cref{eq:model with l1}.
In IPL, the $(k+1)$-th estimate $\x_{k+1}\in\R^d$ is obtained as an inexact solution of a subproblem 
\begin{equation}\label{eq:subproblem for ipl}
    \eqsize{0.95}{\minimize{\x}{\R^d}} \eqsize{0.92}{\norm{\frakS_{\text{RPR}}(\x_k) + \D \frakS_{\text{RPR}}(\x_k)[\x-\x_k]}_1 + \inv{2 t}\norm{\x-\x_k}^2,}
\end{equation}
which is derived by linearizing $\frakS_{\text{RPR}}$ in \cref{eq:model with l1} at $\x_k$ and adding a quadratic term with $t\in\R_{++}$.
This subproblem is solved by \textit{fast iterative shrinkage-thresholding algorithm (FISTA)} with two possible stopping criteria named \textit{(LACC)} and \textit{(HACC)}  (see \cite[Alg. 2]{zheng2024new}).
In our experiments, we adopted (HACC) because \cite[Fig. 2]{zheng2024new} demonstrates that IPL with (HACC) empirically yields a slightly higher \textit{success rate} than IPL with (LACC) (for the definition of success rate, see the sentence just after \cref{eq:definition of success}).
To implement IPL, we used the code released by the author.\noteD
The other competing method \cite{kim2024robust} applies \textit{robust alternating minimization (Robust-AM)} to the modified $\ell_1$ loss-based model \cref{eq:modified model with l1}.
Robust-AM is derived as a Gauss-Newton method for \cref{eq:modified model with l1}, where its estimate sequence is iteratively updated by
\begin{equation}\label{eq:subproblem for robust-am}
    \x_{k+1}\in \argmin{\x\in\R^d} \sum_{i=1}^{n} \abs*{\abra{\ba_i,\x}-\sign(\abra{\ba_i,\x_k})\sqrt{|[\bb]_i|}}.
\end{equation}
For solving this subproblem, two methods, \textit{ADMM-LAD} and \textit{ADMM-LP}, are introduced in \cite{kim2024robust}.
Robust-AM with ADMM-LAD is reported to achieve almost the same success rate as one with ADMM-LP while exhibiting significantly faster convergence (see \cite[Fig. 2]{kim2024robust}).
Hence, we employed ADMM-LAD in our experiments, as in the main experiments in \cite{kim2024robust}. 
We implemented Robust-AM by using the code provided by the author.\noteE

For the initial point of \cref{algorithm}, IPL, and Robust-AM, we generated common $\x_1\in\R^d$ by an existing initialization method \cite[Alg. 3]{duchi2019solving}, which was also used in the experiment in \cite{zheng2024new}.
(This initialization is also mentioned in \cite{kim2024robust} as being consistent with its convergence analysis \cite[Thm. 4.1]{kim2024robust}.)
We terminated each algorithm when one of the following conditions was met: (i) the relative change in the the cost function value satisfied $\frac{\abs{\Phi_j(\x_{k})-\Phi_j(\x_{k-1})}}{|\Phi_j(\x_{k-1})|}<10^{-7}\ (j=1,2,3)$
where $\x_k$ is the $k$-th estimate generated by each algorithm, and $\Phi_1$, $\Phi_2$ and $\Phi_3$ is the cost functions in \cref{eq:model with l1}, \cref{eq:modified model with l1}, and  \cref{eq:model with nonconvex}, respectively, used in IPL, Robust-AM, and  \cref{algorithm};
(ii) the number of iterations reached to 10000;
(iii) the running CPU time exceeded 30 seconds.
All experiments were performed by MATLAB on MacBook Pro (Apple M3, 16GB).

The problem settings, partially inspired by \cite{zheng2024new}, are as follows.
We drew each entry of $A\in\R^{n\times d}$ from the normal distribution $\mathcal{N}(0,1)$.
Each entry of the target signal $\x^\star \in \R^d$ was chosen from $1$ or $-1$ with a probability of $0.5$ respectively.
The additive noise $\varepsilon_i\in\R$ was generated by $\mathcal{N}(0,10^{-6})$.
The index set $\mathcal{I}_{\text{out}}$ was selected uniformly at random with fixed cardinality $\#\mathcal{I}_{\text{out}}$.
Here, let $p_{\text{fail}}:=\#\mathcal{I}_{\text{out}}/n$ denote the proportion of outliers.
Each outlier value $\xi_i\in\R\ (i\in\mathcal{I}_{\text{out}})$ was generated from (i) the (absolute) Cauchy distribution or (ii) the uniform distribution.
More specifically, 
each $\xi_i$ was given by (i) $\xi_i := s \mathcal{M} \tan\rbra{0.5\pi u_i}$ or (ii) $\xi_i := s \mathcal{M} u_i$, where $u_i$ was drawn from the uniform distribution of $[0,1]$,
$\mathcal{M}:= \max_{1\le i \le n}\abra{\ba_i,\x^\star}^2$ was a constant, and a parameter $s\in\R_{++}$ was used to control the scale of outliers.
(Since the results were similar for both types of outliers, 
we present only the results for Cauchy outliers in this section, while those for uniformly distributed outliers are provided in the supplementary material.)
For every fixed $d,n,\#\mathcal{I}_{\text{out}}$, and $s$, we performed estimation on 50 random problem instances obtained by varying $\x^\star, A, \varepsilon_i,u_i$, and elements of $\mathcal{I}_{\text{out}}$.
We judged that an estimation succeeded if the relative error at the final estimate $\x^\diamond \in \R^d$ achieved
\begin{equation}\label{eq:definition of success}
   \min\{\|\x^\star - \x^\diamond\|, \|\x^\star + \x^\diamond\|\}/\|\x^\star\| < 10^{-3}.
\end{equation}
As in \cite{zheng2024new}, we used an estimation performance criterion called “success rate” that is the percentage of the successful estimation out of 50 estimations.

\begin{figure*}[t]
\centering

\begin{subfigure}{0.23\textwidth}
\includegraphics[width=\linewidth]{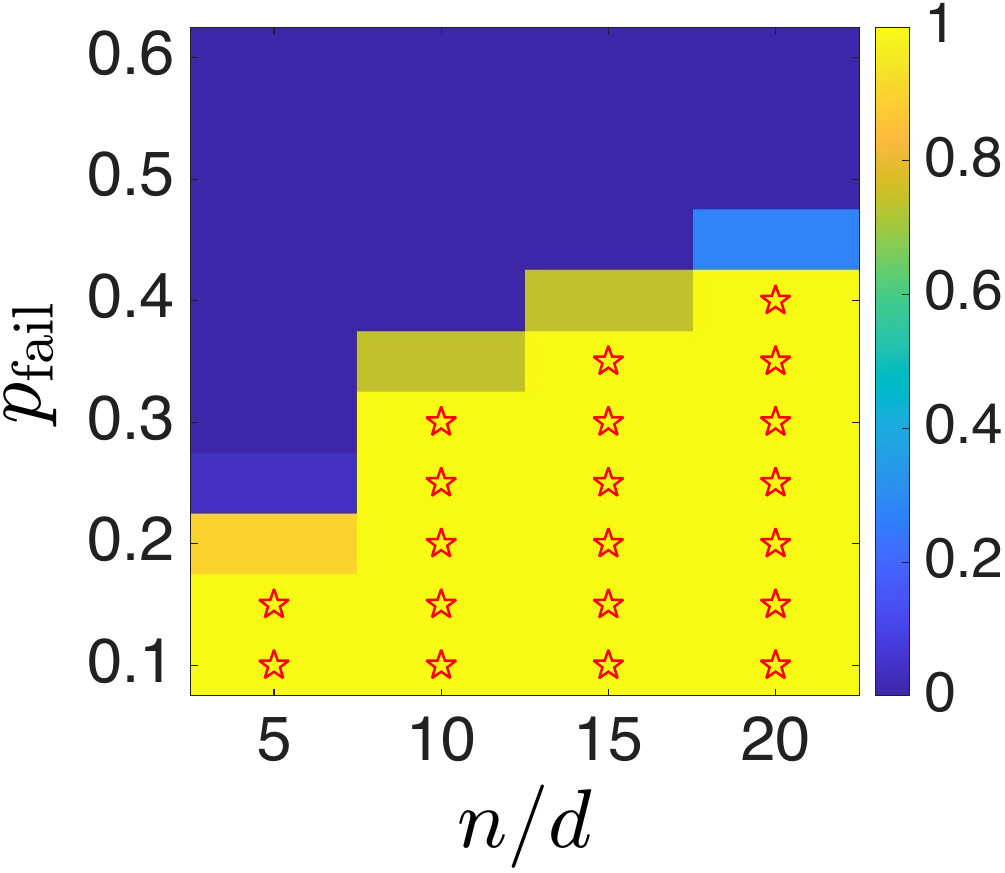}
\caption{$\ell_1$ by IPL}
\end{subfigure}
\begin{subfigure}{0.23\textwidth}
\includegraphics[width=\linewidth]{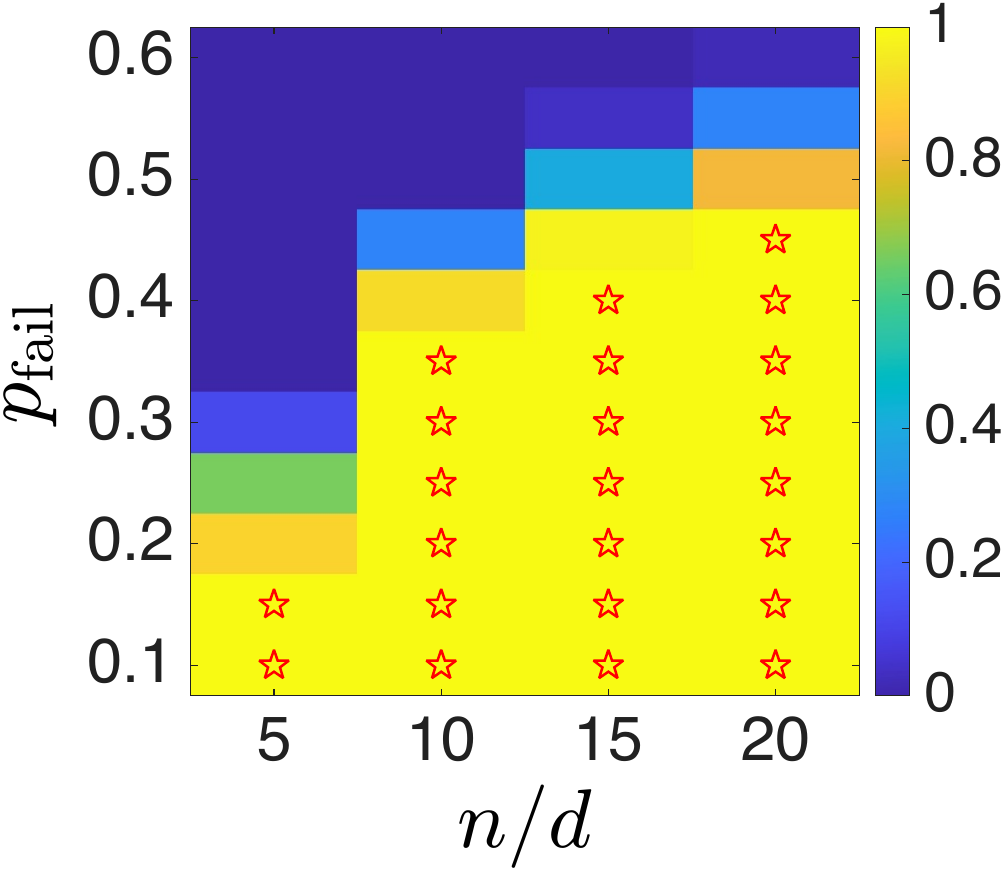}
\caption{Capped $\ell_1$ ($\beta=100$)}
\end{subfigure}
\begin{subfigure}{0.23\textwidth}
\includegraphics[width=\linewidth]{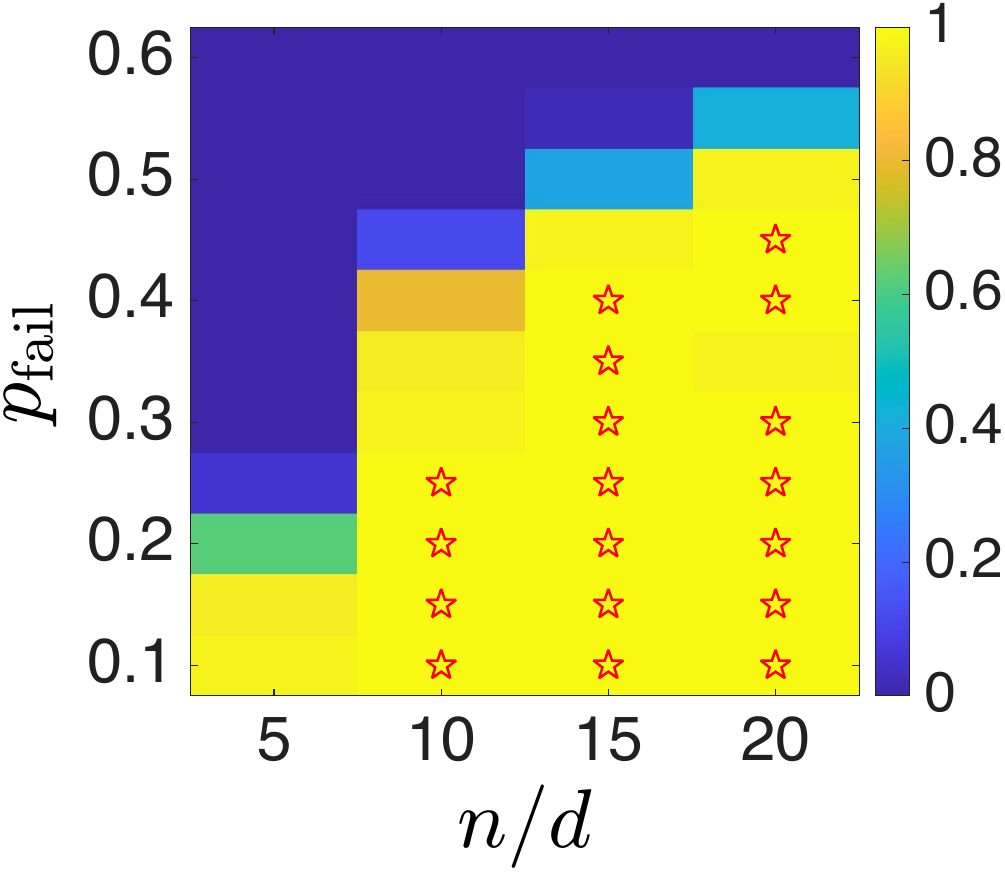}
\caption{Capped $\ell_1$ ($\beta=1000$)}
\end{subfigure}
\begin{subfigure}{0.23\textwidth}
\includegraphics[width=\linewidth]{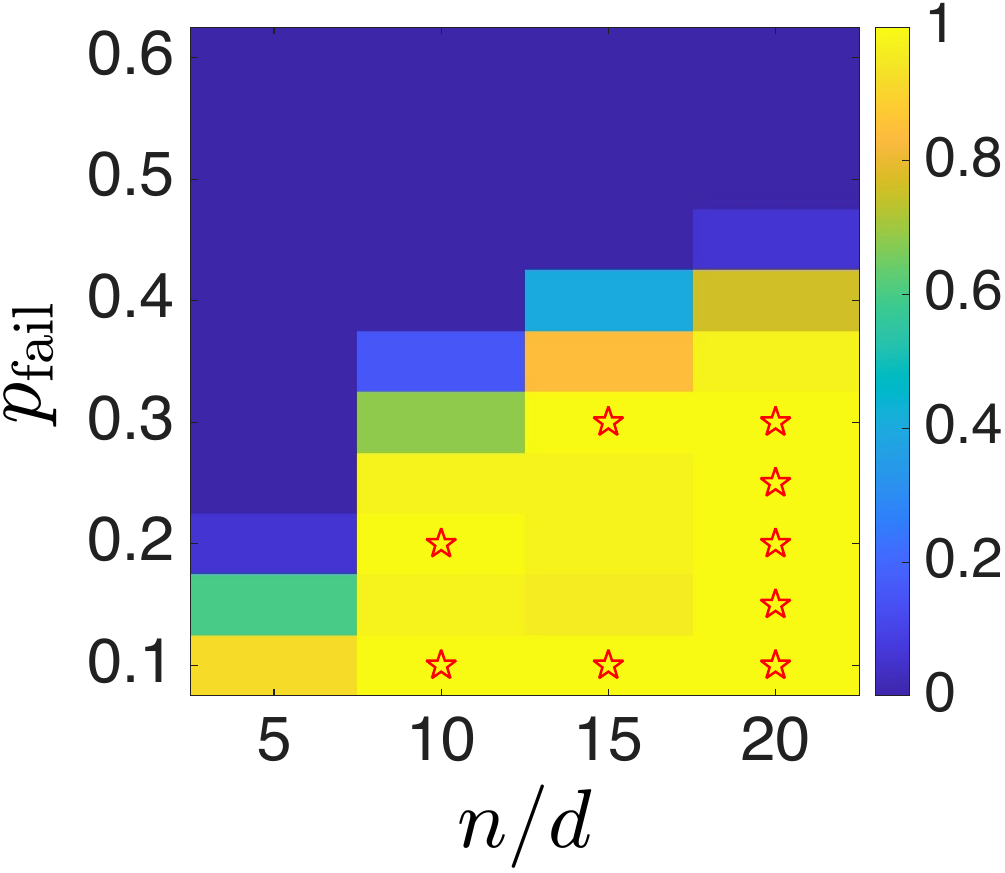}
\caption{Capped $\ell_1$ ($\beta=10000$)}
\end{subfigure}

\begin{subfigure}{0.23\textwidth}
\includegraphics[width=\linewidth]{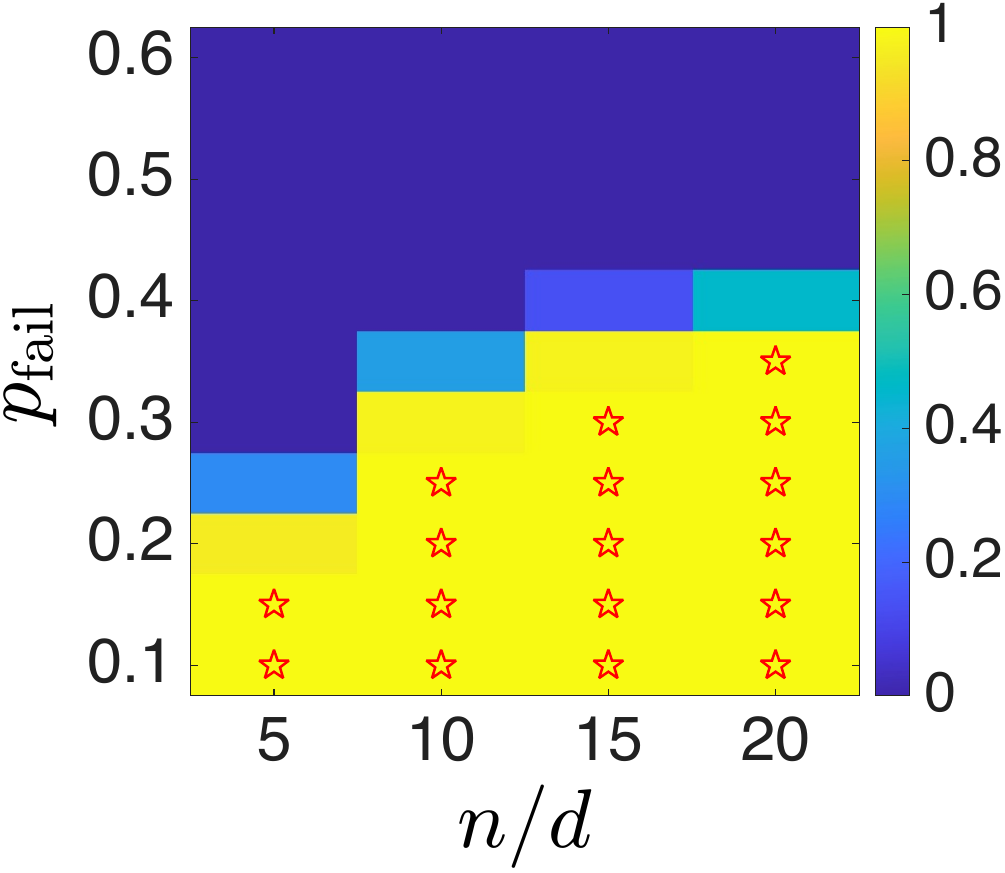}
\caption{$\ell_1$ by Robust-AM}
\end{subfigure} 
\begin{subfigure}{0.23\textwidth}
\includegraphics[width=\linewidth]{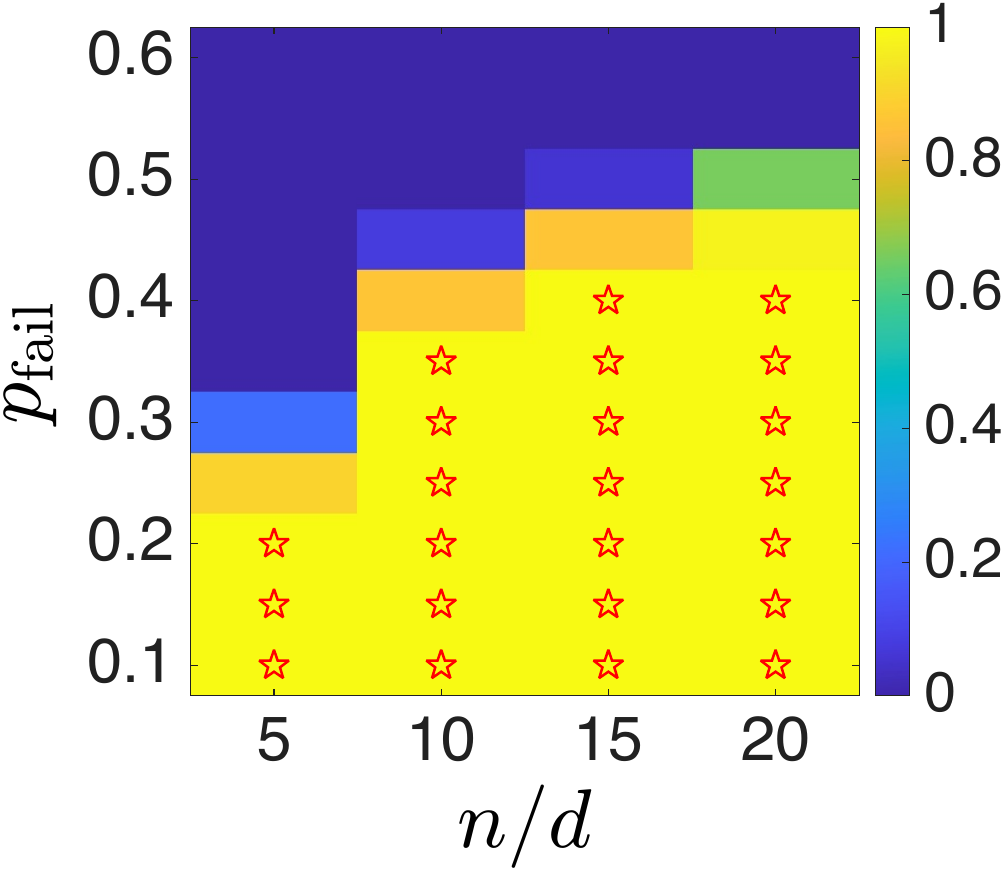}
\caption{Trimmed $\ell_1$ ($K/n=0.2$)}
\end{subfigure}
\begin{subfigure}{0.23\textwidth}
\includegraphics[width=\linewidth]{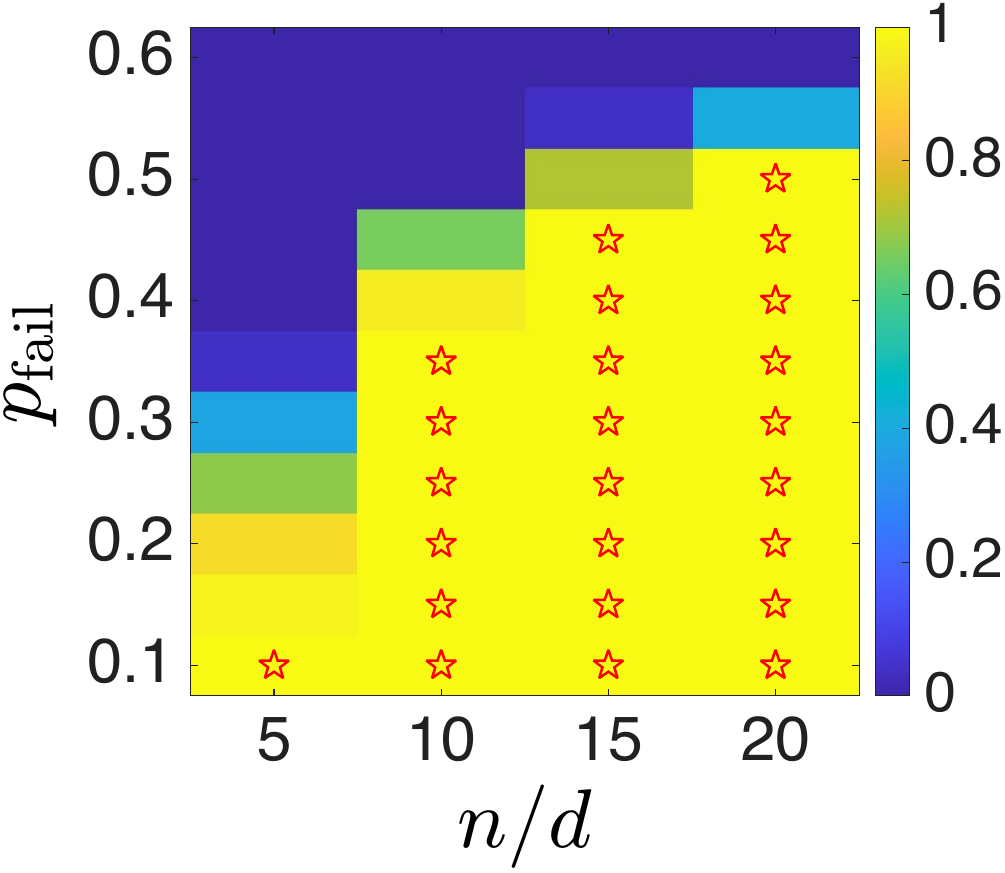}
\caption{Trimmed $\ell_1$ ($K/n=0.3$)}
\end{subfigure}
\begin{subfigure}{0.23\textwidth}
\includegraphics[width=\linewidth]{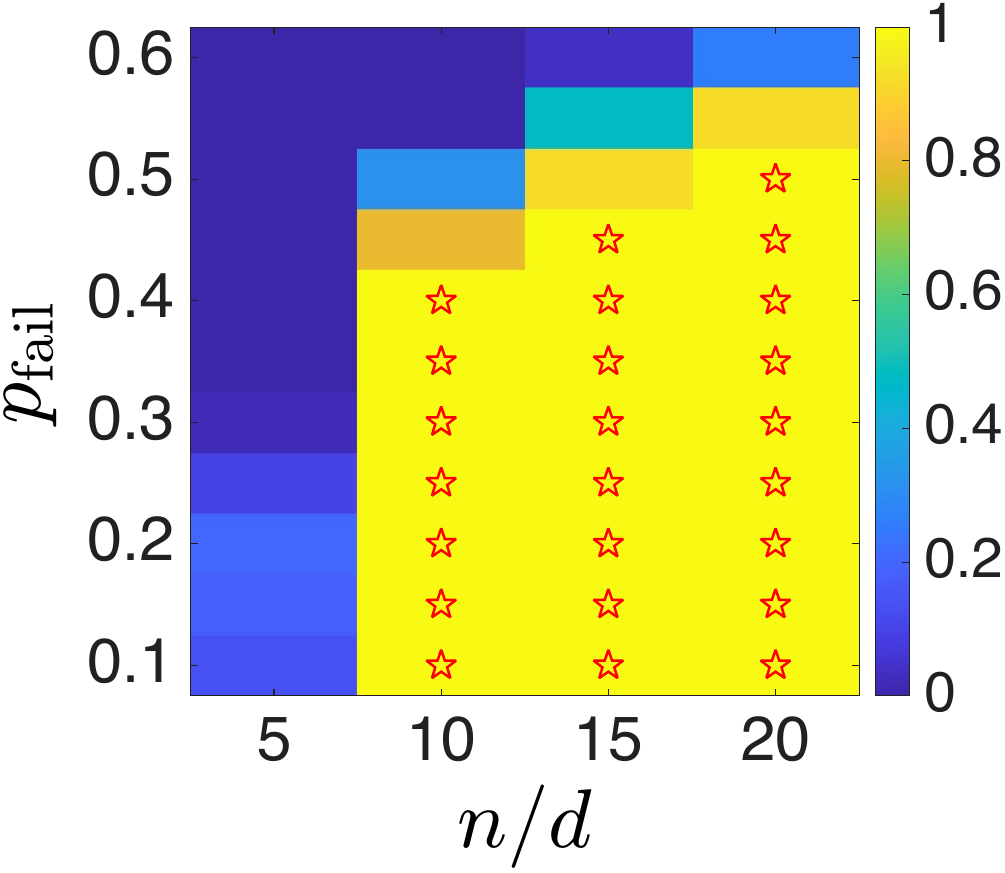}
\caption{Trimmed $\ell_1$ ($K/n=0.4$)}
\end{subfigure}

\caption{Success rates of IPL (a), Robust-AM (e), the proposed method with the capped $\ell_1$ norm (b)-(d) and the trimmed $\ell_1$ norm (f)-(h).
The color of each pixel represents success rate, and red pentagrams stand for success rate of 1.
The experiment was conducted with $d = 100$, $s=1$, and $\xi_i$ drawn from the Cauchy distribution.
}
\label{fig:results_d_100_Cauchy}
\end{figure*}
\begin{figure*}[t]
\centering
\begin{subfigure}{0.23\textwidth}
\includegraphics[width=\linewidth]{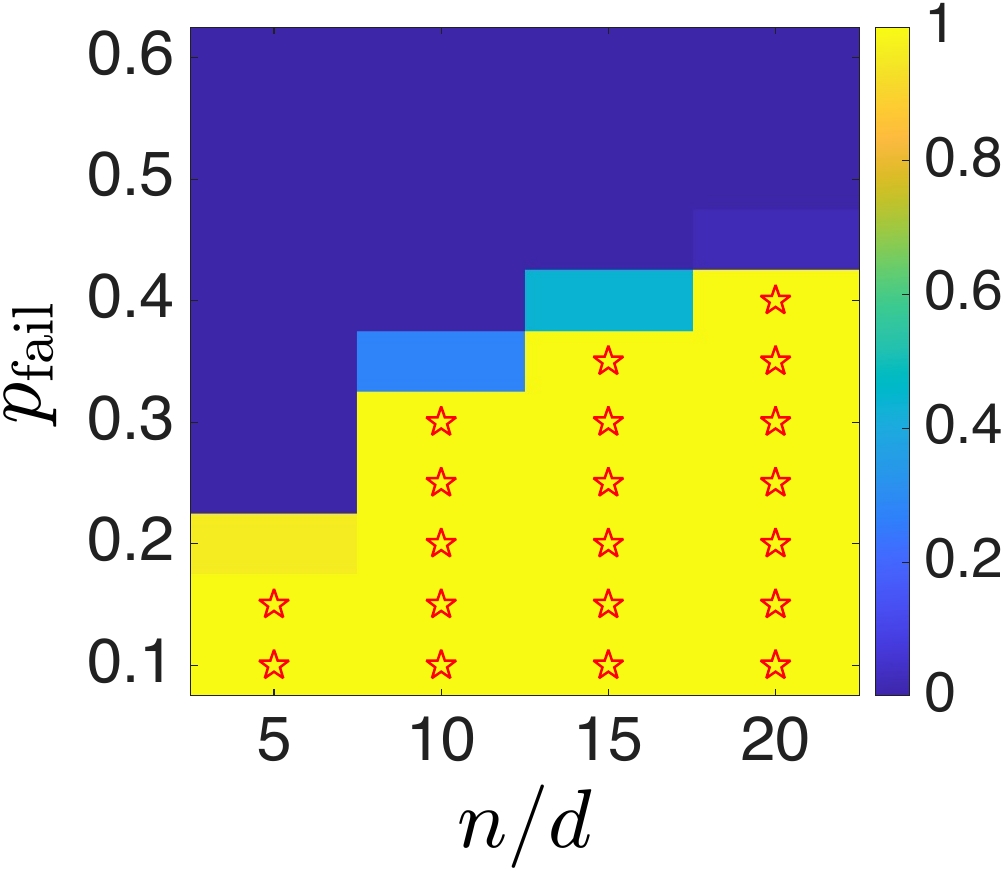}
\caption{$\ell_1$ by IPL}
\end{subfigure}
\begin{subfigure}{0.23\textwidth}
\includegraphics[width=\linewidth]{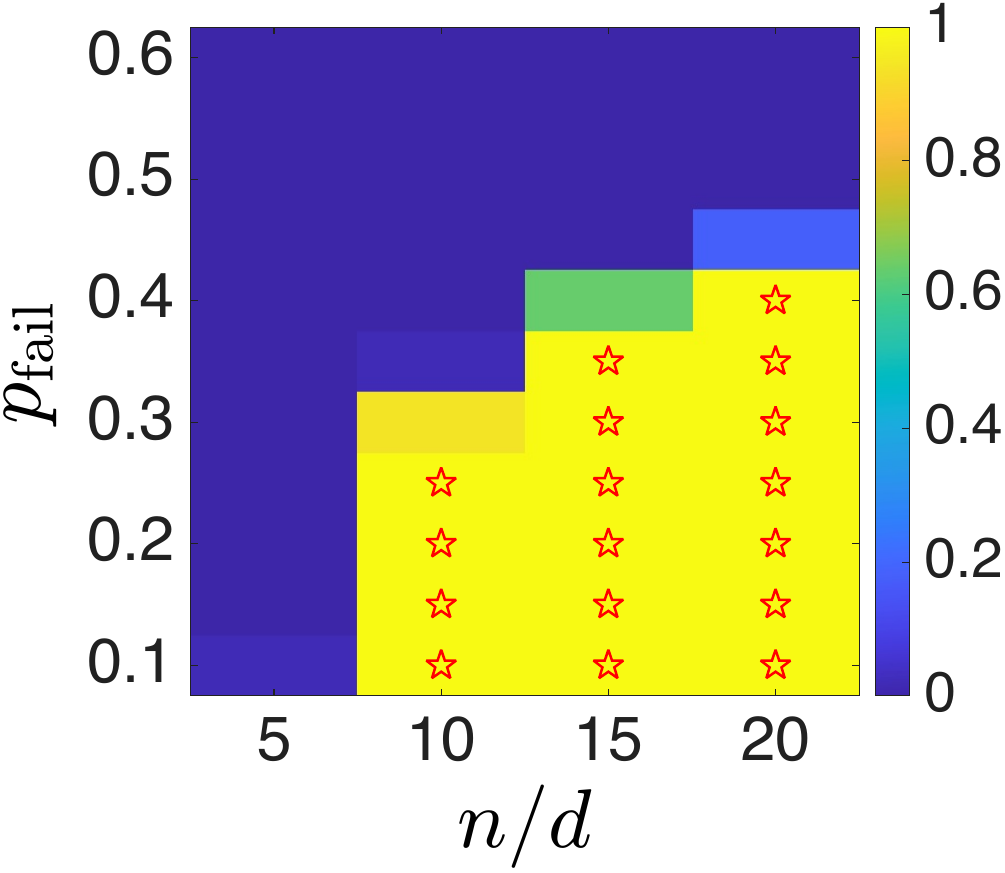}
\caption{Capped $\ell_1$ ($\beta=100$)}
\end{subfigure}
\begin{subfigure}{0.23\textwidth}
\includegraphics[width=\linewidth]{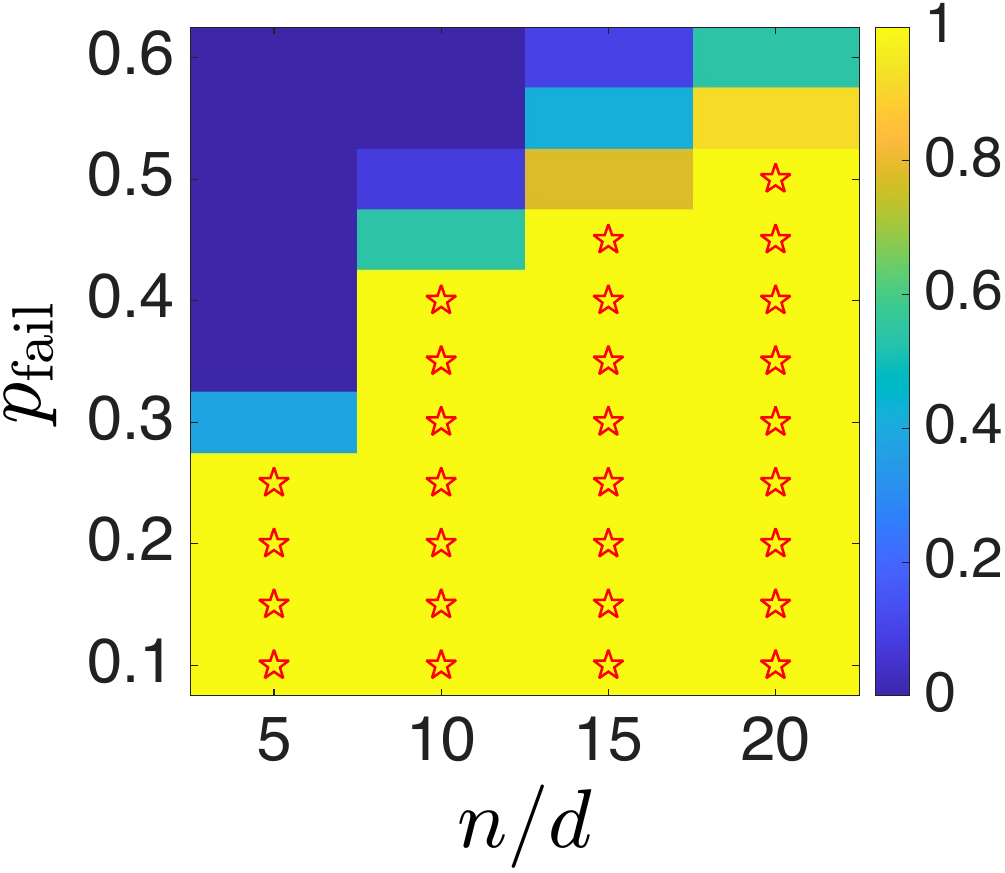}
\caption{Capped $\ell_1$ ($\beta=1000$)}
\end{subfigure}
\begin{subfigure}{0.23\textwidth}
\includegraphics[width=\linewidth]{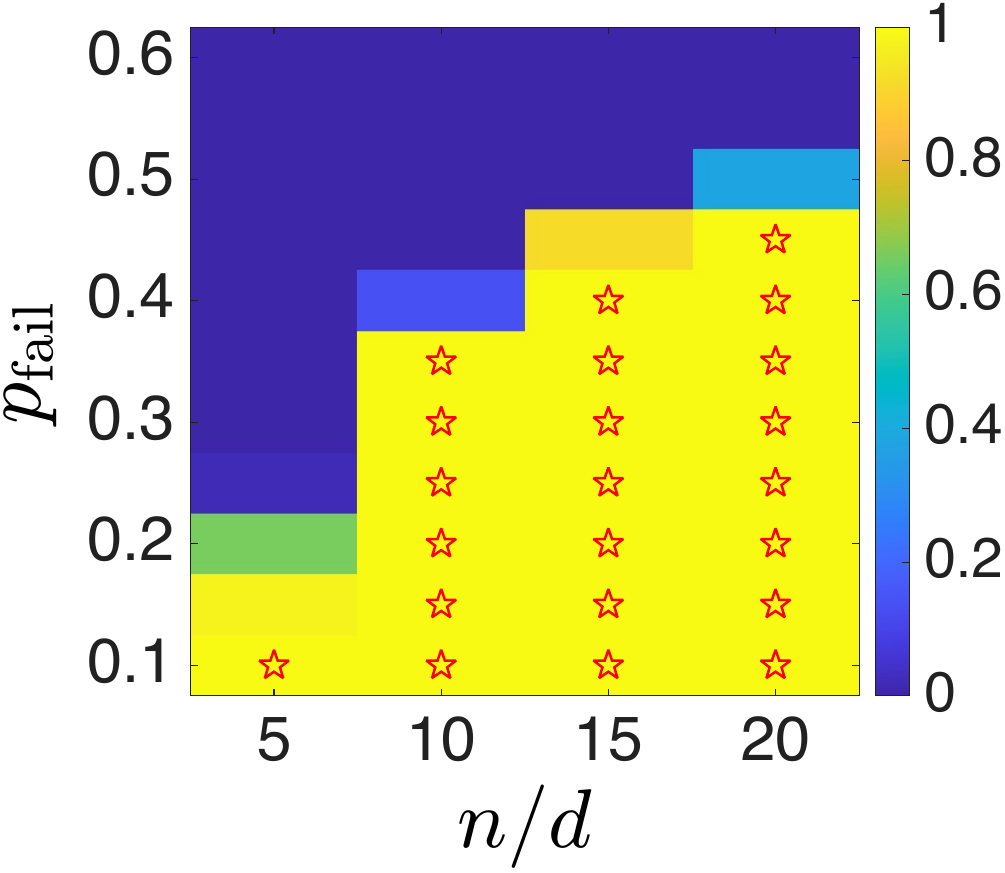}
\caption{Capped $\ell_1$ ($\beta=10000$)}
\end{subfigure}

\hspace{0.23\textwidth}
\begin{subfigure}{0.23\textwidth}
\includegraphics[width=\linewidth]{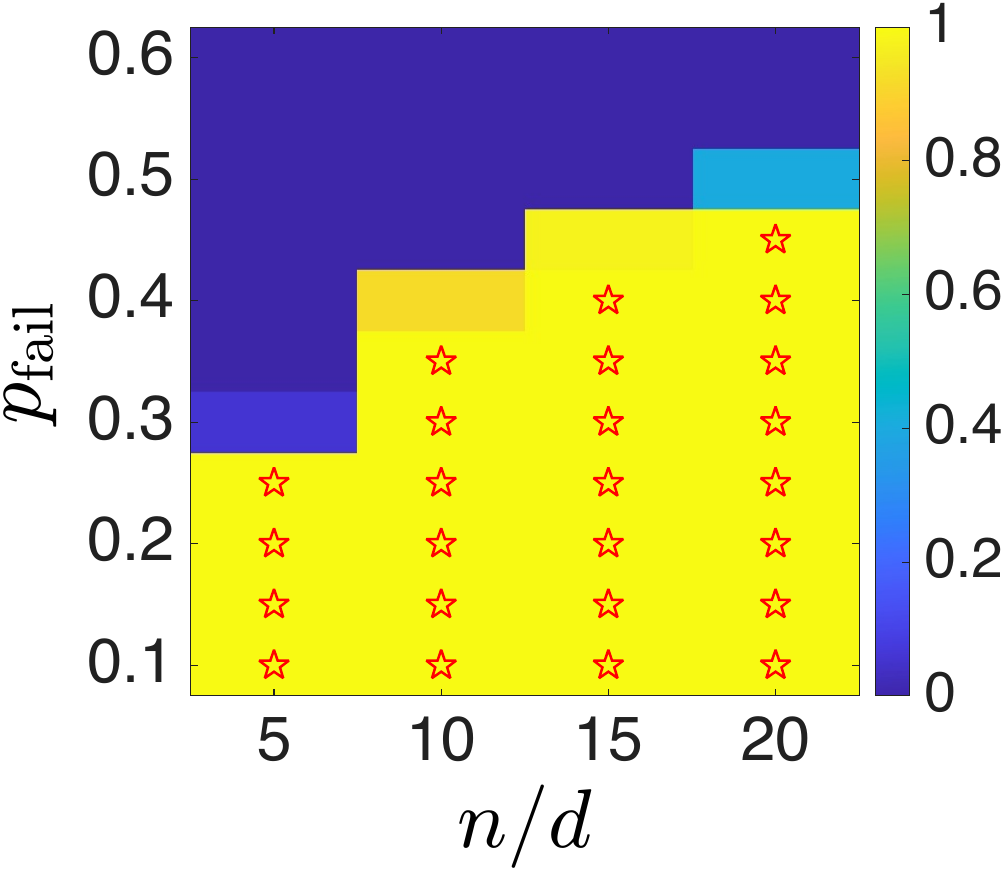}
\caption{Trimmed $\ell_1$ ($K/n=0.2$)}
\end{subfigure}
\begin{subfigure}{0.23\textwidth}
\includegraphics[width=\linewidth]{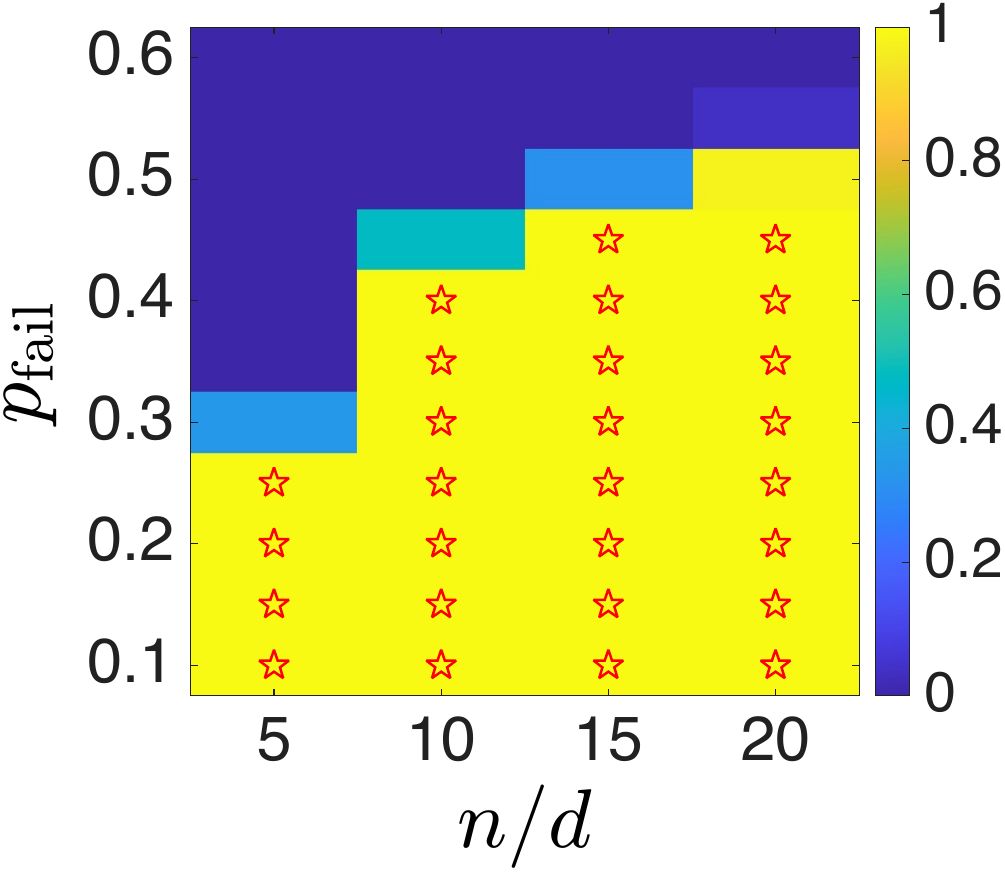}
\caption{Trimmed $\ell_1$ ($K/n=0.3$)}
\end{subfigure}
\begin{subfigure}{0.23\textwidth}
\includegraphics[width=\linewidth]{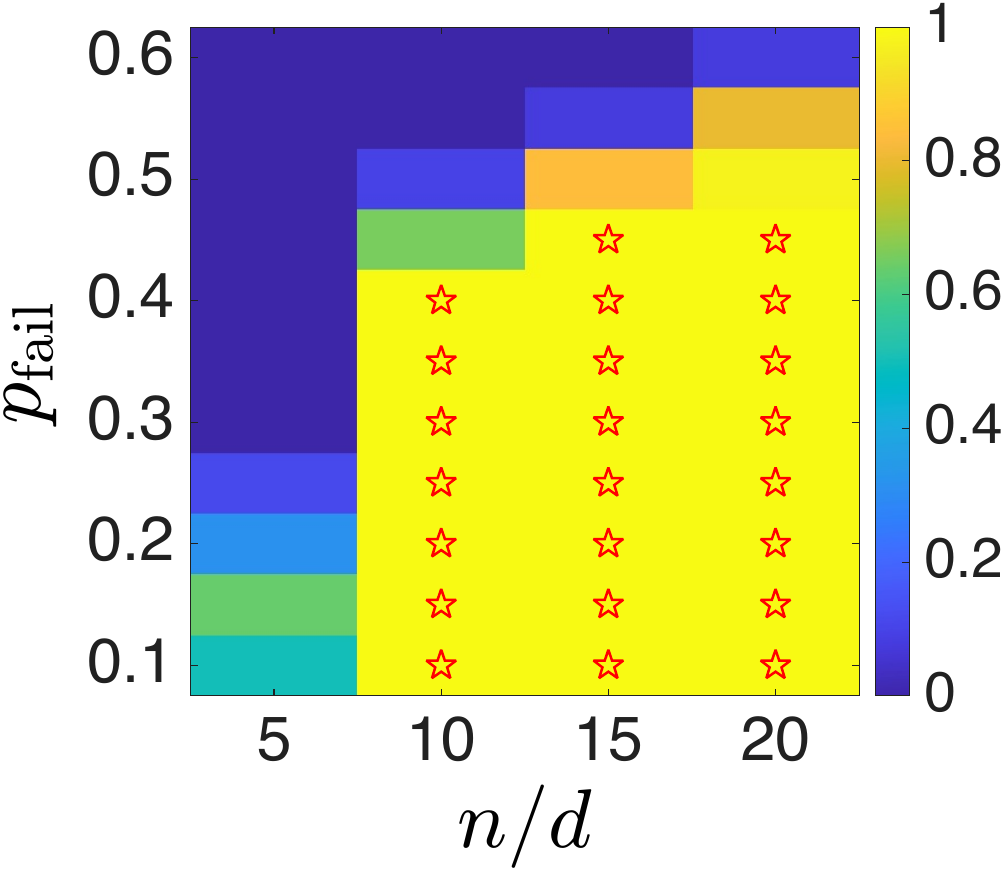}
\caption{Trimmed $\ell_1$ ($K/n=0.4$)}
\end{subfigure}
\caption{Same as \cref{fig:results_d_100_Cauchy}, except that $d=500$.
We omitted the result of Robust-AM because the subproblem solver (ADMM-LAD) did not reach its stopping criterion within 30 seconds even when $p_{\text{fail}}=0.1$.}
\label{fig:results_d_500_Cauchy}
\end{figure*}

In the first experiment, we evaluated the success rate of each method for $n/d \in \{5,10,15,20\}$ and $p_{\text{fail}} \in \Set{0.1+0.05 j}{j\in \{0,1,...,10\}}$.
According to the similar experiments in \cite{kim2024robust,zheng2024new}, we employed $d \in \{100, 500\}$.
We fixed the parameter $s$ at $1$.

\cref{fig:results_d_100_Cauchy} and \cref{fig:results_d_500_Cauchy} show the results for the case $d=100$ and for the case $d=500$, respectively.
In these figures, pixels with 100\% success rate are marked by red pentagrams, and are hereafter referred to as \textit{successful pixels}.
\cref{fig:results_d_100_Cauchy} and \cref{fig:results_d_500_Cauchy} imply that the proposed method with the capped $\ell_1$ norm and the trimmed $\ell_1$ norm tends to achieve relatively high success rate especially in upper region of the figures where $p_{\text{fail}}$ is large.
More precisely, except for the capped $\ell_1$ with $(d,\beta)=(100,10000)$ and $(500,100)$, all the proposed methods have more successful pixels than the existing methods in the region $p_{\text{fail}}\ge 0.35$.

In particular, we observe that the capped $\ell_1$ with properly chosen $\beta$ allows for more successful estimations in a wider region than existing methods.
Specifically, in the case $(d,\beta)=(100,100)$ (resp. $(500,1000)$), 
the set of successful pixels for the capped $\ell_1$ contains those for the existing methods and 3 (resp. 8) additional pixels.
On the other hand, we found that the trimmed $\ell_1$ yields a larger number of successful pixels than existing methods for all tested values of $K/n$.\noteF
These results above are consistent with the intuition in \cref{rma:Robustness of nonconvex DC functions}, where DC loss functions are expected to be robust against outliers.

\begin{figure}[t]
\centering
\includegraphics[width=\columnwidth]{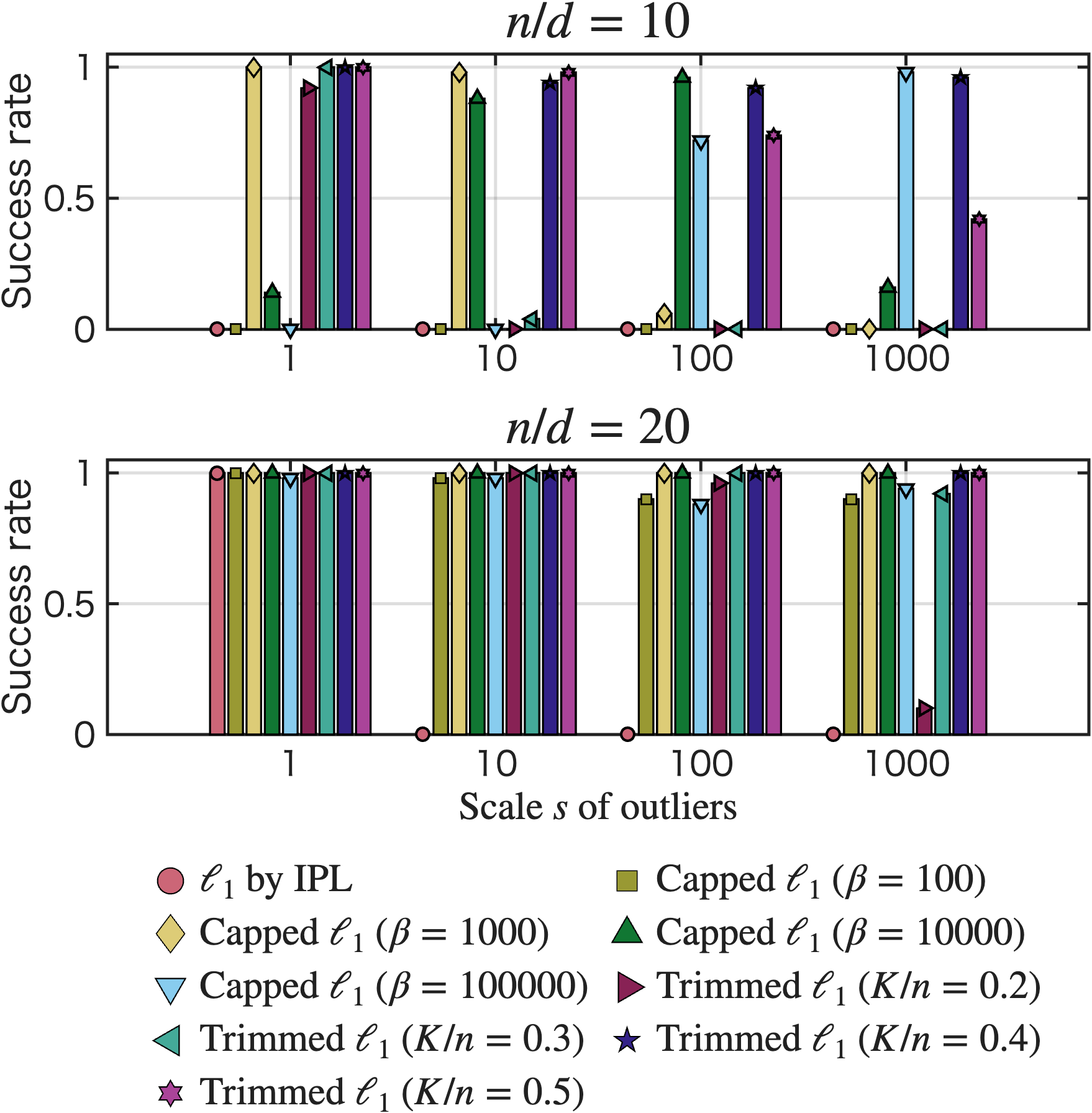}
\caption{Success rate for different values of $s$. $(d,p_{\text{fail}})$ were fixed at $(500,0.4)$.}
\label{fig:different scales of outliers}
\end{figure}

To examine how appropriate choices of $\beta$ and $K$ change with different scales $s$ of outliers, we conducted the second experiment, where $s$ varied in $\{1,10,100,1000\}$.
We fixed $(d,p_{\text{fail}})$ at $(500,0.4)$ in this experiment.

\cref{fig:different scales of outliers}  shows the success rate of each method for the cases $n/d=10$ and $n/d=20$. %
When $n/d=10$, \cref{fig:different scales of outliers} demonstrates that the best performance of the capped $\ell_1$ is obtained with $\beta = 1000$ for $s\in\{1,10\}$, $\beta = 10000$ for $s=100$, and $\beta = 100000$ for $s=1000$, 
which indicates that larger values of $\beta$ are appropriate for large outliers.
In contrast, for the trimmed $\ell_1$, $K/n=0.4$ consistently yields high success rate across all values of $s$, which means that $K$ such that $K/n = p_{\text{fail}}$ is appropriate.
For the case $n/d=20$ where a sufficiently large number of measurements is available, the proposed method performs well across a relatively wide range of $\beta$ and $K$.
(The design of practical parameter selection strategies is beyond the scope of this paper and will be investigated in future work.)

\begin{table*}[t]
    \centering
    \caption{The running CPU time in seconds (outside the parentheses) and the number of iterations (inside the parentheses) for the case $(d,p_{\text{fail}},s)=(500,0.35,1)$.}
    \begin{tabular}{|l|c|c|c|c|}
        \hline
        $n/d$ & 5 & 10 & 15 & 20 \\
        \hline
        \hline
        $\ell_1$ by IPL & 30.00 (250.32) & 23.78 (172.02) & 7.61 (8.86) & 9.11 (6.28) \\
        Capped $\ell_1$ $(\beta = 100)$ & 0.84 (359.96) & 1.60 (362.76) & 1.29 (210.36) & 1.23 (161.14)\\
        Capped $\ell_1$ $(\beta = 1000)$ & 0.33 (135.76) & 0.48 (101.12) & 0.52 (77.48) & 0.51 (59.38) \\
        Capped $\ell_1$ $(\beta = 10000)$ & 0.28 (112.10) & 0.38 (78.04) & 0.37 (50.84) & 0.44 (49.82)\\
        Trimmed $\ell_1$ $(K/n = 0.2)$ & 0.47 (137.90) & 0.59 (88.80) & 0.66 (66.10) & 0.77 (62.56) \\
        Trimmed $\ell_1$ $(K/n = 0.3)$ & 0.92 (290.06) & 3.49 (619.48)  & 3.84 (472.74) & 3.55 (347.74)\\
        Trimmed $\ell_1$ $(K/n = 0.4)$ & 1.97 (664.14) & 2.46 (56.64) & 1.83 (44.92) & 5.36 (46.82)\\
        \hline
    \end{tabular}
    \label{tab:execution time}
\end{table*}
Lastly,  we present, in \cref{tab:execution time}, the running CPU time and the number of iterations of each methods for the representative case $(d,p_{\text{fail}},s)=(500,0.35,1)$.
\cref{tab:execution time} demonstrates that the proposed method consistently requires less running time than IPL.
A possible reason for this is that IPL requires solving subproblems \cref{eq:subproblem for ipl} in each iteration, resulting in a higher computational cost per iteration.
Indeed, \cref{tab:execution time} also shows that the proposed method has a substantially lower per-iteration time than IPL.

%% file: conclusion.tex
\section{Conclusion}
\noindent
We proposed the optimization model \cref{eq:model with nonconvex} with DC loss functions for robust phase retrieval.
For DC composite-type problem (\cref{problem}) including the proposed model, we designed a variable smoothing algorithm (\cref{algorithm}) with a convergence guarantee in terms of a DC composite critical point.
The proposed algorithm was designed to find a DC composite critical point by generating the sequence of points at which the gradient of the smooth surrogate function approaches zero.
Through numerical experiments, we demonstrated the robustness of the proposed model, investigated the relationship between the scale or number of outliers and the appropriate values of $\beta$ and $K$,
and showed the computational efficiency of the proposed algorithm.

\section*{Acknowledgement}
\noindent
We thank Mr.Kim (The Ohaio State University), the first author of \cite{kim2024robust}, for kindly sharing the code of Robust-AM.

%% file: appendix.tex
\appendices

\section{Proof of \cref{lem:Local optimality implies DC criticality}}
\noindent We show that a local minimizer $\x^\star$ of $F$ is a DC composite critical point of $F$.
\begin{proof}[Proof of \cref{lem:Local optimality implies DC criticality}]
    It suffices to prove the \underline{Claim 1:} $\del_{\mathrm{L}} (g\circ\frakS)(\x^\star)\ne\emptyset$ and the \underline{Claim 2:} $\del_{\mathrm{L}}(f\circ\frakS)(\x^\star) \supset \del_{\mathrm{L}}(g\circ\frakS)(\x^\star)$ because
  these two claims imply that $\del_{\mathrm{L}}(f\circ\frakS)(\x^\star)-\del_{\mathrm{L}}(g\circ\frakS)(\x^\star)\ni \bm{v} - \bm{v} = \bm{0}$ with some $\bm{v}\in \del_{\mathrm{L}} (g\circ\frakS)(\x^\star)$.
  
  \underline{\textbf{Claim 1:}}
  Since $\frakS$ is continuously differentiable, $\frakS$ is \textit{locally Lipschitz continuous (strictly continuous)} \cite[Thm. 9.7]{rockafellar2009variational}, i.e., for any $\x\in\R^d$, there exists an open neighborhood $\mathcal{N}(\x)\subset\R^d$ of $\x$ such that $\frakS$ is Lipschitz continuous on $\mathcal{N}(\x)$.
  Because compositions of locally Lipschitz continuous mappings are locally Lipschitz continuous \cite[Exe. 9.8 (c)]{rockafellar2009variational}, $g\circ\frakS$ is locally Lipschitz continuous.
  Thus, we obtain $\del_{\mathrm{L}} (g\circ\frakS)(\x^\star)\ne \emptyset$ by \cite[Thm. 9.13]{rockafellar2009variational}.
  
  \underline{\textbf{Claim 2:}}
  By applying the sum rule of Fre\'{c}het (regular) subdifferential \cite[Cor. 10.9]{rockafellar2009variational} to $f\circ \frakS = F + g\circ \frakS$, and using Fermat's rule $\del_{\mathrm{F}}F(\x^\star)\ni \bm{0}$ \cite[Thm. 10.1]{rockafellar2009variational}, we have
  \begin{align}
    \del_{\mathrm{F}}(f\circ\frakS)(\x^\star) &\supset \del_{\mathrm{F}}F(\x^\star) + \del_{\mathrm{F}}(g\circ\frakS)(\x^\star)\\
    &\supset \del_{\mathrm{F}}(g\circ\frakS)(\x^\star).
  \end{align}
    Then, \cref{fac:equivalence of limiting and frechet subdifferential} yields
    $\del_{\mathrm{L}}(f\circ\frakS)(\x^\star) \supset \del_{\mathrm{L}}(g\circ\frakS)(\x^\star)$.
\end{proof}

\section{Proof of \cref{thm:dc gradient sub-consistency}}
\noindent The following facts are required for the proof of \cref{thm:dc gradient sub-consistency}.
\begin{fac}[Sum rule of outer limit]\label{fac:sum rule of outer limit}
    For a point sequence $\seq{\bm{p}}{k}\subset \R^d$ and a bounded point sequence $\seq{\bm{q}}{k}\subset \R^d$,
    We have $\displaystyle \Limsup_{k\to\infty}(\bm{p}_k + \bm{q}_k) \subset \Limsup_{k\to\infty}\bm{p}_k + \Limsup_{k\to\infty}\bm{q}_k$.\noteG
\end{fac}
\begin{fac}[\fontsize{9.5pt}{10pt}\selectfont Boundedness of gradient sequences {\cite[Fact 2.4~(c)]{kume2025proximal}}]\label{fac:boundedness of gradient sequences}
    In the setting of \cref{thm:dc gradient sub-consistency},
    $\rbra*{\nb (\mor{f}{\mu_k}\circ\frakS)(\x_k)}_{k=1}^{\infty}$ and $\rbra*{\nb (\mor{g}{\mu_k}\circ\frakS)(\x_k)}_{k=1}^{\infty}$ are bounded sequences.
\end{fac}
\begin{fac}[Gradient sub-consistency {\cite[Thm.~4.4~(a)]{kume2024variableLong}}]\label{lem:gradient sub-consistency}
    In the setting of \cref{thm:dc gradient sub-consistency}, we have
    \begin{align}
        &\Limsup_{k\to\infty}\nb (\mor{f}{\mu_k}\circ\frakS)(\x_k) \subset \del_\mathrm{L}(f \circ \frakS)(\bar{\x}),\\
        &\Limsup_{k\to\infty}\nb (\mor{g}{\mu_k}\circ\frakS)(\x_k) \subset \del_\mathrm{L}(g \circ \frakS)(\bar{\x}).
    \end{align}
\end{fac}
\begin{proof}[Proof of \cref{thm:dc gradient sub-consistency}]
    (a) We can see from \cref{fac:boundedness of gradient sequences} that $\rbra*{-\nb (\mor{g}{\mu_k}\circ\frakS)(\x_k)}_{k=1}^{\infty}$ is bounded.
    Thus, we can employ \cref{fac:sum rule of outer limit} with $\bm{p}_k := \nb (\mor{f}{\mu_k}\circ\frakS)(\x_k),\ \bm{q}_k := - \nb (\mor{g}{\mu_k}\circ\frakS)(\x_k)$ to deduce that
    \begin{align}
            &\Limsup_{k\to\infty} \nb F_k(\x_k)= \Limsup_{k\to\infty} \nb \rbra*{\mor{f}{\mu_k}\circ\frakS-\mor{g}{\mu_k}\circ\frakS}(\x_k)\\
            &\subset\ \Limsup_{k\to\infty}\nb (\mor{f}{\mu_k}\circ\frakS)(\x_k) + \Limsup_{k\to\infty} \rbra*{-\nb (\mor{g}{\mu_k}\circ\frakS)(\x_k)}\\
            &= \Limsup_{k\to\infty}\nb (\mor{f}{\mu_k}\circ\frakS)(\x_k) - \Limsup_{k\to\infty}\nb (\mor{g}{\mu_k}\circ\frakS)(\x_k).
    \end{align}
    By combining \cref{lem:gradient sub-consistency}, we get the desired inclusion:
    \begin{align}
        &\Limsup_{k\to\infty}\nb (\mor{f}{\mu_k}\circ\frakS)(\x_k) - \Limsup_{k\to\infty}\nb (\mor{g}{\mu_k}\circ\frakS)(\x_k)\\
        \subset\ &\del_\mathrm{L}(f \circ \frakS)(\bar{\x}) - \del_\mathrm{L}(g \circ \frakS)(\bar{\x}).
    \end{align}

    (b)
    The claim can be derived as follows:
    \begin{align}
        &\liminf_{k\to\infty}\norm{\nb F_k(\x_k)} = \liminf_{k\to\infty}\dist\rbra*{\bm{0},\nb F_k(\x_k)}\\
        &= \dist\rbra*{\bm{0},\Limsup_{k\to\infty} \nb F_k(\x_k)}\\
        &\ge \dist\rbra[\big]{\bm{0},\del_\mathrm{L}(f \circ \frakS)(\bar{\x}) - \del_\mathrm{L}(g \circ \frakS)(\bar{\x})},
    \end{align}
    where we used \cite[Exe. 4.8]{rockafellar2009variational} in the second equality and \cref{thm:dc gradient sub-consistency} (a) in the inequality, respectively.
\end{proof}

\section{Proof of \cref{pro:sufficient condition for descent assumption}}
\noindent We start with (b) and (c), as their proof are relatively simple.
\begin{proof}[Proof of  \cref{pro:sufficient condition for descent assumption} (b) and (c)]
    (b) By applying \cite[Prop. 4.2]{kume2024variableLong}, for any $\mu\in (0,(2\eta)^{-1}]$, 
    we can show that $\nb(\mor{f}{\mu}\circ\frakS)$ and $\nb(\mor{g}{\mu}\circ\frakS)$ are Lipschitz continuous, where their Lipschitz constants are $L_{\D\frakS}L_f + L_{\frakS}^2\mu^{-1}$ and $L_{\D\frakS}L_g + L_{\frakS}^2\mu^{-1}$, respectively. 
    Then, $\nabla F^{\abra{\mu}}$ is $\kappa_\mu$-Lipschitz continuous with $\kappa_\mu:=L_{\D \frakS}(L_f + L_g)+2L_{\frakS}^2 \mu^{-1}$.
    Thus, \cref{eq:descent assumption} immediately follows from the descent lemma (see, e.g., \cite[Lem. 5.7]{beck2017first}).

    (c) For any $\mu\in (0,(2\eta)^{-1}]$ and $\x\in\R^d$, we can deduce from the definition of the Moreau envelope that
    \begin{align}
    (\mor{\widehat{f}}{\mu}\circ \widehat{\frakS})(\x)
    &= \min_{\z\in\R^n, t\in\R} \cbra*{\eqsize{0.95}{f(\z)+t  + \frac{1}{2\mu}} \norm*{
    \begin{bmatrix}
        \frakS(\x) \\ h(\x)
    \end{bmatrix}-
    \begin{bmatrix}
        \z \\ t
    \end{bmatrix}}^2} \\
    &= (\mor{f}{\mu}\circ \frakS)(\x) + \min_{t\in\R}\cbra*{t+\frac{1}{2\mu}\abs*{t-h(\x)}^2} \\
    &= (\mor{f}{\mu}\circ \frakS)(\x) + h(\x) - \frac{\mu}{2},\\
    (\mor{\widehat{g}}{\mu}\circ \widehat{\frakS})(\x) &= (\mor{g}{\mu}\circ \frakS)(\x).
    \end{align}
    Then, we get $\widehat{F}^{\abra{\mu}}(\x) = h(\x)-\frac{\mu}{2} + F^{\abra{\mu}}(\x)$.
    On the other hand, the descent lemma for $h$ implies that, for all $\x,\y\in\R^d$,
    \begin{equation}\label{eq:descent lemma for h}
        h(\y)-\frac{\mu}{2}\le h(\x)-\frac{\mu}{2} + \abra{\nb h(\x), \y-\x} + \frac{L_{\nb h}}{2} \norm{\y-\x}^2
    \end{equation}
    holds.
    Therefore, by summing \cref{eq:descent assumption} and \cref{eq:descent lemma for h},
    we obtain
    \begin{equation}
       \widehat{F}^{\abra{\mu}}(\y)\le \widehat{F}^{\abra{\mu}}(\x) + \abra{\nb \widehat{F}^{\abra{\mu}}, \y-\x} + \frac{\widehat{\kappa}_\mu}{2} \norm{\y-\x}^2,
    \end{equation}
    where $\widehat{\kappa}_\mu:=\kappa_\mu + L_{\nb h}$. 
\end{proof}

We use the following fact and lemma to show \cref{pro:sufficient condition for descent assumption} (a).
\setcounter{dfn}{0}
\begin{fac}[Weak convexity of composite functions {\cite[Lem. 4.2]{drusvyatskiy2019efficiency}}]\label{eq:weak convexity of composite function}
Let a function $\map{\psi}{\R^n}{\R}$ be convex and $L_\psi$-Lipschitz continuous, and suppose that a mapping $\map{\mathcal{F}}{\R^d}{\R^n}$ is differentiable and $\D\mathcal{F}$ is $L_{\D\mathcal{F}}$-Lipschitz continuous.
Then, $\psi\circ\mathcal{F}$ is $L_\psi L_{\D\mathcal{F}}$-weakly convex.
\end{fac}
\begin{lemma}\label{lem:Lipschitz differentiability of composite function}
Let $\mathcal{X},\mathcal{Z}$ be Euclidean spaces.
Consider a closed set $\mathcal{D}\subset \mathcal{Z}$ and its complement $\mathcal{E}:=\mathcal{Z}\setminus \mathcal{D}$.
For a function $\map{J}{\mathcal{Z}}{\R}$ and a mapping $\map{\mathcal{F}}{\mathcal{X}}{\mathcal{Z}}$, assume that
\begin{enumerate}
    \renewcommand{\labelenumi}{(\roman{enumi})}
    \item $J$ is differentiable and $\nb J$ is $L_{\nb J}$-Lipschitz continuous on $\mathcal{Z}$.
    \item $J$ is $L_J$-Lipschitz continuous on $\mathcal{Z}$,
    and thus, $\norm{\nb J(\cdot)}$ is bounded above by $L_J$ on $\mathcal{Z}$ \cite[Thm. 9.7]{rockafellar2009variational}, i.e., 
    \begin{equation}
        (\forall \z\in\mathcal{Z})\quad\norm{\nb J(\z)} \le L_J.
    \end{equation}
    \item $J$ is affine on $\mathcal{D}$, i.e.,
    \begin{equation}
        (\exists \bm{u} \in \mathcal{Z}, \exists \tau\in\mathbb{R}, \forall \bm{z} \in\mathcal{Z})\quad J(\z) = \abra*{\bm{u},\z} + \tau.
    \end{equation}
    \item $\mathcal{F}$ is differentiable and $\D\mathcal{F}$ is $L_{\D\mathcal{F}}$-Lipschitz continuous on $\mathcal{X}$.
    \item $\norm{\D\mathcal{F}(\cdot)}_\mathrm{op}$ is bounded above by $M>0$ on $\mathcal{F}^{-1}(\mathcal{E}):=\Set*{\bm{v}\in\mathcal{X}}{\mathcal{F}(\bm{v})\in\mathcal{E}}$.%
\end{enumerate}
Then, $\nb (J\circ\mathcal{F})$ is $L_{\nb (J\circ\mathcal{F})}$-Lipschitz continuous on $\mathcal{X}$ with $L_{\nb (J\circ\mathcal{F})}:=M^2L_{\nb J}+L_JL_{\D\mathcal{F}}$.
\end{lemma}
\begin{proof}
By using the chain rule for $\nb(J\circ\mathcal{F})$, it holds for all $\x,\y\in\R^d$ that 
\begin{align}
    &\norm*{\nb(J\circ\mathcal{F})(\x)-\nb(J\circ\mathcal{F})(\y)}\\
    &= \norm[\big]{(\D\mathcal{F}(\x))^*[\nb J(\mathcal{F}(\x))]- (\D\mathcal{F}(\y))^*[\nb J(\mathcal{F}(\y))]}\\
    &\begin{multlined}[\linewidth]
        \le\norm[\big]{(\D\mathcal{F}(\x))^*[\nb J(\mathcal{F}(\x))-\nb J(\mathcal{F}(\y))]}\\
        +\norm[\big]{(\D\mathcal{F}(\x)-\D\mathcal{F}(\y))^*[\nb J(\mathcal{F}(\y))]}
    \end{multlined}\\
    &\begin{multlined}[\linewidth]
        \le \norm{\D\mathcal{F}(\x)}_{\mathrm{op}}\norm{\nb J(\mathcal{F}(\x))-\nb J(\mathcal{F}(\y))}\\
        + \norm{\D\mathcal{F}(\x)-\D\mathcal{F}(\y)}_{\mathrm{op}}\norm{\nb J(\mathcal{F}(\y))}
    \end{multlined}\\
    &\eqsize{0.95}{\le \norm{\D\mathcal{F}(\x)}_{\mathrm{op}}\norm{\nb J(\mathcal{F}(\x))-\nb J(\mathcal{F}(\y))} + L_JL_{\D\mathcal{F}}\norm{\x-\y},\label{eq:triangle inequality}}
\end{align}
where the last inequality follows from the assumptions (ii) and (iv).
(Note: $X^*$ stands for the \textit{adjoint operator} of the linear operator $X$.
When $X$ is regarded as a matrix, $X^*$ coincides with its transpose $X^T$.)

Here, we consider two cases according to whether the open line segment $\ell(\x,\y):=\Set{t\x+(1-t)\y}{0<t<1}$ intersects $\mathcal{F}^{-1}(\mathcal{D})$  or not.

\underline{\textbf{Case 1:}} $\ell(\x,\y)\cap \mathcal{F}^{-1}(\mathcal{D}) = \emptyset$\\
Since $\norm*{\nb(J\circ\mathcal{F})(\x)-\nb(J\circ\mathcal{F})(\y)} \le L_{\nb (J\circ\mathcal{F})}\norm{\x-\y}$ obviously holds in the trivial case $\x=\y$, we assume $\x\neq \y$.
From $\ell(\x,\y)\cap \mathcal{F}^{-1}(\mathcal{D}) = \emptyset$, we have
\begin{align}\label{eq:inclusion of line segemnt}
    \ell(\x,\y)&\subset \mathcal{X} \setminus \mathcal{F}^{-1}(\mathcal{D})=\mathcal{F}^{-1}\rbra*{\mathcal{Z}\setminus \mathcal{D}}=\mathcal{F}^{-1}(\mathcal{E}).
\end{align}
Then, the assumption (v) implies that
\begin{equation}\label{eq:boundedness of DF on l}
    \norm{\D\mathcal{F}(\bm{v})}_{\mathrm{op}} \le M \quad(\bm{v}\in \ell(\x,\y)).
\end{equation}
Moreover, from the continuity of $\norm{\D \mathcal{F}(\cdot)}_\mathrm{op}$,
we have
\begin{equation}\label{eq:upper bound of DF(x)}
    \norm{\D\mathcal{F}(\x)}_\mathrm{op}=\lim_{t\nearrow 1}\norm{\D\mathcal{F}(t\x+(1-t)\y)}_\mathrm{op} \leq \lim_{t\nearrow 1} M =M,
\end{equation}
and we also get $\norm{\D\mathcal{F}(\y)}_\mathrm{op}\le M$ in the same manner.
Thus, the mean value inequality (see, e.g., \cite[p.155]{Dieudonné1960Foundations}) yields that
\begin{align}\label{eq:Lipschitz continuity of F on l}
    &\norm{\mathcal{F}(\x)-\mathcal{F}(\y)} \leq \sup_{0\le t \le 1}\norm{\D\mathcal{F}(t\x+(1-t)\y)}_\mathrm{op}\norm{\x-\y}\\
    &= \sup_{\bm{v}\in \ell(\x,\y)\cup\{\x,\y\}}\norm{\D\mathcal{F}(\bm{v})}_\mathrm{op}\norm{\x-\y} \leq M\norm{\x-\y}.
\end{align}
Therefore, by combining the assumption (i) and \cref{eq:upper bound of DF(x)}, we have
\begin{align}
    &\norm{\D\mathcal{F}(\x)}_{\mathrm{op}}\norm{\nb J(\mathcal{F}(\x))-\nb J(\mathcal{F}(\y))}\\
    &\le\norm{\D\mathcal{F}(\x)}_{\mathrm{op}}M L_{\nb J}\norm{\x-\y} \le M^2 L_{\nb J}\norm{\x-\y}.\label{eq:Lipschitz inequality for first term}
\end{align}
From \cref{eq:triangle inequality} and \cref{eq:Lipschitz inequality for first term}, we obtain
\begin{equation}
    \norm*{\nb(J\circ\mathcal{F})(\x)-\nb(J\circ\mathcal{F})(\y)} \le L_{\nb (J\circ\mathcal{F})}\norm{\x-\y}.
\end{equation}

\begin{figure}[t]
\centering

\begin{tikzpicture}

    \def\mcurve{(-2.5,-0.5)
    .. controls (-1.7,1.3) and (-1.1,1.3) .. (-0.5,0.5)
    .. controls (-0.1,-0.2) and (0.4,-0.2) .. (0.75,0.5)
    .. controls (1.2,1.3) and (1.8,1.3) .. (2.5,-0.5)
    }

    \path[name path=curve] \mcurve;

    \coordinate (x) at (-3,0.5);
    \coordinate (y) at (3,0.7);
    
    \path[name path=line] (x) -- (y);

    \path[name intersections={of=curve and line, name=i}];

    \fill[blue!20] \mcurve -- (3.5,-0.5)-- (3.5,1.5) -- (-3.5,1.5) -- (-3.5, -0.5) -- cycle;
    \fill[red!20]  \mcurve -- (2.5,-0.5) -- (-2.5,-0.5) -- cycle;

    \draw \mcurve;
    \fill (x) circle (1.5pt);
    \node[above] at (x) {$\x$};
    \fill (y) circle (1.5pt);
    \node[above] at (y) {$\y$};
    \draw[thick] (x) -- (y);
    
    \fill[red] (i-1) circle (1.5pt);
    \node[above] at (i-1) {$\btilde{x}$};
    \fill[red] (i-4) circle (1.5pt);
    \node[above] at (i-4) {$\btilde{y}$};

    \node at (0.1,1.1) {$\mathcal{F}^{-1}(\mathcal{E})$};
    \node at (-1,-0.2) {$\mathcal{F}^{-1}(\mathcal{D})$};
\end{tikzpicture}

\caption{Illustration of the positions of $\btilde{x}$ and $\btilde{y}$}
\label{fig:Illustration of the positions of x_tilde and y_tilde}
\end{figure}

\underline{\textbf{Case 2:}} $\ell(\x,\y)\cap \mathcal{F}^{-1}(\mathcal{D}) \ne \emptyset$\\
We define points $\btilde{x},\btilde{y}$ on the closed line segment $\mathrm{cl}(\ell(\x,\y))=\ell(\x,\y)\cup\{\x,\y\}$ as
\begin{align}
    &\btilde{x}:=\argmin{\bm{v}\in \mathfrak{I}}\norm{\x-\bm{v}},\ \btilde{y}:=\argmin{\bm{v}\in \mathfrak{I}}\norm{\y-\bm{v}}\\
    &\text{with}\quad \mathfrak{I}:=\mathrm{cl}\rbra*{\ell(\x,\y)}\cap \mathcal{F}^{-1}(\mathcal{D})
\end{align}
(see \cref{fig:Illustration of the positions of x_tilde and y_tilde} for an intuitive illustration). Note that $\btilde{x}$ and $\btilde{y}$ are well-defined due to the compactness of $\mathfrak{I}$.
For $\btilde{x},\btilde{y}$, there exist $t_{\tilde{x}}, t_{\tilde{y}} \in [0,1]$ such that $t_{\tilde{y}} \leq t_{\tilde{x}}$ and 
$\btilde{x}=t_{\tilde{x}} \x+ (1-t_{\tilde{x}})\y, \btilde{y}=t_{\tilde{y}} \x+ (1-t_{\tilde{y}})\y$.
From $\x-\btilde{x} =  (1-t_{\tilde{x}})(\x-\y),\ \btilde{x}-\btilde{y}=(t_{\tilde{x}} - t_{\tilde{y}})(\x-\y),\ \btilde{y}-\y=t_{\tilde{y}}(\x-\y)$, we have
\begin{equation}\label{eq:decomposition of line segment}
    \norm{\x-\y}=\norm{\x-\btilde{x}}+\norm{\btilde{x}-\btilde{y}}+\norm{\btilde{y}-\y}.
\end{equation}
In what follows, we consider the three line segments $\ell(\x,\btilde{x}),\ \ell(\btilde{y},\y)$, and $\ell(\btilde{x},\btilde{y})$.

Because $\ell(\x,\btilde{x})\cap \mathcal{F}^{-1}(\mathcal{D})=\emptyset$ and $\ell(\btilde{y},\y)\cap \mathcal{F}^{-1}(\mathcal{D})=\emptyset$ follow from the definition of $\btilde{x}$ and $\btilde{y}$, the same argument as the Case 1 yields that 
\begin{align}
    &\norm*{\nb(J\circ\mathcal{F})(\x)-\nb(J\circ\mathcal{F})(\btilde{x})} \le L_{\nb (J\circ\mathcal{F})}\norm{\x-\btilde{x}}, \label{eq:Lipschitz continuity on x and xtilde} \\
    &\norm*{\nb(J\circ\mathcal{F})(\y)-\nb(J\circ\mathcal{F})(\btilde{y})} \le L_{\nb (J\circ\mathcal{F})}\norm{\y-\btilde{y}}.\label{eq:Lipschitz continuity on y and ytilde}
\end{align}
For $\ell(\btilde{x},\btilde{y})$, the assumption (iii) together with $\btilde{x},\btilde{y}\in \mathcal{F}^{-1}(\mathcal{D})$ implies that
\begin{align}
    \norm{\nb J(\mathcal{F}(\btilde{x}))-\nb J(\mathcal{F}(\btilde{y}))} = \norm{\bm{u} - \bm{u}}=0,
\end{align}
and then, \cref{eq:triangle inequality} with the substitution $(\x,\y)=(\btilde{x},\btilde{y})$ yields
\begin{equation}\label{eq:Lipschitz continuity on xtilde and ytilde}
    \norm*{\nb(J\circ\mathcal{F})(\btilde{x})-\nb(J\circ\mathcal{F})(\btilde{y})} \le L_JL_{\D\mathcal{F}}\norm{\btilde{x}-\btilde{y}}.
\end{equation}
By using \cref{eq:Lipschitz continuity on x and xtilde,eq:Lipschitz continuity on xtilde and ytilde,eq:Lipschitz continuity on y and ytilde}, we obtain
\begin{align}
    &\norm*{\nb(J\circ\mathcal{F})(\x)-\nb(J\circ\mathcal{F})(\y)}\\
    &
    \begin{multlined}[t][\linewidth]
        \le\norm*{\nb(J\circ\mathcal{F})(\x)-\nb(J\circ\mathcal{F})(\btilde{x})}\\
        +\norm*{\nb(J\circ\mathcal{F})(\btilde{x})-\nb(J\circ\mathcal{F})(\btilde{y})}\\
        +\norm*{\nb(J\circ\mathcal{F})(\btilde{y})-\nb(J\circ\mathcal{F})(\y)}
    \end{multlined}\\
    &\le\eqsize{0.97}{ L_{\nb (J\circ\mathcal{F})}\norm*{\x-\btilde{x}}+L_JL_{\D\mathcal{F}}\norm{\btilde{x}-\btilde{y}} + L_{\nb (J\circ\mathcal{F})}\norm*{\btilde{y}-\y}}\\
    &\le L_{\nb (J\circ\mathcal{F})}(\norm*{\x-\btilde{x}}+\norm{\btilde{x}-\btilde{y}}+\norm*{\btilde{y}-\y})\\
    &=L_{\nb (J\circ\mathcal{F})}\norm{\x-\y}. \qquad(\because \cref{eq:decomposition of line segment})
\end{align}

\end{proof}
\begin{proof}[Proof of \cref{pro:sufficient condition for descent assumption} (a)]
The proof is completed by showing that there exist $\varpi_1', \varpi_1'',\varpi_2,\in\R_{++}$ such that 

\begin{align}
    -(\mor{g}{\mu}&\circ \frakS_{\text{RPR}})(\y) \le -(\mor{g}{\mu}\circ \frakS_{\text{RPR}})(\x) \\
    &- \abra{\nb (\mor{g}{\mu}\circ \frakS_{\text{RPR}})(\x), \y-\x} + \frac{\varpi_1'}{2}\norm{\y-\x}^2,&\label{eq:descent lemma for goS} \\
    (\mor{f}{\mu}&\circ \frakS_{\text{RPR}})(\y) \le (\mor{f}{\mu}\circ \frakS_{\text{RPR}})(\x)\\
    + \abra{&\nb (\mor{f}{\mu}\circ \frakS_{\text{RPR}})(\x), \y-\x} \eqsize{0.95}{+ \frac{\varpi_1''+\varpi_2\mu^{-1}}{2}\norm{\y-\x}^2,} \label{eq:descent lemma for foS}
\end{align}
because \cref{eq:descent assumption} is obtained by summing \cref{eq:descent lemma for goS} and \cref{eq:descent lemma for foS}, and by setting $\varpi_1:=\varpi_1'+\varpi_1''$.
In the following, we show \cref{eq:descent lemma for goS} and \cref{eq:descent lemma for foS} separately.

For \cref{eq:descent lemma for goS}, it is enough to prove the following inequality:
\begin{align}
    (\mor{g}{\mu}\circ \frakS_{\text{RPR}})(\y) + \frac{\varpi_1'}{2}\norm{\y}^2 \ge (\mor{g}{\mu}\circ \frakS_{\text{RPR}})(\x) + \frac{\varpi_1'}{2}\norm{\x}^2& \\
    + \abra*{\nb (\mor{g}{\mu}\circ \frakS_{\text{RPR}})(\x)+\varpi_1'\x,\y-\x}&.
\end{align}
This is equivalent to the convexity of
$(\mor{g}{\mu}\circ \frakS_{\text{RPR}})+\frac{\varpi_1'}{2}\norm{\cdot}^{2}$, that is, $\varpi_1'$-weak convexity of $\mor{g}{\mu}\circ\frakS_{\text{RPR}}$.
To show this weak convexity, we can use \cref{eq:weak convexity of composite function} with $\psi:=\mor{g}{\mu}$  and $\mathcal{F}:=\frakS_{\text{RPR}}$.
Indeed, the assumptions of \cref{eq:weak convexity of composite function} are easily checked as follows:
(i) $\mor{g}{\mu}$ is $L_g$-Lipschitz continuous because $g$ is $L_g$-Lipschitz continuous (see \cite[Lem. 3.3]{bohm2021variable});
(ii) since every choice of $g$ in \cref{tab:options of phi} is convex, $\mor{g}{\mu}$ is also convex \cite[Thm. 6.55]{beck2017first};
(iii) the derivative $\D\frakS_{\text{RPR}}:\x\mapsto [2\abra{\ba_1,\x }\ba_1,2\abra{\ba_2,\x }\ba_2,...,2\abra{\ba_n,\x }\ba_n]^T$ is $(2\sqrt{\sum_{i=1}^n \norm{\ba_i}^4})$-Lipschitz continuous
because we have, for all $\x,\y\in\R^d$,
\begin{align}
    &\norm{\D\frakS_{\text{RPR}}(\x) - \D\frakS_{\text{RPR}}(\y)}_{\text{op}}\\
    &= \sup_{\|\bm{v}\| \le 1} \norm{\rbra*{\D\frakS_{\text{RPR}}(\x) - \D\frakS_{\text{RPR}}(\y)}[\bm{v}]}\\
    &= \sup_{\|\bm{v}\| \le 1} \sqrt{\sum_{i=1}^n \rbra*{2\abra{\ba_i,\x-\y} \abra{\ba_i,\bm{v}}}^2}\\
    &\le \sup_{\|\bm{v}\| \le 1} \sqrt{\sum_{i=1}^n \rbra*{2\norm{\ba_i}\norm{\x-\y} \norm{\ba_i}\norm{\bm{v}}}^2}\\
    &\le 2\sqrt{\sum_{i=1}^n\norm{\ba_i}^4} \norm{\x-\y}. \label{eq:Lipschitz continuity of DS_RPR}
\end{align}
As a result, the weak convexity of $\mor{g}{\mu}\circ \frakS_{\text{RPR}}$ is proven, and therefore, \cref{eq:descent lemma for goS} holds with $\varpi_1':=2L_g\sqrt{\sum_{i=1}^n \norm{\ba_i}^4}$.

To prove \cref{eq:descent lemma for foS}, we show that $\nb (\mor{f}{\mu}\circ \frakS_{\text{RPR}})$ is ($\varpi_1''+\varpi_2\mu^{-1}$)-Lipschitz continuous, which is a sufficient condition of \cref{eq:descent lemma for foS} as stated in the descent lemma (see, e.g., \cite[Lem. 5.7]{beck2017first}).
We note that all $f$ in \cref{tab:options of phi} have the form $f(\z)=\lambda\sum_{i=1}^n \abs{[\z]_i}\ (\z\in\R^n)$ with some $\lambda\in\R_{++}$,
and thereby, their Moreau envelopes can be derived as $\mor{f}{\mu}(\z)=\sum_{i=1}^n\mor{(\lambda\abs{\cdot})}{\mu}([\z]_i) = \sum_{i=1}^n\widehat{r}_{\lambda,\mu}([\z]_i)$ \cite[Lem. 6.57,\ Exm. 6.59]{beck2017first} (see \cref{tab:options of phi} for the definition of $\widehat{r}_{\lambda,\mu}$).
In order to invoke \cref{lem:Lipschitz differentiability of composite function} with $J:=\widehat{r}_{\lambda,\mu}$, $\mathcal{F}:=\frakS_{\text{RPR},i}:=\abra{\ba_i,\cdot\ }^2 - [\bb]_i\ (i\in\{1,2,\ldots,n\})$, and $(\mathcal{D},\mathcal{E}):=([\mu\lambda,\infty),(-\infty,\mu\lambda))$,
we check that the assumptions (i)-(v) in \cref{lem:Lipschitz differentiability of composite function} hold.
\begin{enumerate}
    \renewcommand{\labelenumi}{(\roman{enumi})}
    \item $\nb \widehat{r}_{\lambda,\mu}=\nb \mor{(\lambda\abs{\cdot})}{\mu}$ is $\mu^{-1}$-Lipschitz continuous from \cite[Thm. 6.60]{beck2017first}.
    \item From the definition of $\widehat{r}_{\lambda,\mu}$, it is obviously $\lambda$-Lipschitz continuous.
    \item $\widehat{r}_{\lambda,\mu}(t)=t-\mu\lambda^2/2$ for $t \ge \mu\lambda$, i.e., it is affine on $[\mu\lambda,\infty)$.
    \item \scalebox{0.97}[1]{$\D \frakS_{\text{RPR},i}=2\abra{\ba_i,\cdot\, }\ba_i^T$ is $(2\norm{\ba_i}^2)$-Lipschitz continuous.}
    \item It holds for any $\x \in \frakS_{\text{RPR},i}^{-1}\rbra[\big]{(-\infty,\mu\lambda)}$ that\\
        $\norm{\D\frakS_{\text{RPR},i}(\x)}_\mathrm{op}=\norm{2\abra{\ba_i,\x}\ba_i^T}_{\text{op}}=2\norm{\ba_i}\abs{\abra{\ba_i,\x}}=2\norm{\ba_i} \sqrt{\frakS_{\text{RPR},i}(\x) + [\bb]_i} < 2\norm{\ba_i}\sqrt{\mu\lambda+\abs{[\bb]_i}}.$
\end{enumerate}
From above, \cref{lem:Lipschitz differentiability of composite function} yields that $\nb (\widehat{r}_{\lambda,\mu}\circ\frakS_{\text{RPR},i})$ is Lipschitz continuous with a Lipschitz constant
$4\norm{\ba_i}^2(\mu\lambda+\abs{[\bb]_i})\mu^{-1}+  \lambda (2\norm{\ba_i}^2) =  6\norm{\ba_i}^2\lambda + 4\norm{\ba_i}^2\abs{[\bb]_i} \mu^{-1}$.
Then, for all $\x,\y\in\R^d$, we have
\begin{align}
    &\norm{\nabla (\mor{f}{\mu}\circ\frakS_{\text{RPR}})(\x) - \nabla (\mor{f}{\mu}\circ\frakS_{\text{RPR}})(\y)}^{2}\\
    &\leq \sum_{i=1}^n (6\norm{\ba_i}^2\lambda + 4\norm{\ba_i}^2\abs{[\bb]_i} \mu^{-1})^2 ([\x]_i - [\y]_i)^2 \\
    &\leq \max_{1\le i\le n} (6\norm{\ba_i}^2\lambda + 4\norm{\ba_i}^2\abs{[\bb]_i} \mu^{-1})^2 \sum_{j=1}^n ([\x]_j - [\y]_j)^2 \\
    &= \rbra[\Big]{\max_{1\le i\le n} (6\norm{\ba_i}^2\lambda + 4\norm{\ba_i}^2\abs{[\bb]_i} \mu^{-1})}^2 \sum_{j=1}^n ([\x]_j - [\y]_j)^2 \\
    &\le \rbra[\Big]{6\lambda\max_{1\le i\le n}\norm{\ba_i}^2 + 4\mu^{-1}\max_{1\le i\le n}\norm{\ba_i}^2\abs{[\bb]_i}}^2 \norm{\x-\y}^{2}.
\end{align}
Thus, $\nb (\mor{f}{\mu}\circ\frakS_{\text{RPR}})$ is $(\varpi_1''+\varpi_2\mu^{-1})$-Lipschitz continuous,
where $\eqsize{0.99}{\varpi_1'' := 6\lambda\max_{1\le i\le n} \norm{\ba_i}^2,\ \varpi_2:=4 \max_{1\le i\le n}\norm{\ba_i}^2\abs{[\bb]_i}}$.

\end{proof}

\section{Proof of \cref{exa:initial guess satisfying assumption}}
\noindent We show that the choices of initial guesses $(\gamma_{\text{init},k})_{k=1}^\infty$ in \cref{exa:initial guess satisfying assumption} satisfy \cref{asm:initial guess of backtracking}.
\begin{proof}[Proof of \cref{exa:initial guess satisfying assumption}]
    (a) The inequality \cref{eq:assumption on initial guess of backtracking} is obviously satisfied with $\delta:=2(1-c)$.

    (b) Since $\seq{\mu}{k}$ is non-increasing from \cref{eq:conditions of mu}(iii),
    $\kappa_{\mu_k}:=\varpi_1 + \varpi_2/\mu_k$ is non-decreasing, and then, we have $\kappa_{\mu_k} \ge \kappa_{\mu_1}\ (k\in\N)$.
    Hence, the inequality $\gamma_{\text{init},k}:=\gamma_{\text{init}}\ge \delta \kappa_{\mu_k}^{-1} $ holds with $\delta := \gamma_{\text{init}}\kappa_{\mu_1}$.

    (c) By the induction, we show $\gamma_{\text{init},k} \geq \delta \kappa_{\mu_{k}}^{-1}\ (\forall k\in \mathbb{N})$ in \cref{eq:assumption on initial guess of backtracking} with $\delta := \delta_0:= \min \{\gamma_0\kappa_{\mu_1},2(1-c)\rho\}$. 
    For $\gamma_{\text{init},1} := \gamma_0$, we have $\gamma_{\text{init}, 1} = (\gamma_{0}\kappa_{\mu_{1}})\kappa_{\mu_{1}}^{-1}\ge \delta_{0}\kappa_{\mu_{1}}^{-1}$.
    On the other hand, by using \cref{lem:finite termination of backtracking} (a) and the induction hypothesis $\gamma_{\text{init},k} \ge \delta_0 \kappa_{\mu_k}^{-1}$, we obtain
    $\gamma_{\text{init},k+1} := \gamma_k \ge \min\{\gamma_{\text{init}, k}, 2(1-c)\kappa_{\mu_{k}}^{-1}\rho\} \ge \min \cbra*{\delta_0,2(1-c)\rho}\kappa_{\mu_k}^{-1} = \delta_0 \kappa_{\mu_k}^{-1} \ge \delta_0 \kappa_{\mu_{k+1}}^{-1}$,
    where the last equality follows from $2(1-c)\rho \ge \delta_0$, and the last inequality holds since  $(\kappa_{\mu_k})_{k=1}^\infty$ is non-decreasing.
    This completes the inductive proof.
\end{proof}

\section{Proof of \cref{thm:convergence theorem}}
\noindent
We make use of the next fact to prove \cref{thm:convergence theorem}.
\begin{fac}[Properties of Moreau envelope of Lipschitz continuous function]\label{fac:properties of Moreau envelope of Lipschitzian}
    Let a function $\map{\psi}{\R^n}{\R}$ be $\eta_\psi$-weakly convex and $L_\psi$-Lipschitz continuous.
    Then, for any $\z\in\R^n$ and $\mu_{\rmone},\mu_{\rmtwo}\in(0,\eta_\psi)$ such that $\mu_{\rmtwo}<\mu_{\rmone}$, we have
    \begin{equation}\label{eq:gap between two moreau envelopes}
        \mor{\psi}{\mu_{\rmone}}(\z) \le \mor{\psi}{\mu_{\rmtwo}}(\z) \le \mor{\psi}{\mu_{\rmone}}(\z) + (\mu_\rmone - \mu_\rmtwo) L_\psi^2
    \end{equation}
    (see \cite[Eq. (21)]{kume2024variableLong} and the subsequent discussion in \cite{kume2024variableLong}).
    Moreover, by letting $\mu_\rmtwo\searrow 0$, we obtain the following (see \cref{Properties of Moreau envelope} (a)):
    \begin{equation}\label{eq:gap between moreau and original}
        \mor{\psi}{\mu_{\rmone}}(\z) \le \psi(\z) \le \mor{\psi}{\mu_{\rmone}}(\z) + \mu_\rmone L_\psi^2. 
    \end{equation}
\end{fac}
\begin{proof}[Proof of \cref{thm:convergence theorem}]
    This proof is inspired by those for \cite[Thm. 4.1]{bohm2021variable} and \cite[Thm. 4.8]{kume2024variableLong}.

    (a) Since the stepsize $\gamma_k$ output by \cref{backtracking} satisfies the Armijo condition \cref{eq:armijo condition}, we have
    \begin{equation} \label{eq:inequality about Fk xk+1 and Fk xk}
        F_k(\x_{k+1}) \le F_k(\x_k)-c \gamma_k \norm{\nb F_k(\x_k)}^2 \backin{k}{\N}.
    \end{equation}
    On the other hand, by virtue of \cref{eq:gap between two moreau envelopes} and $\mu_{k+1}\leq \mu_{k}$ (see \cref{eq:conditions of mu}(iii)), the following inequality holds for any $\x\in\R^d$:
    \begin{align} 
        &F_k(\x) = \mor{f}{\mu_k}(\frakS(\x)) -  \mor{g}{\mu_k}(\frakS(\x))\\
        &\ge \mor{f}{\mu_k}(\frakS(\x)) -  \mor{g}{\mu_{k+1}}(\frakS(\x))\\
        &\ge \mor{f}{\mu_{k+1}}(\frakS(\x))  - (\mu_k - \mu_{k+1}) L_f^2-  \mor{g}{\mu_{k+1}}(\frakS(\x))\\
        &= F_{k+1}(\x) - (\mu_k - \mu_{k+1}) L_f^2. \label{eq:inequality about Fk and Fk+1}
    \end{align}
    By combining  \cref{eq:inequality about Fk and Fk+1} with $\x:=\x_{k+1}$ and \cref{eq:inequality about Fk xk+1 and Fk xk}, we obtain
    \begin{align}\label{eq:inequality about Fk+1 xk+1 and Fk xk}
        \scalebox{0.95}{$F_{k+1}(\x_{k+1}) \le F_k(\x_k) -c \gamma_k \norm{\nb F_k(\x_k)}^2 + (\mu_k - \mu_{k+1}) L_f^2,$}
    \end{align}
    and thus,
    \begin{align}\label{eq:inequality about Fk+1 xk+1 and Fk xk without squared gradient}
        F_{k+1}(\x_{k+1}) \le F_k(\x_k) + (\mu_k - \mu_{k+1}) L_f^2.
    \end{align}
    By summing up \cref{eq:inequality about Fk+1 xk+1 and Fk xk} from $k = \underline{k}$ to $\bar{k}$, we deduce that
    \begin{equation}
        \eqsize{0.93}{\sum_{k=\underline{k}}^{\bar{k}} c \gamma_k \norm{\nb F_k(\x_k)}^2 \le F_{\underline{k}}(\x_{\underline{k}}) - F_{\bar{k}+1}(\x_{\bar{k}+1}) + (\mu_{\underline{k}} - \mu_{\bar{k}+1}) L_f^2.}
    \end{equation}
    We use \cref{eq:inequality about Fk+1 xk+1 and Fk xk without squared gradient} repeatedly to get
    \begin{equation}
        \eqsize{0.93}{\sum_{k=\underline{k}}^{\bar{k}} c \gamma_k \norm{\nb F_k(\x_k)}^2 \le F_1(\x_1) - F_{\bar{k}+1}(\x_{\bar{k}+1}) + (\mu_1 - \mu_{\bar{k}+1}) L_f^2.}
    \end{equation}
    Here, we see from \cref{eq:gap between moreau and original} in \cref{fac:properties of Moreau envelope of Lipschitzian} that
    \begin{align}
        F_{\bar{k}+1}(\x_{\bar{k}+1}) &= \mor{f}{\mu_{\bar{k}+1}}(\frakS(\x_{\bar{k}+1})) - \mor{g}{\mu_{\bar{k}+1}}(\frakS(\x_{\bar{k}+1}))\\
        &\ge f(\frakS(\x_{\bar{k}+1}))-\mu_{\bar{k}+1}L^2_f - g(\frakS(\x_{\bar{k}+1}))\\
        &= F(\x_{\bar{k}+1})- \mu_{\bar{k}+1}L^2_f,
    \end{align}
    by a similar discussion in \cref{eq:inequality about Fk and Fk+1}. Then, we have
    \begin{align}
        \sum_{k=\underline{k}}^{\bar{k}} c \gamma_k \norm{\nb F_k(\x_k)}^2 &\le F_1(\x_1) - F(\x_{\bar{k}+1}) + \mu_1 L_f^2\\
        &\le F_1(\x_1) - \inf_{\x\in\R^d}F(\x) + \mu_1 L_f^2< \infty, \label{eq:upper bound of gamma gradient sum}
    \end{align}
    where $\inf_{\x\in\R^d} F(\x) > -\infty$ holds by \cref{problem} (c).
    Now, from \cref{asm:initial guess of backtracking} and \cref{lem:finite termination of backtracking} (a), we can bound $\gamma_k$ from below, i.e., we have
    \begin{equation}
        \gamma_k \ge \bar{\delta}\kappa_{\mu_k}^{-1}:= \min \{\delta,2(1-c)\rho\}\kappa_{\mu_k}^{-1}.
    \end{equation}
    From this inequality, we obtain
    \begin{align}
        &\sum_{k=\underline{k}}^{\bar{k}} c \gamma_k \norm{\nb F_k(\x_k)}^2 \ge c\bar{\delta}\sum_{k=\underline{k}}^{\bar{k}}  \kappa^{-1}_{\mu_k}\norm{\nb F_k(\x_k)}^2\\
        &= c\bar{\delta}\sum_{k=\underline{k}}^{\bar{k}}  \frac{\mu_k}{\varpi_1\mu_k + \varpi_2}\norm{\nb F_k(\x_k)}^2\\
        &\ge c\bar{\delta}\sum_{k=\underline{k}}^{\bar{k}}  \frac{\mu_k}{\varpi_1(2\eta)^{-1} + \varpi_2}\norm{\nb F_k(\x_k)}^2\\
        &\ge \frac{2\eta c\bar{\delta}}{\varpi_1+2\eta\varpi_2}\min_{\underline{k}\le k \le \bar{k}}\norm{\nb F_k (\x_k)}^2 \sum_{k=\underline{k}}^{\bar{k}}\mu_k,
        \label{eq:lower bound of gamma gradient sum}
    \end{align}
    where we employed $\mu_k\le(2\eta)^{-1}$ in the second inequality.
    Hence, \cref{eq:lemma for convergence analysis} follows from \cref{eq:upper bound of gamma gradient sum} and \cref{eq:lower bound of gamma gradient sum} with
    \begin{equation}
        \eqsize{0.94}{C := \frac{\rbra{\varpi_1+2\eta\varpi_2}\rbra[\big]{F_1(\x_1) - \inf_{\x\in\R^d}F(\x) + \mu_1L_f^2}}{2\eta c\bar{\delta}}\in\R_{++}}.
    \end{equation}

    (b)
    From \cref{eq:conditions of mu} (ii), we can get the following by taking the limit as $\bar{k}\to\infty$ on the both sides of \cref{eq:lemma for convergence analysis}:
    \begin{equation}
        \inf_{\underline{k}\le k }\norm{\nb F_k(\x_k)} \le 0.
    \end{equation}
    Hence, \cref{eq:convergence theorem} is obtained by taking the limit as $\underline{k}\to\infty$.

    (c)
    From \cref{thm:convergence theorem} (b), we can construct $(\x_{m(l)})_{l=1}^\infty$ such that $\lim_{l\toinf}\norm{\nb F_{m(l)}(\x_{m(l)})} = 0$, e.g., in the same manner as in \cite[footnote 11]{kume2024variableLong}.
    \cref{thm:dc gradient sub-consistency} (b) leads to $\text{dist}\rbra[\Big]{\bm{0},\del_\mathrm{L}(f \circ \mathfrak{S})(\bar{\x})-\del_\mathrm{L}(g \circ \mathfrak{S})(\bar{\x})}=0$ for any cluster point $\bar{\x}$ of $(\x_{m(l)})_{l=1}^\infty$, which completes the proof.
\end{proof}